\documentclass{amsart}

\usepackage{mathabx}
\usepackage{graphicx}
\usepackage{amsfonts}
\usepackage{amsmath}
\usepackage{amscd}
\usepackage{amssymb}
\usepackage{verbatim}
\usepackage{hyperref}
\usepackage{marginnote}

\newenvironment{proof*}{\noindent\emph{Proof}}{$\square$\smallskip}

\newtheorem{theorem}{Theorem}[section]

\newtheorem{Definition}[theorem]{Definition}
\newtheorem{lemma}[theorem]{Lemma}
\newtheorem{Example}[theorem]{Example}
\newtheorem{corollary}[theorem]{Corollary}
\newtheorem{Remark}[theorem]{Remark}
\newtheorem{proposition}[theorem]{Proposition}

\newtheorem{claim}[theorem]{Claim}

\newtheorem{Exercise}[theorem]{Exercise}
\newtheorem{Exercises}[theorem]{Exercises}
\newtheorem{Notation}[theorem]{Notation}
\newtheorem{Convention}[theorem]{Convention}
\newtheorem{standing assumption}[theorem]{Standing Assumption}

\newenvironment{definition}{\begin{Definition}\normalfont}{\end{Definition}}
\newenvironment{example}{\begin{Example}\normalfont}{\end{Example}}
\newenvironment{remark}{\begin{Remark}\normalfont}{\end{Remark}}

\newenvironment{convention}{\begin{Convention}\normalfont}{\end{Convention}}

\newcommand{\id}{\ensuremath{\mathrm{id}}} 

\title[Finiteness properties]{Finiteness properties of locally defined groups}  
\author[D.~S.~Farley]{Daniel S. Farley}
\address{Department of Mathematics\\ Miami University\\ Oxford, OH 45056 U.S.A.}
\email{farleyds@muohio.edu}

\author[B.~Hughes]{Bruce Hughes}
\address{Department of Mathematics\\ Vanderbilt University\\ Nashville, TN 37240 U.S.A.}
\email{bruce.hughes@vanderbilt.edu}

\date{\today}

\begin{document}

\begin{abstract} 
Let $X$ be a set and let $S$ be an inverse semigroup of partial bijections of $X$. Thus, an element of $S$ is a bijection between two subsets of $X$, and the set $S$ is required to be closed under the operations of taking inverses and compositions of functions. We define  $\Gamma_{S}$ to be the set of self-bijections of $X$ in which each $\gamma \in \Gamma_{S}$ is expressible as a union of finitely many members of $S$. This set is a group with respect to composition.

The groups $\Gamma_{S}$ form a class containing numerous widely studied groups, such as Thompson's group $V$, the Nekrashevych-R\"{o}ver groups, Houghton's groups, and the Brin-Thompson groups $nV$, among many others.

We offer a unified construction of geometric models for $\Gamma_{S}$ and a general framework for studying the finiteness properties of these groups.
\end{abstract}

\subjclass[2010]{Primary 20F65, 20J05; Secondary 20M18}

\keywords{generalized Thompson groups,  inverse semigroups, finiteness properties} 

\maketitle

\setcounter{tocdepth}{2} \tableofcontents

\section{Introduction}

This paper is about the topological finiteness properties $F_{n}$. A group $G$ has \emph{type $F_{n}$} if it is the fundamental group of an aspherical CW-complex with finite $n$-skeleton. Thus, all groups are of type $F_{0}$, finitely generated groups are precisely the groups of type $F_{1}$, and finitely presented groups are precisely the groups of type $F_{2}$. Type $F_{n}$, for $n \geq 3$, is a topologically-defined refinement of the latter properties. A group has \emph{type $F_{\infty}$} if it has type $F_{n}$ for all $n \geq 0$. The book \cite{BookbyRoss} by Geoghegan discusses these properties in greater depth, and also provides a great deal of relevant background.    

The ``locally defined'' groups from the title are groups of bijections of a given set $X$, where the bijections in question have a suitable ``piecewise'' definition. Thompson's group $V$ is an illustrative example. There are many others, as we explain below.

Our basic approach to topological finiteness properties can be traced back to work of Brown and Geoghegan from the 1980's \cite{BrownGeoghegan}, who proved that Thompson's group $F$ has type $F_{\infty}$. In subsequent work \cite{Brown}, Brown showed that Thompson's groups $F$, $T$, and $V$ all have type $F_{\infty}$. The arguments were essentially Morse-theoretic in nature. Brown also established the $F_{\infty}$ property for several classes of generalized Thompson groups $F_{n,r}$, $T_{n,r}$, and $V_{n,r}$. Similar arguments from \cite{Brown} showed that Houghton's groups $H_{n}$ are of type $F_{n-1}$ but not $F_{n}$. We refer the reader to Cannon, Floyd, and Parry \cite{CFP} for an introduction to Thompson's groups $F$, $T$, and $V$, and to some of their generalizations. See also Examples \ref{example:V} and \ref{example:Houghton}
for (somewhat non-standard) definitions of Thompson's group $V$ and the Houghton groups $H_{n}$.

Later, Melanie Stein \cite{Stein} associated much more economical simplicial complexes to the class of Stein-Thompson groups and used the results to determine their finiteness properties and to compute homology groups.   

In recent years, there have been numerous studies of the finiteness properties of what might be called ``generalized Thompson groups'' -- see \cite{QV}, \cite{BelkMatucci},\cite{Bieri}, \cite{BraidedV}, \cite{FarleyHughes}, 
\cite{nV}, \cite{Lodha}, \cite{Skipper}, \cite{Thumann}, \cite{Witzel}, and \cite{WZcloning}, for instance. As a rule, these  follow the general strategy pioneered by Brown, often also using variants of the construction due to Stein.

Our purpose here is to offer a general setting for proving results of the above kind. We consider groups $G$ of bijections of a set $X$ with the property that the bijections have a  ``piecewise'' definition; i.e., for each $g \in G$, there are finite partitions $\{ U_{1}, \ldots, U_{n} \}$ and $\{ V_{1}, \ldots, V_{n} \}$
of $X$ such that the restrictions $g_{\mid U_{i}}: U_{i} \rightarrow V_{i}$ are each taken from a fixed set of partial transformations. We assume that the latter set is closed under compositions and inverses of functions. These closure properties make the set of partial transformations into an inverse semigroup. (We will make little use of the formal theory of inverse semigroups, for which \cite{SemigroupBook} is a reference.)    

Thus, we are led to start with a \emph{semigroup action} on a set $X$, which is simply a collection $S$ of partial bijections of $X$ that is closed under the operations of taking compositions and inverses.  We let  $\Gamma_{S}$ be the set of all bijections of $X$ that are piecewise determined by $S$ (in the sense of the previous paragraph). Straightforward checking shows that $\Gamma_{S}$ is a group.

To associate a natural geometry to the above set-up, we can follow the basic strategy of \cite{FarleyHughes} and \cite{Hughes}. Hughes \cite{Hughes} defined a certain class of groups that act by homeomorphisms on compact ultrametric spaces, which are called \emph{finite similarity structure (FSS) groups} in \cite{FarleyHughes}. Given a compact ultrametric space $X$,  a \emph{finite similarity structure}
is a function $\mathrm{Sim}$ that assigns to each pair $(B_{1},B_{2})$ of balls in $X$ a finite \emph{similarity set} $\mathrm{Sim}(B_{1},B_{2})$ of surjective similarities $h: B_{1} \rightarrow B_{2}$. The sets $\mathrm{Sim}(B_{1},B_{2})$ are required to satisfy various ``groupoid-like'' properties. A finite similarity structure
$\mathrm{Sim}$ determines a group $\Gamma_{\mathrm{Sim}}$. Elements of $\Gamma_{\mathrm{Sim}}$ are bijections $\gamma: X \rightarrow X$ that are \emph{locally determined} by $\mathrm{Sim}$; i.e., there are partitions $\{ B_{1}, \ldots, B_{n} \}$ and $\{ \widehat{B}_{1}, \ldots, \widehat{B}_{n} \}$ of $X$ such that the restrictions 
$\gamma_{\mid B_{i}}$ are members of the similarity sets
$\mathrm{Sim}(B_{i}, \widehat{B}_{i})$. 

Using the similarity sets, Hughes \cite{Hughes} defined an equivalence relation on pairs $(f,B)$, where $B$ is a compact ultrametric ball and $f:B \rightarrow X$ is a \emph{local similarity embedding}; i.e., there is a partition $\{ \widehat{B}_{1}, \ldots, \widehat{B}_{m} \}$ of $B$ into ultrametric balls such that the restrictions $f_{\mid \widehat{B}_{i}}$ are similarities from the similarity sets
$\mathrm{Sim}(\widehat{B}_{i}, f(\widehat{B}_{i}))$. The definition of the equivalence relation is in terms of a certain commutative diagram that involves elements from the similarity sets. (A generalized version of the definition occurs below, as Definition \ref{definition:pairs}.) We let $[f,B]$ denote the equivalence class of $(f,B)$. The main argument of \cite{Hughes} used the set of all such equivalence classes to prove that FSS groups have the Haagerup property.

In \cite{FarleyHughes}, the authors produced a proper action of $\Gamma_{\mathrm{Sim}}$ on a contractible simplicial complex. A vertex in the complex is a collection of equivalence classes
\[ \{ [f_{1},B_{1}], \ldots, [f_{n},B_{n}] \} \]
where the set $\{ f_{i}(B_{i}) : i = 1, \ldots, n \}$ is a partition of $X$. The vertices become a directed set under a natural expansion relation. (A general definition of the expansion relation appears here as Definition \ref{definition:expansion}.) The entire construction can be seen as a generalization of the ones from Brown \cite{Brown}. The authors were able to prove that certain classes of FSS groups have type $F_{\infty}$ by following Brown's method.

The theory of FSS groups, as sketched above, extends naturally to the groups $\Gamma_{S}$, where $S$ is an inverse semigroup acting on a set $X$. The inverse semigroup $S$ specifies certain subsets $D \subseteq X$ as \emph{domains}, so-called because each $D$ is the domain of some $s \in S$. We let $\mathcal{D}_{S}$ denote the set of domains. The domains play the role in the general theory that metric balls play in the more specific setting of FSS groups. Thus, we define an equivalence relation on the set of 
pairs $(f,D)$, where $D$ is a domain and $f: D \rightarrow X$ is an embedding that is locally determined by $S$ (in a suitable sense). Our definition of the equivalence relation (Definition \ref{definition:sstructure}) uses a \emph{structure function} $\mathbb{S}$, which assigns to each pair $(D_{1},D_{2}) \in \mathcal{D}_{S} \times \mathcal{D}_{S}$ a certain \emph{structure set}
$\mathbb{S}(D_{1},D_{2}) \subseteq S$. The  structure sets $\mathbb{S}(D_{1},D_{2})$ are required to have ``groupoid-like'' properties that resemble those of the similarity sets from the theory of FSS groups. We can then use the structure sets $\mathbb{S}(D_{1},D_{2})$ to define an equivalence relation on pairs $(f,D)$, where $D$ is a domain and $f:D \rightarrow X$ is an embedding that is locally determined by the inverse semigroup $S$. Denoting the equivalence class of $(f,D)$ by $[f,D]$, we can follow the general procedure of \cite{FarleyHughes} to produce a contractible complex upon which $\Gamma_{S}$ acts. Here we must also specify a certain \emph{pattern function} 
$\mathbb{P}$ to help determine the expansion relation (which is otherwise defined very much as in the basic theory of FSS groups). The entire construction generalizes the ones from \cite{FarleyHughes} and (therefore) the ones from \cite{Brown}.

Our finiteness results ultimately use simplified versions of the above complexes, which can be seen as generalizations of the ones constructed by Stein \cite{Stein}. Here we introduce \emph{expansion schemes}, denoted $\mathcal{E}$, which are a device for simplifying our first (directed set) construction in a controlled way. More specifically, an expansion scheme $\mathcal{E}$ determines the collection of \emph{$\mathcal{E}$-chains}, which are the simplices in our complexes. The $\mathcal{E}$-chains are direct descendents of Stein's ``elementary intervals''. The basic (and probably most important) examples of expansion schemes are very similar to ones described (not under the same name) in \cite{Stein}, but we also offer a more general theory, which holds the potential for greater flexibility in applications. 

Thus, to summarize, we propose the following sequence of choices in studying the finiteness properties of generalized Thompson groups:
\begin{enumerate}
\item Choose a generalized Thompson group $\Gamma_{S}$ with the associated set $S$ of partial transformations. The set $S$ specifies a set $\mathcal{D}_{S}$ of domains.
\item Pick a structure function $\mathbb{S}$ and a pattern function $\mathbb{P}$. These choices depend upon (but are not completely determined by) the inverse semigroup $S$. The pair $(\mathbb{S},\mathbb{P})$ is called an \emph{$S$-structure} (Definition \ref{definition:sstructure}). The choice of $S$-structure completely determines a directed set upon which $\Gamma_{S}$ acts by order-preserving bijections. The simplicial realization $\Delta$ of this directed set is analogous to the constructions due to Brown \cite{Brown}.
\item Pick an expansion scheme $\mathcal{E}$, which will depend upon the choice of $S$-structure. The expansion scheme will (in general) determine a more economical simplicial complex $\Delta^{\mathcal{E}}$ upon which $\Gamma_{S}$ acts. The complex $\Delta^{\mathcal{E}}$ is analogous to the ones introduced by Stein \cite{Stein}.
\end{enumerate}
 
These choices lead us up to the point where we can determine finiteness properties. The actual analysis of finiteness properties proceeds along well-established lines: we use Brown's finiteness criterion along with an analysis of the descending links in $\Delta^{\mathcal{E}}$. 

The details of the analysis are relegated to the main body of the paper; however, we mention two major ingredients. The first is the idea of a pseudovertex. A \emph{pseudovertex} is a collection of pairs
\[ \{ [f_{1}, D_{1}], \ldots, [f_{m}, D_{m}] \}, \] 
where the images $f_{i}(D_{i})$ are pairwise disjoint, but are not required to form a partition of $X$; i.e., it may be that their union is not all of $X$.  All vertices are pseudovertices, but not conversely. The second ingredient is the ubiquity of product decompositions among subcomplexes of $\Delta$ and $\Delta^{\mathcal{E}}$. Indeed, the product decompositions are most easily described using the vocabulary of pseudovertices - see, for instance, the crucial Proposition \ref{proposition:partitionedstar}. We therefore formulate most of our basic results about the expansion partial order in the general setting of pseudovertices. This is the case especially in Section \ref{section:Sstructures}, which contains the fundamentals about the partial order induced by expansion.

The main applications to finiteness properties of groups appear in the final section. We recover proofs that $V_{n,r}$ \cite{Brown}, $nV$ \cite{BrinHighD, nV} , $QV$ \cite{QV, Nucinkis}, R\"{o}ver's group \cite{BelkMatucci, Nek, Rover}, and FSS groups \cite{FarleyHughes} are of type $F_{\infty}$. We also consider a class of groups based on products and show that they are of type $F_{\infty}$. (These results are intended, in effect, as a ``proof of concept''. We make no attempt at being complete. Indeed, a complete exposition of possible applications to finiteness properties would not fit in any paper of a reasonable length.)

On the other hand, it may be useful to the reader to mention a few groups whose finiteness properties are not handled in this paper. Some of these groups can almost certainly be handled using similar techniques to the ones described here; others definitely require completely different methods. Among the former, we have Thompson's groups $F$ and $T$, which are not considered at all here. The Stein-Thompson groups \cite{Stein} are not considered, either, although it also seems likely that these groups are amenable to our methods. The Houghton groups $H_{n}$ are used as a running example, but we do not compute their finiteness properties here (as was done in \cite{Brown}). Moving to a more speculative case, one has the groups defined by Bieri and Sach \cite{Bieri}. We have defined some of these groups here (see Example \ref{example:bieri}), but an analysis of their finiteness properties is outside the scope of this paper. Another group to mention in the current context is the Lodha-Moore group \cite{LM}, which is known to have type $F_{\infty}$ \cite{Lodha}.

The braided Thompson group $BV$ is known to have type $F_{\infty}$ \cite{BraidedV}. This group, in contrast to the others mentioned above, is probably impossible to handle by anything resembling the methods of this paper. To explain why, it may be helpful to divide the generalized Thompson groups into multiple classes, depending upon the property of Thompson groups that they generalize. One class generalizes the ``piece-wise'' nature of Thompson's groups. This paper sketches a general theory of the finiteness properties of groups in this class. Another class generalizes the tree-pair definition of Thompson's groups. The group $BV$ clearly lies in the latter class; \cite{WZcloning} might be regarded as a general framework for studying this second ``tree-like'' class of generalized Thompson group. The work of Thumann \cite{Thumann} emphasizes what might be considered a ``strand diagram"-centered view of generalized Thompson groups.   

Let us describe the structure of the paper. In Section \ref{section:2}, we collect various results about simplicial complexes, especially the simplicial realizations of partially ordered sets. Most of the section consists of rather standard definitions, but we call attention to Lemma \ref{lemma:nconnectedness}, which ultimately allows us to compute the  connectivity of $\Delta^{\mathcal{E}}$ (Theorem \ref{theorem:bigone}). The section concludes with a brief description of simplicial products and the Nerve Theorem. In Section \ref{section:3}, we introduce the basic objects of study, namely the groups $\Gamma_{S}$, and we establish a few basic properties of the set of domains $\mathcal{D}_{S}$. This section also introduces the compact ultrametric property and inverse semigroup actions on product spaces. The section concludes with a number of examples. Section \ref{section:Sstructures} contains numerous fundamental ideas. We introduce the all-important idea of an $S$-structure and define the equivalence relation on pairs $(f,D)$. We also describe the partially ordered sets of pseudovertices, including the fact that suitable collections of pseudovertices are directed sets. The section ends with a number of examples of $S$-structures and an explicit characterization of the expansion relation in certain cases. In Section \ref{section:5}, we assemble the results of Section \ref{section:Sstructures} into a proof that the groups $\Gamma_{S}$ act on two explicitly-described contractible simplicial complexes $\Delta$ and $\Delta(\mathcal{V}^{ord})$, both of which depend on the $S$-structure. The latter complex is designed to have smaller vertex stabilizers than the former. We also describe these vertex stabilizers (in both cases), and offer information about the orbits under the $\Gamma_{S}$-action. Section \ref{section:expansionscheme} describes expansion schemes and the properties of complexes determined by them. The section concludes with descriptions of several expansion schemes. Section \ref{section:7} states Brown's finiteness criterion, and contains some generalities about the descending link. The section concludes with a sufficient condition for the descending link to be $n$-connected. The method uses the machinery of nerves of covers. Section \ref{section:8} establishes the finiteness properties  
of a wide range of generalized Thompson groups.

The main body of the argument begins in Section \ref{section:Sstructures}. Section \ref{section:2} contains a large amount of standard material, most of which is used much later in the paper. The reader may therefore want to skim Section \ref{section:2} on a first reading. Section \ref{section:3} introduces the basic objects of study (namely the groups $\Gamma_{S}$), but it should be possible to continue into the main body of the paper with just a few basic definitions. The reader can probably skip the long subsection on the compact ultrametric property (Subsection \ref{subsection:CUP}) on a first reading, referring back to it as necessary.

\section{Combinatorial preliminaries} \label{section:2}

This section collects a number of combinatorial preliminaries. In Subsection \ref{subsection:ordercomplex}, we review the order complex construction, which associates a simplicial complex to a partially ordered set. Subsection \ref{subsection:ordercomplex} also contains standard definitions of the descending and ascending links and stars.

Subsection \ref{subsection:Elemma} establishes a principle (Lemma \ref{lemma:nconnectedness}) that will eventually be used to prove that the complexes $\Delta^{\mathcal{E}}$ determined by expansion schemes are highly connected -- see Section \ref{section:expansionscheme} and Theorem \ref{theorem:bigone}.

Subsection \ref{subsection:product} contains a definition of simplicial products, which will be vital in analyzing the topology of the complexes to be constructed in the rest of the paper.

Finally, Subsection \ref{subsection:nerves} contains a definition of the nerve of a cover, and a suitably general form of the Nerve Theorem.

\subsection{Partially ordered sets and simplicial complexes} \label{subsection:ordercomplex}

\begin{definition} (order complex of a partially ordered set $P$; simplicial complex on $P$) If $P$ is a partially ordered set, then the \emph{order complex of $P$}, denoted $\Delta(P)$, is the abstract simplicial complex $( \mathcal{V}_{P}, \mathcal{S}_{P})$ such that 
\begin{align*}
\mathcal{V}_{P} &= P;  \\
\mathcal{S}_{P} &= \{ S \mid S \text{ is a non-empty finite chain in } P \}. 
\end{align*}

A \emph{simplicial complex on $P$} is a subcomplex of $\Delta(P)$ having $P$ as its vertex set.
\end{definition}

\begin{definition} (interval subcomplexes in the simplicial complex on $\mathcal{V}$)
Let $K = (\mathcal{V}, \mathcal{S})$ be a simplicial complex on the partially ordered set $\mathcal{V}$. If $x \in \mathcal{V}$, then we let
$K_{[x, \infty)}$ be the simplicial complex $(\mathcal{V}_{\geq x}, \mathcal{S}_{\geq x})$, 
where 
\begin{align*}
\mathcal{V}_{\geq x} &= \{ v \in \mathcal{V} \mid v \geq x \} \\
\mathcal{S}_{\geq x} &= \{ S \in \mathcal{S} \mid S \subseteq \mathcal{V}_{\geq x} \}.
\end{align*}
We can similarly define $K_{(-\infty,x]}$, by simply replacing each ``$\geq$" with ``$\leq$'' in the above definition. 

Finally, for vertices $x,y \in \mathcal{V}$, we define $K_{[x,y]} = (\mathcal{V}_{\geq x} \cap \mathcal{V}_{\leq y}, \mathcal{S}_{\geq x} \cap \mathcal{S}_{\leq y})$. 
\end{definition}

\begin{remark}
In practice, we will make little distinction between abstract simplicial complexes on the one hand and (geometric) simplicial complexes on the other. We will refer to both as simplicial complexes, trusting that the specific meaning will be clear from the context.

When specifying a subcomplex $K'$ of an (abstract) simplicial complex $K = (\mathcal{V}, \mathcal{S})$, it suffices to specify a collection of simplices $\mathcal{S}' \subseteq \mathcal{S}$ that is closed under taking non-empty subsets, since the vertex set for $K'$ is then determined by the equality
\[ \mathcal{V}' = \bigcup_{S' \in \mathcal{S}'} S'. \]
It will therefore be convenient to write $K' = \mathcal{S}'$ instead of $K' = (\mathcal{V}', \mathcal{S}')$ in what follows.  
\end{remark}

\begin{definition} \label{definition:linkandstar} (links and stars)
Let $K= (\mathcal{V}, \mathcal{S})$ be a simplicial complex, and let $v \in \mathcal{V}$. We recall that the \emph{star} of the vertex $v$, denoted $st(v,K)$, 
is 
\[  st(v,K) = \{ S \in \mathcal{S} \mid S \subseteq S', \text{ for some } S' \in \mathcal{S} \text{ such that }v \in S' \}. \] 

The \emph{link} of $v$ is defined as follows:
\[ lk(v,K) = \{ S - \{ v \} \mid S \in st(v,K) \}. \]
If the vertex set $\mathcal{V}$ is also partially ordered, then we can further define
\begin{align*}
st_{\downarrow}(v,K) &= st(v,K_{(-\infty,v]}); \\
st_{\uparrow}(v,K) &= st(v,K_{[v,\infty)}); \\
lk_{\downarrow}(v,K) &= lk(v,K_{(-\infty,v]}); \\
lk_{\uparrow}(v,K) &= lk(v,K_{[v,\infty)});
\end{align*}
These are the \emph{descending} and \emph{ascending stars} (respectively) and the \emph{descending} and \emph{ascending links} (respectively) of $v$ in $K$.
All of these are subcomplexes of $K$.
\end{definition}


\subsection{Ranked directed sets and $n$-connected simplicial complexes}
\label{subsection:Elemma}

\begin{definition} (ranked directed set) \label{definition:rankdir}
Let $(P,\leq)$ be a partially ordered set. We say that $(P,\leq)$ is a \emph{directed set} if, whenever $a, b \in P$, there is some $c \in P$ such that $a \leq c$ and $b \leq c$.

We say that  the directed set $(P, \leq)$ is a \emph{ranked directed set}
if there is also a \emph{ranking function} $r: P \rightarrow \mathbb{N}$ such that if $s_{1} < s_{2}$, then
$r(s_{1}) < r(s_{2})$. 
\end{definition}

\begin{lemma} \label{lemma:nconnectedness} (A sufficient condition for $n$-connectedness)
Let $(P, \leq)$ be a ranked directed set, and let $K$ be a simplicial complex on $P$. If, for every $x<y$ in $P$, $lk(x, K_{[x,y]})$ is
$(n-1)$-connected, then, for every pair $a<b$ in $P$, the complexes $K_{(-\infty, b]}$, $K_{[b,\infty)}$, and $K_{[a,b]}$ are $n$-connected. In particular, $K$ is $n$-connected.
\end{lemma}

\begin{proof}
We first assume that $lk(x, K_{[x,y]})$ is non-empty whenever $x<y$. We will show that $K_{(-\infty,b]}$ is connected. For this, it will suffice to show that every vertex in $K_{(-\infty,b]}$ can be connected to $b$ by a path. We let $x$ be a vertex of $K_{(-\infty,b]}$ and induct on the difference $r(b) - r(x)$. If $r(b) - r(x) =0$, then we must have $b=x$ (since $x \leq b$), so there is nothing to prove. If $r(b)-r(x) > 0$, then $lk(x,K_{[x,b]})$ is non-empty by our assumption. If $x'$ is a vertex in this link, then we must have
$r(b)-r(x') < r(b)-r(x)$. By induction, $x'$ is connected to $b$ by a path. Since $x$ and $x'$ are connected by an edge in $K_{[x,b]}$ by the definition of the link, $x$ can be connected to $b$ by a path. This shows that $K_{(-\infty,b]}$ is path connected, as required. Exactly the same argument shows that the complexes $K_{[a,b]}$ are connected under the same hypotheses.

Now suppose that $lk(x,K_{[x,y]})$ is path connected whenever $x<y$. We will argue that $K_{(-\infty,b]}$ is simply connected. Indeed, by the previous case, we know that $K_{(-\infty,b]}$ is connected. Let $c$ be a loop based at $b$. We can assume, by cellular approximation, that the image of $c$ lies entirely inside the $1$-skeleton of $K_{(-\infty,b]}$. We can further assume that $c$ is a combinatorial edge-path. We define 
\[ rk(c) = max \{ r(b)-r(x) \mid x \text{ is a vertex in } \mathrm{Im}\, c \} \]
and induct on $rk(c)$. If $rk(c)=0$, then $c$ is the constant path at $b$, and there is nothing to prove. Assume $rk(c)>0$. We let
\[ L = \bigcup_{x_{i}} K_{[x_{i},b]}, \]
where the union ranges over all vertices $x_{i}$ lying in $\mathrm{Im}\,c$. The complex $L$ may, a priori, be infinite, but there are at most finitely many vertices $x$ such that $r(b)-r(x) = rk(c)$, since all such vertices must lie on the original loop $c$. 

For a given such $x$ we can write
\[ L = st(x,L) \cup (L-\{x\}). \]
We note that $st(x,L)$ is contractible (since the star of any vertex in a simplicial complex is contractible) and that the intersection 
$st(x,L) \cap (L-\{x\})$ is homotopy equivalent to $lk(x,L) = lk(x,K_{[x,b]})$, and thus connected. It follows from van Kampen that $\pi_{1}(L - \{ x \}) \rightarrow \pi_{1}(L)$ is surjective, so $\pi_{1}(L, L-\{x\}) = 0$. This means that the loop $c$ can be altered in order to miss $x$, while remaining unchanged outside the star of $x$. We can argue similarly at each vertex $x$ such that $r(b)-r(x)=rk(c)$, eventually finding a new path $c'$ path homotopic to $c$ and satisfying
$rk(c') < rk(c)$. It follows by induction on $rk$ that $c$ is homotopic to the constant path, so $K_{(-\infty,b]}$ is simply connected. The same argument shows that $K_{[a,b]}$ is simply connected under the same hypotheses. 
   
Now assume that $lk(x,K_{[x,y]})$ is $(n-1)$-connected whenever $x<y$, where $n \geq 2$. We want to show that $K_{(-\infty,b]}$ is $n$-connected; by the induction hypothesis we know that $K_{(-\infty,b]}$ is $(n-1)$-connected. 
Let $f: S^{n} \rightarrow K$ be a continuous map. Let $L$ denote the smallest subcomplex of $K$ that contains $f(S^{n})$ (i.e., the \emph{carrier} of $f(S^{n})$). Then $L$ is a finite simplicial complex, since $f(S^{n})$ is compact.  We set
\[ L' = \bigcup_{x' \in L^{(0)}} K_{[x',b]}. \]
Note that $L'$ has at most finitely many vertices of minimal rank $m$, since all such vertices must be in $L$. We define
\[ rk(f) = max \{ r(b) - r(x) \mid x \in (L')^{(0)} \}. \]
Assume that $x$ is such that $r(b)-r(x) = rk(f)$. We can express $L'$ as the union 
$(L'-\{ x \}) \cup st(x,L')$, where the intersection is homotopy equivalent to $lk(x,L')= lk(x,K_{[x,b]})$, which is $(n-1)$-connected by hypothesis. The Mayer-Vietoris sequence combined with the relative Hurewicz theorem now implies that
\[ \pi_{n}(L', L' - \{ x \}) = 0. \]
Thus, we may homotope the map $f: S^{n} \rightarrow L'$ so that its image lies in the subcomplex of $K$ spanned by 
$(L')^{(0)} - \{ x \}$. We can do this while keeping $f$ unchanged outside the star of $x$. We repeat this procedure until the carrier of the new map $f_{1}: S^{n} \rightarrow K$ contains only vertices of rank strictly greater than $m$. If follows that $rk(f_{1}) < rk(f)$, which shows, by induction, that $f$ is null-homotopic in $K_{(-\infty,b]}$. It follows that $K_{(-\infty,b]}$ is $n$-connected. An exactly similar argument shows that 
$K_{[a,b]}$ is $n$-connected under the same hypotheses.

This proves the lemma for the complexes $K_{(-\infty,b]}$ and $K_{[a,b]}$. 
The remaining cases follow from these cases and from the directed set condition. Indeed, assume that $K_{(-\infty,b]}$ is $n$-connected for every $b \in K^{(0)}$. Consider any map $f: S^{n} \rightarrow K$. The carrier $L$ of $f$ is a finite subcomplex of $K$, and therefore there is a vertex $b \in K^{(0)}$ that is a common upper bound of all vertices in $L$. Thus, $f$ is null-homotopic in $K_{(-\infty,b]}$ and, thus, in $K$. This proves the lemma for $K$; the proof for the complexes $K_{[b,\infty)}$ is similar. 
\end{proof}

\subsection{Products of simplicial complexes} \label{subsection:product}

It is well-known that the product of simplicial complexes does not, in general, have a natural simplicial complex structure. However, given a family $P_{1}, \ldots, P_{n}$ of partially ordered sets, one can compare the order complex $\Delta(P_{1} \times \ldots \times P_{n})$ of the product 
with the product $\Delta(P_{1}) \times \ldots \times \Delta(P_{n})$ of the individual order complexes. A result that can be found in Walker \cite{Walker} shows that these spaces are homeomorphic (with respect to the compactly generated topology), which allows us to put a simplicial complex structure on the latter space. Indeed, more importantly, this enables us to put a simplicial complex structure on products $K_{1} \times \ldots \times K_{n}$, where $K_{i}$ is a subcomplex of $\Delta(P_{i})$. We summarize this result and a few related consequences in this subsection.

\begin{definition} \label{definition:product} (simplicial product)
Let $P_{1}, P_{2}, \ldots, P_{k}$ be partially ordered sets and, for $i=1, \ldots, k$, let $K_{i}$ be a simplicial complex on 
$P_{i}$. We endow $\prod_{i=1}^{k} P_{i}$
with the natural coordinate-wise partial order; thus, if $v = (v_{1}, \ldots, v_{k}) \in \prod_{i=1}^{k} P_{i}$ and $v' = (v'_{1}, \ldots, v'_{k}) \in \prod_{i=1}^{k} P_{i}$, then $v \leq v'$ if and only if $v_{i} \leq v'_{i}$ for all $i \in \{ 1, \ldots, k \}$. For $j=1, \ldots, k$, we have natural projection maps
\[ \pi_{j}: \prod_{i=1}^{k} P_{i} \rightarrow P_{j} \]
satisfying $\pi_{j}(v_{1}, \ldots, v_{j}, \ldots, v_{k}) = v_{j}$. 

The \emph{simplicial product} of $K_{1}, \ldots, K_{k}$ is the simplicial complex on $\prod_{i=1}^{k} P_{i}$ with the property that a chain $C \subseteq \prod_{i=1}^{k} P_{i}$ is a simplex if and only if $\pi_{j}(C)$ is a simplex in $K_{j}$, for $j=1, \ldots, k$. 
\end{definition}

\begin{theorem} \label{theorem:product} (Product Theorem) Let $K$ denote the simplicial product of
$K_{1}, \ldots,$ $K_{k}$.
\begin{enumerate}
\item the geometric realization of $K$ is homeomorphic to the product $K_{1} \times \ldots \times K_{k}$ of the realizations of the factors;
\item the geometric realization of the link $lk(v,K)$ of a vertex $v = (v_{1}, \ldots, v_{k})$ is homeomorphic to the join of the realizations of the factors; i.e., 
\[ lk(v,K) \cong \Asterisk_{i=1}^{k} lk(v_{i},K_{i}). \]
\end{enumerate}
\end{theorem}

\begin{proof}
For the first statement, see Theorem 3.2 from \cite{Walker}. Note, in particular, that the formula for the homeomorphism in Theorem 3.2 restricts to the desired homeomorphism in our case.  The second statement is Exercise
2.24(3) on page 24 of \cite{RS}.
\end{proof}

\subsection{Nerves of covers and the Nerve Theorem} \label{subsection:nerves}

Our applications to finiteness properties of groups will involve the Nerve Theorem. We recall a standard definition of the nerve; the form of the Nerve Theorem that we will use can be found in \cite{AB}.

\begin{definition} \label{definition:nerve} (the nerve of a cover) Let $S$ be any set, and let $\mathcal{C}$ be a cover of $S$. We let $\mathcal{N}(\mathcal{C})$ denote the \emph{nerve of the cover $\mathcal{C}$}. The vertices of $\mathcal{N}(\mathcal{C})$ are the non-empty elements $C \in \mathcal{C}$. A subset $\{ C_{0}, \ldots, C_{k} \} \subseteq \mathcal{C}$ is a simplex if and only if $C_{0} \cap \ldots \cap C_{k} \neq \emptyset$.
\end{definition}

\begin{theorem} \label{theorem:Nerve} \cite{AB} (Nerve Theorem) Let $K$ be a simplicial complex and let 
$\{ K_{i} \}_{i \in I}$ be a family of subcomplexes of $K$ such that $K = \bigcup_{i \in I} K_{i}$. If every non-empty finite intersection $K_{i_{1}} \cap \ldots \cap K_{i_{t}}$ is  $(k-t+1)$-connected, then $K$ is $k$-connected if and only if $\mathcal{N}(\{K_{i}\}_{i\in I})$ is $k$-connected. \qed
\end{theorem}

\section{The group determined by an inverse semigroup} \label{section:3}

Our viewpoint throughout the rest of the paper will be that a generalized Thompson group is determined locally by a fixed set of partial bijections, which naturally has the structure of an inverse semigroup $S$. We recall the definition of inverse semigroups in Subsection \ref{subsection:inverse}. We will use very little of the theory of inverse semigroups, all of which can be found in \cite{SemigroupBook}. In Subsection \ref{subsection:domains}, we collect a few basic properties that the set of domains $\mathcal{D}_{S}$ must have, where a ``domain'' is simply the domain of some element of $S$. Subsection \ref{subsection:thegroups} defines the groups $\Gamma_{S}$, which play the role of generalized Thompson groups. 

Subsection \ref{subsection:CUP} describes the compact ultrametric property, which is of great importance in the examples and applications of the theory to be developed. Subsection \ref{subsection:productaction} describes the actions of inverse semigroups on products.

The section concludes with Subsection \ref{subsection:examples}, which contains several examples.

\subsection{Inverse semigroups and monoids} \label{subsection:inverse}

\begin{definition} (inverse semigroup; inverse monoid) 
Let $S$ be a set with an associative binary operation (i.e., a \emph{semigroup}). An element $e \in S$ is an \emph{idempotent} if $e^{2} = e$.

We say that $S$ is \emph{regular} if for every $x \in S$, there is $y \in S$ such that
$xyx=x$. We say that a regular semigroup $S$ is an \emph{inverse semigroup} if any two idempotents of $S$ commute. If an inverse semigroup $S$ has a two-sided identity element, then $S$ is an \emph{inverse monoid}. 

If $S$ is an inverse semigroup, then for every $x \in S$, there is a unique $y \in S$ such that $xyx=x$ and $yxy=y$ \cite{SemigroupBook}. This $y$ is called the \emph{inverse} of $x$. We will often denote this inverse by $x^{-1}$.
\end{definition}

\begin{definition} (partial bijections) Let $X$ be a set. A \emph{partial bijection} of $X$ is a bijection $f: A \rightarrow B$ between subsets $A$ and $B$ of $X$. 

If $f: A \rightarrow B$ and $g: C \rightarrow D$ are partial bijections of $X$, then the \emph{composition} is the partial bijection $g \circ f: f^{-1}(C) \rightarrow g(B \cap C)$ defined by $(g \circ f)(x) = g(f(x))$, for each $x \in f^{-1}(C)$. 
\end{definition}

\begin{proposition} (inverse semigroups as sets of partial bijections)
Let $X$ be an arbitrary set, and let $S$ be a set of partial bijections of $X$ that is closed under compositions and inverses. The set $S$ is an inverse semigroup under the operation of composition. The inverse of $s \in S$ is its usual inverse (as a function). 
Every inverse semigroup arises in this way.
\end{proposition}

\begin{proof}
Assume that $S$ is a collection of partial bijections with the given properties. Let $s \in S$ and assume $s: A \rightarrow B$, where $A$, $B \subseteq X$. We note that $s^{-1}s = id_{A}$, and $id_{A} \in S$ because $S$ is closed under inverses and compositions. We now have $ss^{-1}s = s \circ id_{A} = s$, so $S$ is regular.

We claim that idempotents in $S$ all have the form $id_{C}: C \rightarrow C$, for some $C \subseteq X$. Indeed, let $e: A \rightarrow B$ be an idempotent. Since $e(e(x)) = e(x)$ for all $x$ in $A$, $e(x) = x$ for all $x \in A$ (because $e$ is injective). Thus, 
$e = id_{A}$ and $A = B$.

It is now clear that any two idempotents in $S$ commute, so $S$ is an inverse semigroup. 

The converse, that every inverse semigroup is realizable as a set of partial bijections, is the content of the Wagner-Preston Theorem \cite{SemigroupBook}, which is the counterpart for inverse semigroups of Cayley's Theorem for groups.
\end{proof}

Thus, inverse semigroups are the algebraic structure corresponding to partial bijections in the same way that groups are the algebraic structure corresponding (via Cayley's Theorem) to permutations.

\begin{convention} \label{convention:pb} For the rest of the paper, we fix a set $X$. 
We let $PB(X)$ denote the set of partial bijections of $X$. Note that $PB(X)$ is (of course) an inverse semigroup under composition.
We fix, for the remainder of the argument, an inverse semigroup $S \subseteq PB(X)$ such that $S$ contains the empty function, which we denote by $0$ when necessary.

It will occasionally be useful to refer to the above set-up as an \emph{action of $S$ on $X$}. We emphasize, however, that the inverse semigroups under consideration  are always defined as sets of partial bijections.

Note that the empty function is a zero in $S$; i.e.,  if $s \in S$, then
$0 \circ s = 0 = s \circ 0$.
\end{convention}

\subsection{The set of domains} \label{subsection:domains}

\begin{definition} A \emph{domain} $D$ is the domain of some $s \in S$. We let $\mathcal{D}_{S}$ denote the set of all domains
as $s$ ranges over all $s \in S$. We typically write $\mathcal{D}$ instead if $S$ is understood.

We let $\mathcal{D}^{+}_{S}$ denote the subcollection of non-empty domains. We similarly write $\mathcal{D}^{+}$ if $S$ is understood.
\end{definition}

\begin{remark} \label{remark:ranges} We note that the range (or image) of any $s \in S$ is also a domain in the above sense, due to the fact that $S$ is closed under inverses.
\end{remark}

\begin{lemma} \label{lemma:domainproperties} (closure properties of $\mathcal{D}$) The set $\mathcal{D}$ is closed under intersections and translation by elements of $S$; that is,
\begin{enumerate}
\item if $D_{1}$ and $D_{2}$ are in $\mathcal{D}$, then $D_{1} \cap D_{2} \in \mathcal{D}$.
\item if $D \in \mathcal{D}$ and $s \in S$, then $s(D) \in \mathcal{D}$.
\end{enumerate}
\end{lemma}

\begin{proof}
We prove (1). Let $D_{1}, D_{2} \in \mathcal{D}$ and assume that the domains of $s_{1}$ and $s_{2} \in S$ are
$D_{1}$ and $D_{2}$, respectively. It follows that $s^{-1}_{1} s_{1} = id_{D_{1}}$
and $s^{-1}_{2}s_{2} = id_{D_{2}}$. Thus, 
\[ s^{-1}_{1}s_{1}s^{-1}_{2}s_{2}  = id_{D_{1}} \circ \id_{D_{2}} = id_{D_{1} \cap D_{2}}.\]
Since $s^{-1}_{1}s_{1}s^{-1}_{2}s_{2} \in S$, $D_{1} \cap D_{2} \in \mathcal{D}$.

Now we prove (2). Let $D \in \mathcal{D}$ and let $s \in S$. There is some $\hat{s} \in S$ such that
$D$ is the domain of $\hat{s}$. It follows that $\hat{s}^{-1} \hat{s} = id_{D}$,
so $s \hat{s}^{-1} \hat{s} = s \circ id_{D}$. The image of the latter function is
$s(D)$. It follows that the domain of $id_{D} \circ s^{-1}$ is $s(D)$, so $s(D) \in \mathcal{D}$.
\end{proof}

\begin{corollary} \label{corollary:closureunderrestriction}(Closure under restriction to subdomains)
The inverse semigroup $S$ is closed under restriction to subdomains. That is, if
$s \in S$ and $D$ is contained in the domain of $s$, then $s_{\mid D} \in S$.
\end{corollary}

\begin{proof}
Let $s \in S$ and let $D \in \mathcal{D}_{S}$. Since $D$ is a domain, it follows that there is some $t \in S$ having the domain $D$. It now follows directly that $s_{\mid D} = st^{-1}t \in 
S$.
\end{proof}

\begin{convention} \label{convention:Xadomain} We will assume in all that follows that $X$ can be expressed as a finite disjoint union of domains.
This assumption is satisfied in all interesting cases that come to mind (and is automatic under certain general hypotheses; see [compact u-metric case], for instance).

If $X$ cannot be so expressed, one option is simply to add $id_{X}$ to the set $S$. This has  the effect of adjoining an identity to $S$ and also forces $X$ to be a member of $\mathcal{D}$. It is occasionally inconvenient to include $X$ in the set of domains, however, so we will not do this in general. 
\end{convention}

\subsection{The group determined by $S$} \label{subsection:thegroups} 

\begin{definition} \label{definition:hatS}(locally determined by $S$) Let $A, B \subseteq X$. A bijection $\hat{s}: A \rightarrow B$ is \emph{locally determined by $S$} if, for some $m \geq 0$,
\[ A = \coprod_{i=1}^{m} D_{i}; \quad B= \coprod_{i=1}^{m} E_{i}, \]
$\hat{s}_{\mid D_{i}}$ is a bijection from $D_{i}$ to $E_{i}$ for each $i$, and $\hat{s}_{\mid D_{i}} \in S$ for each $i$. Note that the sets $D_{i}$ are assumed to be pairwise disjoint and the $E_{i}$ are likewise assumed to be pairwise disjoint.

We let $\widehat{S}$ denote the set of partial bijections of $X$ that are locally determined by $S$.
We let $\Gamma_{S}$ denote the subset of $\widehat{S}$ consisting of bijections of $X$. We say that $\Gamma_{S}$ is the \emph{group locally determined by $S$}. 
\end{definition}

\begin{proposition} The set $\widehat{S}$ is an inverse semigroup and $\Gamma_{S}$ is a group with respect to the natural operations.
\end{proposition}

\begin{proof}
It suffices to show that $\widehat{S}$ is an inverse semigroup. It will follow that $\Gamma_{S}$ is a group, since $\Gamma_{S}$ is non-empty by Convention \ref{convention:Xadomain} and the property of being a self-bijection of $X$ is closed under taking compositions and inverses.

Let $\hat{s} \in \widehat{S}$. We can write
\[ \hat{s} = \coprod_{i=1}^{m} s_{i}, \]
where $s_{i}: A_{i} \rightarrow B_{i}$ and $s_{i} \in S$, for $i=1, \ldots, m$, and where each of $\{ A_{i} \mid i = 1, \ldots, m \}$ and $\{ B_{i} \mid i = 1, \ldots, m \}$ is a collection of pairwise disjoint domains. It follows directly that
\[ \hat{s}^{-1} = \coprod_{i=1}^{m} s^{-1}_{i}. \]
Thus $\hat{s}^{-1} \in \widehat{S}$ because each $s^{-1}_{i} \in S$, due to the fact that $S$ is closed under inverses. It follows that $\widehat{S}$ is closed under inverses.

Let $\hat{s}_{1}, \hat{s}_{2} \in \widehat{S}$. We write 
\[ \hat{s}_{1} = \coprod_{i=1}^{m} \hat{s}_{1i}; \quad \hat{s}_{2} = \coprod_{j=1}^{n} \hat{s}_{2j}. \]
The composition is
\[ \hat{s}_{1}\hat{s}_{2} = \coprod_{(i,j) \in \mathcal{I}} \hat{s}_{1i}\hat{s}_{2j}, \]
where $(i,j) \in \mathcal{I}$ if and only if $\widehat{D}_{1i} \cap \widehat{E}_{2j} \neq \emptyset$, where $\widehat{D}_{k\ell}$ and $\widehat{E}_{k\ell}$ are the domain and image (respectively) of $\hat{s}_{k \ell}$. It follows that $\widehat{S}$ is closed under compositions as well. Thus, $\widehat{S}$ is an inverse semigroup.
\end{proof}

\subsection{The compact ultrametric property} \label{subsection:CUP}

Thompson's groups $F$, $T$, and $V$ all admit descriptions as transformations of the Cantor set $\mathcal{C} = \prod_{n=1}^{\infty} \{ 0, 1 \}$. A natural metric   makes $\mathcal{C}$ into a compact ultrametric space (see Example \ref{example:clarifying} for additional details). In fact, many of the generalizations of Thompson's groups from the literature can also be described via actions on a suitable compact ultrametric space. Inverse semigroup actions on compact ultrametric spaces will be a significant source of examples in this paper, and carry the benefit of being particularly easy to work with.

For our purposes, it is usually not important to work directly with the metric. Indeed, it can be burdensome in practice to produce such a metric in the first place. Our approach will be to abstract the basic properties of balls in a compact ultrametric space. We will say that the domains in a set $X$ satisfy the ``compact ultrametric property'' (Definition \ref{definition:umetricproperty}) if the domains have combinatorial properties like the balls in a compact ultrametric space. The presence (or even existence) of an ultrametric is unimportant.  

Our main goal here is to describe the combinatorics of the domains in question.

\begin{definition} (the compact ultrametric property) \label{definition:umetricproperty} Assume that
\begin{enumerate}
\item (nested domains) if $D_{1}, D_{2} \in \mathcal{D}_{S}$ and $D_{1} \cap D_{2} \neq \emptyset$, then $D_{1} \subseteq D_{2}$ or $D_{2} \subseteq D_{1}$, and
\item (finite complementation) if $D \in \mathcal{D}_{S}$, then the complement $X-D$ may be written as the union of finitely many members of $\mathcal{D}_{S}$.
\end{enumerate}
We say that $\mathcal{D}_{S}$ has the \emph{compact ultrametric property}. 
\end{definition}

\begin{example} \label{example:clarifying} Let $X$ be the set of all infinite binary strings. We define a metric on $X$ as follows: if 
\[a = a_{1}a_{2}a_{3}\ldots \in X \quad \text{ and } \quad  b = b_{1}b_{2}b_{3}\ldots \in X,\] we let $p(a,b)$ denote the length of the greatest prefix common to both. (Thus, if $a_{1}a_{2}\ldots a_{k} = b_{1}b_{2}\ldots b_{k}$ but $a_{k+1} \neq b_{k+1}$, then $p(a,b) = k$.) We then define $d(a,b) = 2^{-p(a,b)}$. It can be checked that the function $d: X \times X \rightarrow \mathbb{R}$ is an \emph{ultrametric}; i.e., $d$ is a metric that satisfies the following strong version of the triangle inequality: if $x,y,z \in X$, then 
\[ d(x,y) \leq \mathrm{max}\{  d(x,z), d(y,z) \}.\]
Each finite string $\omega = \omega_{1}\ldots \omega_{k}$ determines a ball $B_{\omega}$, which simply consists of all infinite strings in $X$ having $\omega$ as a prefix. All balls in $X$ can be described in this way. Additionally, the metric space $(X,d)$ is
compact, although this seems somewhat less obvious.

In Example \ref{example:V}, we will define an inverse semigroup $S_{V}$ acting on $X$ with the property that the set $\mathcal{D}^{+}_{S_{V}}$ of domains consists precisely of the balls $B_{\omega}$. The group $\Gamma_{S_{V}}$ is Thompson's group $V$. Note that it is entirely straightforward to check that $\mathcal{D}^{+}_{S_{V}}$ satisfies the conditions from Definition \ref{definition:umetricproperty}.
\end{example} 

\begin{convention} \label{convention:standingumetric}
We will assume that $\mathcal{D}_{S}$ has the compact ultrametric property for the remainder of this subsection.
\end{convention}

\begin{remark} \label{remark:partitioning}
We note that, in Definition \ref{definition:umetricproperty}(2), the difference $X-D$ may be written as the \emph{disjoint} union of finitely many members of $\mathcal{D}_{S}$, because of property (1): if any two domains in the union overlap, one must be contained in the other, so we throw out the smaller domain and repeat as necessary, until we obtain a partition of $X - D$.

We note also that properties (1) and (2) result in a finite difference property: if $D_{1}$ and $D_{2}$ are domains and $D_{1} \subseteq D_{2}$, then $D_{2} - D_{1}$ may be written as a disjoint union of finitely many domains. (Simply consider a partition $\mathcal{P}$ of $X-D_{1}$ by finitely many domains. The set $\mathcal{P}' = \{ D_{2} \cap P \mid P \in \mathcal{P} \text{ and }D_{2} \cap P \neq \emptyset  \}$ is the desired partition.)
\end{remark}

\begin{proposition} (bounding ascending chains in $\mathcal{D}_{S}$) \label{proposition:boundedascending}
If $D \in \mathcal{D}^{+}_{S}$, then there is a constant $m \in \mathbb{N}$ such that every ascending chain
\[ D= D_{1} \subsetneq D_{2} \subsetneq D_{3} \subsetneq \ldots. \]
has total length no more than $m$.

In particular, there is no infinite strictly ascending chain of domains starting at $D$.
\end{proposition}

\begin{proof}
By Definition \ref{definition:umetricproperty}(2) and Remark \ref{remark:partitioning}, we can express $X-D$ as a finite disjoint union of domains $\widehat{D}_{1}, \widehat{D}_{2}, \ldots, 
\widehat{D}_{n}$. It follows that
\[ \mathcal{P} = \{ D, \widehat{D}_{1}, \ldots, \widehat{D}_{n} \} \]
is a partition of $X$. We write $D = \widehat{D}_{0}$.

We claim that every chain
\[ D = D_{1} \subsetneq D_{2} \subsetneq \ldots \]
has length no more than $n+1$ (i.e., we can set $m=n+1$). 

If not, then we can pick $x_{1} \in D_{1}$, $x_{2} \in D_{2}-D_{1}, x_{3} \in D_{3} - D_{2}, \ldots$ such that 
$T = \{ x_{1}, x_{2}, \ldots \}$ has at least $n+2$ elements. It follows from the Pigeonhole principle that there are $x_{i}, x_{j} \in T$, $i<j$, such that both $x_{i}$ and $x_{j}$ are in $\widehat{D}_{k}$, for some $k \in \{ 0, 1, \ldots, n \}$. Since $x_{i} \in D_{i} \cap \widehat{D}_{k}$, we must have $D_{i} \subseteq \widehat{D}_{k}$ or $\widehat{D}_{k} \subseteq D_{i}$; the former possibility is ruled out because it implies $D = D_{i} \cap D \subseteq \widehat{D}_{k} \cap \widehat{D}_{0} = \emptyset$, a contradiction. Thus, $\widehat{D}_{k} \subseteq D_{i}$, which implies $x_{j} \in D_{i} \subseteq D_{j-1}$, a contradiction.
\end{proof}

\begin{definition} \label{definition:depth}
Let $D$, $E$ be nonempty domains in $\mathcal{D}_{S}$, with $D \subseteq E$. We let 
$\mathrm{depth}_{E}(D)$ be the length of the largest increasing chain of domains beginning at $D$ and ending with $E$.  

If $\mathcal{P}$ is a finite partition of $E$, then
\[ \mathrm{depth}_{E}(\mathcal{P}) = \text{max} \{ \mathrm{depth}_{E}(D) \mid D \in \mathcal{P}\}. \]
\end{definition}

\begin{remark} \label{remark:shallow}
In both cases of Definition \ref{definition:depth}, the depth is a positive integer. E.g., $\mathrm{depth}_{E}(E)=1$.
\end{remark}

\begin{lemma} (greatest proper subdomains) \label{proposition:maximaldomains} Let $D \in \mathcal{D}^{+}$. Either
\begin{enumerate}
\item $D$ has no proper non-empty subdomain, or
\item for each $x \in D$, there is a greatest proper subdomain $D_{x}$ of $D$ such that $x \in D_{x}$.
\end{enumerate}
\end{lemma}

\begin{proof}
Let $D \in \mathcal{D}^{+}$, and suppose that $D' \subseteq D$ is a proper non-empty subdomain. By Remark \ref{remark:partitioning}, $D-D'$ can be partitioned by finitely many domains; thus, we have a partition of $D$ in the form
\[ \mathcal{P} = \{ D', D_{1}, D_{2}, \ldots, D_{m} \}, \]
where each member of $\mathcal{P}$ is a non-empty domain.

Let $x \in D$. The set $\mathcal{D}_{x} = \{ \widehat{D} \in \mathcal{D}^{+} \mid x \in \widehat{D} \subsetneq D \}$ is necessarily a chain, by the nested domains property from Definition \ref{definition:umetricproperty}. The set $\mathcal{D}_{x}$ is clearly non-empty (since, in particular, some domain from $\mathcal{P}$ must be in $\mathcal{D}_{x}$), and must therefore contain a maximal element; otherwise, we could select an infinite ascending chain
\[ \widehat{D}_{1} \subsetneq \widehat{D}_{2} \subsetneq \widehat{D}_{3} \subsetneq \ldots \]
from $\mathcal{D}_{x}$; this would contradict Proposition \ref{proposition:boundedascending}. 
A maximal element $D_{x}$ of $\mathcal{D}_{x}$ is the desired greatest proper subdomain containing $x$.
\end{proof}   

\begin{corollary} (the maximal partition of $D$) \label{corollary:maximal} Let $D \in \mathcal{D}^{+}$, and assume that $D$ properly contains some non-empty subdomain. For $x \in D$, we let $D_{x}$ denote the maximal proper subdomain of $D$ that contains $x$. 

The collection 
\[ \mathcal{P}_{D} = \{ D_{x} \mid x \in D \}. \]
is a finite partition of $D$ by proper subdomains, and any other partition $\mathcal{P}'$ of $D$ by proper subdomains is a refinement of $\mathcal{P}_{D}$.
\end{corollary}

\begin{proof}
First, 
each set $D_{x}$ is clearly non-empty. Assume that $D_{x_{1}} \cap D_{x_{2}} \neq \emptyset$. By the nested domains condition from Definition \ref{definition:umetricproperty}, $D_{x_{1}} \subseteq D_{x_{2}}$ or $D_{x_{2}} \subseteq D_{x_{1}}$. 
Neither inclusion can be proper, by the maximality of 
$D_{x_{1}}$ and $D_{x_{2}}$, so $D_{x_{1}} = D_{x_{2}}$. Thus, 
the elements of $\mathcal{P}_{D}$ are pairwise disjoint. It is also clear that 
\[ D = \bigcup_{x \in D} D_{x}, \]
so $\mathcal{P}_{D}$ is a partition of $D$ by proper subdomains.

We must show that $\mathcal{P}_{D}$ is finite. As in the proof of Proposition \ref{proposition:maximaldomains}, we can produce a partition $\mathcal{P} = \{ \widehat{D}_{1}, \ldots, \widehat{D}_{m} \}$ of $D$ into proper subdomains. Since the elements of $\mathcal{P}_{D}$ are maximal with respect to inclusion, and each $\widehat{D}_{i}$ intersects some $D_{x}$ non-trivially, we must have $\widehat{D}_{i} \subseteq D_{x_{i}}$, for some $x_{i} \in D$. It follows that $\{ D_{x_{1}}, \ldots, D_{x_{m}} \} \subseteq \mathcal{P}_{D}$ is a cover of $D$ and thus $\{ D_{x_{1}}, \ldots, D_{x_{m}} \} = \mathcal{P}_{D}$. Thus, $\mathcal{P}_{D}$ is finite.

Now we show that every other partition $\mathcal{P}'$ of $D$ by proper subdomains is a refinement of $\mathcal{P}_{D}$. Let $D' \in \mathcal{P}'$. Since $D' \neq \emptyset$, there is some $x \in D' \subsetneq D$. It follows from the maximality of $D_{x}$ that $D' \subseteq D_{x}$. Thus, $\mathcal{P}'$ is a refinement of $\mathcal{P}_{D}$.
\end{proof}

\begin{definition} (the maximal partition)  \label{definition:maximal} Let $D$ be a domain. If $D$ properly contains a non-empty subdomain, then we let $\mathcal{P}_{D}$ denote the partition of $D$ from Corollary \ref{corollary:maximal}; otherwise, we let
$\mathcal{P}_{D} = \{ D \}$.
In either case, 
we call $\mathcal{P}_{D}$ \emph{the maximal partition of $D$}.
\end{definition}

\begin{remark} (maximal vs. minimal) \label{remark:maxormin}
The adjective ``maximal'' is slightly at odds with later definitions, notably the definition of expansion (Definition \ref{definition:expansion}). In the latter definition, taking refinements 
will result in larger objects, rather than smaller ones. The maximal partition will therefore represent a \emph{minimal} upper bound under the expansion relation.  
\end{remark}

\begin{lemma} \label{lemma:maximalgeneration} (maximal partitions generate all partitions) Let $D$ be a domain. For every finite partition $\mathcal{P}$ of $D$ into proper subdomains, there is a sequence of partitions
\[ \{ D \} = \mathcal{P}_{0}, \mathcal{P}_{1}, \ldots, \mathcal{P}_{\ell} = \mathcal{P}, \]
where, for $i=0, 1, \ldots, \ell-1$, 
\[ \mathcal{P}_{i+1} = (\mathcal{P}_{i} - \{ D' \}) \cup \mathcal{P}_{D'} \]
for some $D' \in \mathcal{P}_{i}$.
\end{lemma}

\begin{proof}
We prove the lemma by induction on $\mathrm{depth}_{D}(\mathcal{P})$, where $D$ is an arbitrary domain and $\mathcal{P}$ is an arbitrary partition of $D$. If $\mathrm{depth}_{D}(\mathcal{P}) = 1$, then $\mathcal{P} = \{ D \}$; if $\mathrm{depth}_{D}(\mathcal{P}) = 2$, then $\mathcal{P} = \mathcal{P}_{D}$. We may therefore assume that
$\mathrm{depth}_{D}(\mathcal{P})  = n \geq 3$.  By Corollary \ref{corollary:maximal}, $\mathcal{P}$ is a refinement of $\mathcal{P}_{D}$. Let
\[ \mathcal{P}_{D} = \{ E_{1}, \ldots, E_{m} \}. \]
For $i=1, \ldots, m$, let $\mathcal{P}_{i}$ denote the subset of $\mathcal{P}$ that partitions $E_{i}$. It suffices to show that $\mathrm{depth}_{E_{i}}(\mathcal{P}_{i}) \leq n-1$. This is clear; suppose $D' \in \mathcal{P}_{i}$ is such that there is a sequence
\[ D' = D''_{1} \subsetneq D''_{2} \subsetneq \ldots \subsetneq D''_{k} = E_{i}, \]
where $k \geq n$. It follows that we can append $D$ to the end of this sequence, resulting in a sequence of length at least $n+1$. This shows that the depth of $\mathcal{P}$ is at least $n+1$, a contradiction.
\end{proof}

Finally, we record a straightforward proposition for future reference.

\begin{proposition} \label{proposition:Sinvarianceofmaximalsubdivision}($S$-invariance of the maximal partition) 
If $s \in S$ and $D_{1}, D_{2} \in \mathcal{D}^{+}$ are the domain and image of $s$, respectively, then $s(\mathcal{P}_{D_{1}}) = \mathcal{P}_{D_{2}}$.
\qed
\end{proposition}

\subsection{Product actions} \label{subsection:productaction}

In this subsection, we briefly consider inverse semigroup actions on products $X_{1} \times \ldots \times X_{n}$. Our main goal here is to set some terminology. 

A secondary goal is to present Example \ref{example:pathology}, which will influence our definition of ``$S$-structures'' (Definition \ref{definition:sstructure}), and, thus, the definition of ``expansion'' (Definition \ref{definition:expansion}).

\begin{definition} \label{definition:products}
 (product actions) If $X_{1}, \ldots, X_{n}$ are sets and $S_{1}, \ldots, S_{n}$ are inverse semigroups such that, for $i=1, \ldots, n$, $S_{i}$ acts on $X_{i}$, then an $n$-tuple $(s_{1}, \ldots, s_{n}) \in S_{1} \times \ldots \times S_{n}$ defines a partial bijection of $X_{1} \times \ldots \times X_{n}$ as follows. If any of the $s_{i}$ are $0$, then $(s_{1}, \ldots, s_{n})$ determines the empty function. If none of the $s_{i}$ are $0$, then $(s_{1},s_{2}, \ldots, s_{n})$ is defined by the rule
 \[ (s_{1}, \ldots, s_{n})(x_{1}, \ldots, x_{n}) = (s_{1}(x_{1}), \ldots, s_{n}(x_{n})), \]
 where $x_{i}$ is in the domain of $s_{i}$, for $i=1, \ldots, n$. 
 
It will be convenient to call the above action the \emph{product action} of $S_{1} \times \ldots \times S_{n}$ on $X_{1} \times \ldots \times X_{n}$, even though the above collection of partial bijections is not isomorphic to the usual direct product (see Remark \ref{remark:Rees}). 

If $X_{1}= X_{2} = \ldots = X_{n}$ and $S_{1} = S_{2} = \ldots = S_{n}$, we will denote the product semigroup by $nS$. More generally, we may sometimes use the notation $S_{(1,\ldots,n)}$ to refer to the product action of $S_{1} \times \ldots \times S_{n}$.
\end{definition}

\begin{remark} \label{remark:Rees}
The above assignment of a partial bijection to an $n$-tuple $(s_{1}, \ldots, s_{n})$ descends to an action of the inverse semigroup 
 \[ (S_{1} \times \ldots \times S_{n})/I, \]
 on $X_{1} \times \ldots \times X_{n}$, where $I$ is the two-sided ideal 
 \[ I= \{ (s_{1}, \ldots, s_{n}) \mid s_{i} = 0, \text{ for some } i\}, \]
and the above quotient is the usual Rees quotient \cite{SemigroupBook} of the product semigroup $S_{1} \times \ldots \times S_{n}$ by the ideal $I$.
\end{remark}
 
\begin{remark} \label{remark:productdomains} (the set of domains for the product action)
If $S_{(1,\ldots,n)}$ is the product semigroup, then the corresponding set of domains is
\[ \mathcal{D}^{+}_{S_{(1,\ldots,n)}} = \{ D_{1} \times \ldots \times
D_{n} \mid D_{i} \in \mathcal{D}^{+}_{S_{i}} \}. \]
\end{remark}

\begin{example} \label{example:pathology} (partitions with no non-trivial proper coarsenings)
Partitions of a given domain into smaller domains will become increasingly important in later sections. It will be especially important to have a degree of control over the form that such partitions can take.

A case in point occurs when the set of domains $\mathcal{D}^{+}_{S}$ has the compact ultrametric property (Definition \ref{definition:umetricproperty}).  In such a case, as we have seen, any given non-empty domain $D$ admits a maximal partition $\mathcal{P}_{D}$, of which any other non-trivial partition $\mathcal{P}$ of $D$ is a refinement (see Corollary \ref{corollary:maximal}). Thus, every non-trivial partition of $D$ ``factors through'' a unique partition $\mathcal{P}_{D}$. Our example exhibits a strong contrast to this property.         

Let $X$ be the set of all infinite binary strings (as in Example \ref{example:clarifying}). There is an inverse semigroup $S_{V}$ such that 
\[ \mathcal{D}^{+}_{S_{V}} = \{ B_{\omega} \mid \omega \text{ is a  finite binary string} \} \]
(see Example \ref{example:V}). The set $\mathcal{D}^{+}_{S_{V}}$ satisfies the compact ultrametic property. Consider the product semigroup $3S_{V}$, which acts on triples of infinite strings. A domain for an element of $3S_{V}$ is a dyadic brick $B_{\omega_{1}} \times B_{\omega_{2}} \times B_{\omega_{3}}$, and all such bricks are domains. 

We denote the length of a finite binary string $\omega$ by $|\omega|$.
Consider the partition 
\begin{align*}
\mathcal{P}_{n} = &\{ B_{0} \times B_{0} \times B_{\epsilon} \}
\cup \{ B_{1} \times B_{\epsilon} \times B_{\omega 0} \mid |\omega|=n \} 
 \cup  \{ B_{0} \times B_{1} \times B_{\omega 0} \mid |\omega|=n \} \\
& \cup  \{ B_{\epsilon} \times B_{1} \times B_{\omega 1} \mid |\omega|=n \}
\cup  \{ B_{1} \times B_{0} \times B_{\omega 1} \mid |\omega|=n \}
\end{align*}
of $X^{3}$. 

Let $\mathcal{P}$ be any partition of $X^{3}$ into dyadic bricks, and suppose that $\mathcal{P}_{n}$ properly refines $\mathcal{P}$. It follows that there is some dyadic brick
$B = B_{\omega_{1}} \times B_{\omega_{2}} \times B_{\omega_{3}} \in \mathcal{P}$ such that $B$ contains two or more bricks from $\mathcal{P}_{n}$. Note also that $B$ is partitioned by bricks from $\mathcal{P}_{n}$, by the definition of refinement. We can use these two properties to prove that $\mathcal{P} = \{ X^{3} \}$. Thus, the only proper coarsening of $\mathcal{P}_{n}$ via dyadic bricks is the trivial partition $\{ X^{3} \}$. 

The proof involves analyzing several cases. Let us consider just part of the proof in case $n=1$; the rest can be argued similarly. Assume that $\mathcal{P}_{1}$ refines $\mathcal{P}$; let $B = B_{\omega_{1}} \times B_{\omega_{2}} \times B_{\omega_{3}} \in \mathcal{P}$ be a dyadic brick that contains two or more bricks from $\mathcal{P}_{1}$. Assume (for instance) that 
\[ B_{1} \times B_{\epsilon} \times B_{00}, B_{1} \times B_{\epsilon} \times B_{10} \subseteq B_{\omega_{1}} \times B_{\omega_{2}} \times B_{\omega_{3}}. \]
It follows that $B_{1} \subseteq B_{\omega_{1}}$, $B_{\epsilon} \subseteq B_{\omega_{2}}$, and $B_{10}, B_{00} \subseteq B_{\omega_{3}}$. We easily conclude that $B_{\omega_{1}} = B_{1}$ or $B_{\epsilon}$, while $B_{\omega_{2}} = B_{\omega_{3}} = B_{\epsilon}$. 
Now note that $B$ must have a non-empty intersection with the brick
$B_{\epsilon} \times B_{1} \times B_{01} \in \mathcal{P}_{1}$. It follows that $B$ contains the latter brick, which also forces $B_{\omega_{1}} = B_{\epsilon}$. Thus, $B = B_{\epsilon}^{3} = X^{3}$, which implies that $\mathcal{P} = \{ X^{3} \}$, as claimed.

It follows that if $\mathcal{T}$ is a family of partitions of $X^{3}$ into dyadic bricks, and $\mathcal{T}$ has the property that every proper partition of $X^{3}$ into dyadic bricks factors through (i.e., refines) one of the partitions in $\mathcal{T}$, then $\mathcal{T}$ is infinite.

The above property of $X^{3}$ is highly undesirable in the applications (to finiteness properties, for instance). We will therefore need to put careful restrictions on the allowable partitions in order to proceed successfully.    
 \end{example}

\subsection{Examples of the groups $\Gamma_{S}$} \label{subsection:examples}

\begin{example} \cite{Brown, CFP} \label{example:V} (Thompson's group $V$)
Let $X$ be the set of all infinite binary strings. If $\omega$ is any finite binary string, then we define $B_{\omega}$ to be the set of infinite binary strings having $\omega$ as a prefix. We let $\epsilon$ denote the empty binary string; thus, $B_{\epsilon}=X$.

For each pair $(\omega_{1}, \omega_{2})$ of finite binary strings, we define a transformation $\sigma_{\omega_{1}, \omega_{2}}: B_{\omega_{1}} \rightarrow B_{\omega_{2}}$ by the rule
\[ \sigma_{\omega_{1}, \omega_{2}}(\omega_{1}a_{1}a_{2}\ldots a_{n}\ldots) = \omega_{2}a_{1}a_{2}\ldots a_{n} \ldots. \]
Here the $a_{i}$ denote binary digits. Thus, $\sigma_{\omega_{1},\omega_{2}}$ removes the prefix $\omega_{1}$ and replaces it with $\omega_{2}$. It is clear that each $\sigma_{\omega_{1},\omega_{2}}$ is a partial bijection of $X$, and that $\sigma^{-1}_{\omega_{1}, \omega_{2}} = \sigma_{\omega_{2},\omega_{1}}$.

We let 
\[ S_{V} = \{ \sigma_{\omega_{1}, \omega_{2}} \mid \omega_{1},\omega_{2} \text{ are finite binary strings} \} \cup \{ 0 \}, \]
where $0$ denotes the empty function. It is straightforward to check that $S_{V}$ is an inverse semigroup. The associated set of domains is
\[ \mathcal{D}_{S_{V}}^{+} = \{ B_{\omega} \mid \omega \text{ is a finite binary string} \}. \]
We note that $\mathcal{D}_{S_{V}}^{+}$ satisfies the compact ultrametric property.
The group $\Gamma_{S_{V}}$ is Thompson's group $V$.
\end{example}

\begin{example} \cite{QV, Nucinkis} \label{example:QV} (The group $QV$) Let $\mathcal{T}$ be the rooted infinite binary tree. Let $X = \mathcal{T}^{0}$, the set of vertices of $\mathcal{T}$. Thus, $X$ is the set of all finite binary strings, including the empty string $\epsilon$, which is the root of $\mathcal{T}$. For any finite binary string $\omega$, we let
$\mathcal{T}_{\omega}$ be the infinite binary subtree having $\omega$ as its root. The set $\mathcal{T}^{0}_{\omega}$ therefore consists of all finite binary strings having $\omega$ as a (not necessarily proper) prefix. 

We define two types of transformations. If $\omega_{1}, \omega_{2}$ are finite binary strings, then define $\tau_{\omega_{1},\omega_{2}}$ to be
the unique bijection from $\{ \omega_{1} \}$ to $\{ \omega_{2} \}$. We define $\sigma_{\omega_{1},\omega_{2}}: \mathcal{T}_{\omega_{1}}^{0} \rightarrow \mathcal{T}_{\omega_{2}}^{0}$ by the rule
\[ \sigma_{\omega_{1},\omega_{2}}(\omega_{1}a_{1}\ldots a_{n}) = \omega_{2}a_{1} \ldots a_{n}. \]
Thus, $\sigma_{\omega_{1},\omega_{2}}$ removes the prefix $\omega_{1}$ and attaches the prefix $\omega_{2}$ as in Example \ref{example:V}, although this time $\sigma_{\omega_{1},\omega_{2}}$ determines a transformation of finite binary strings.

We let
\[ S_{QV} = \{ \sigma_{\omega_{1},\omega_{2}} \mid \omega_{1}, \omega_{2} \in X \} \cup \{ \tau_{\omega_{1}, \omega_{2}} \mid \omega_{1}, \omega_{2} \in X \} \cup \{ 0 \}, \]
 where $0$ is the empty function. 
 
 It is straightforward to check that $S_{QV}$ is an inverse semigroup under composition. The associated set of domains is
 \[ \mathcal{D}^{+}_{S_{QV}} = \{ \{ \omega \} \mid \omega \in X \} \cup \{ \mathcal{T}_{\omega}^{0} \mid \omega \in X \}. \]
 We note that $\mathcal{D}^{+}_{S_{QV}}$ satisfies the compact ultrametric property.
 The group $\Gamma_{S_{QV}}$ is isomorphic to $QV$.
 \end{example}
 
 
 \begin{example} \label{example:nV} \cite{BrinHighD, nV}(The Brin-Thompson groups $nV$) 
 Let $X$ denote the set of all infinite binary strings (as in Example \ref{example:V}). Consider the product action of $nS_{V}$ on $X^{n}$ (see Definition \ref{definition:products}). 
 
 The corresponding set of domains is
 \[ \mathcal{D}^{+}_{nS_{V}} = \{ B_{u_{1}} \times \ldots \times B_{u_{n}} \mid u_{i} \text{ is a finite binary string for each } i \}. \]
 Here $\mathcal{D}^{+}_{nS_{V}}$ does not satisfy the compact ultrametric property. (For instance, if $n=2$, the sets $B_{\epsilon} \times B_{0}$ and $B_{1} \times B_{\epsilon}$ intersect, but neither is contained in the other.) 
 The group $\Gamma_{nS_{V}}$ is isomorphic to $nV$.
 \end{example}

\begin{example} \label{example:Rover} \cite{BelkMatucci, Nek, Rover} (The R\"{o}ver group) If $x \in \{ 0, 1\}$, then let $\overline{x}$ denote the opposite binary digit; i.e., $\overline{0} = 1$ and $\overline{1} = 0$. We will let $X$ denote the set of all infinite binary strings (as in Example \ref{example:V}). 

We define four transformations of $a,b,c,d:X \rightarrow X$ by the following rules, where each $x_{i}$ represents an individual binary digit:
\begin{align*}
a(x_{1}x_{2}\ldots) &= \overline{x}_{1}x_{2}x_{3}\ldots; \\
b(0x_{2}x_{3} \ldots) & = 0a(x_{2}\ldots); \\
b(1x_{2}x_{3} \ldots) & = 1c(x_{2}\ldots); \\
c(0x_{2}x_{3} \ldots) & = 0a(x_{2}\ldots); \\
c(1x_{2}x_{3} \ldots) & = 1d(x_{2}\ldots); \\
d(0x_{2}x_{3} \ldots) & = 0x_{2}x_{3}\ldots; \\
d(1x_{2}x_{3} \ldots) & = 1b(x_{2}x_{3}\ldots).
\end{align*}
The transformations $a$,$b$,$c$,$d$ are the generators of Grigorchuk's first group $G$; each determines a homeomorphism of the Cantor set $X$ \cite{Grigorchuk}.
We define the transformations $\sigma_{\omega_{1},\omega_{2}}$ as from Example \ref{example:V}. We further define, for each $g \in G$,
\[ g_{\omega_{1},\omega_{2}} = \sigma_{\epsilon,\omega_{2}}g\sigma_{\omega_{1},\epsilon}. \]
Let 
\[ S_{R} = \{ g_{\omega_{1},\omega_{2}} \mid \omega_{1},\omega_{2} \text{ are finite binary strings} \} \cup \{ 0 \}. \]
The set $S_{R}$ is an inverse semigroup. The proof uses the self-similarity property of $G$ in an essential way. (A reference for the self-similarity property is \cite{Grigorchuk}.)  The set of domains is
\[ \mathcal{D}^{+}_{R} = \{ B_{\omega} \mid \omega \text{ is a finite binary string} \}. \] The set $\mathcal{D}^{+}_{R}$ satisfies the compact ultrametric property.
The group $\Gamma_{S_{R}}$ is the R\"{o}ver group \cite{Rover}. 

Nekrashevych \cite{Nek} subsequently studied a more general class of groups, which are now often called the \emph{Nekrashevych-R\"{o}ver} groups. In this paper, we will concentrate on the group originally considered by R\"{o}ver, although the range of the applicability of our methods to more general Nekrashevych-R\"{o}ver groups remains an interesting question. 
\end{example} 

\begin{example} \label{example:Houghton} \cite{Brown, Houghton} (Houghton's groups) Let
$X_{n} = \{1, \ldots, n \} \times \mathbb{N}$. We define two types of basic transformations.
\begin{enumerate}
\item If $(j_{1},k_{1}), (j_{2},k_{2}) \in X$, then let $\alpha_{(j_{1},k_{1})}^{(j_{2},k_{2})}$ be the unique function with domain $\{ (j_{1}, k_{1}) \}$ and codomain $\{ (j_{2}, k_{2}) \}$. 
\item If $(j,k) \in X_{n}$  and $\ell \in \mathbb{Z}$ is such that $k+\ell > 0$, then define
$\beta_{(j,k)}^{\ell}: \{ j \} \times \{ k, k+1, \ldots \} \rightarrow \{ j \} \times \{ k+\ell, k+\ell+1, \ldots \}$ by the rule
$\beta_{(j,k)}^{\ell}(j,x) = (j,x+\ell)$.
\end{enumerate}

The set 
$S_{H_{n}}$ of all the $\alpha_{(j_{1},k_{1})}^{(j_{2},k_{2})}$, all of  the $\beta_{(j,k)}^{\ell}$, and the empty function $0$
is an inverse semigroup.

The corresponding set of domains is 
\[ \mathcal{D}^{+}_{S_{H_{n}}} = \{ \{ (j,k) \} \mid (j,k) \in X \} \cup 
\{ \{j \} \times \{ k, k+1, \ldots \} \mid (j,k) \in X \}, \]
which satisfies the compact ultrametric property.
The group $\Gamma_{S_{H_{n}}}$ is Houghton's group $H_{n}$.
\end{example}

\begin{example} \label{example:bieri}
(Bieri-Sach examples) Consider the inverse semigroup $S_{H_{2}}$ from Example \ref{example:Houghton} and the associated set $X_{2} = \{ 1, 2 \} \times \mathbb{N}$. We identify $X_{2}$ with the set of integers, matching $(1,k)$ with $k-1$ and $(2,k)$ with $-k$. This makes $S_{H_{2}}$ an inverse semigroup of partial bijections of $\mathbb{Z}$.

It will be convenient to introduce notation for the domains of 
$S_{H_{2}}$. Define
\begin{align*}
R^{+}_{k} &= \{ k, k+1, \ldots \}, \text{ for }k \geq 0; \\
R^{-}_{k} &= \{ \ldots, k-1, k \}, \text{ for }k < 0; \\
P_{k} &= \{ k \}, \text{ for }k \in \mathbb{Z}.
\end{align*}

 There are three types of transformations in $S_{H_{2}}$:
\begin{enumerate}
\item bijections between singleton sets $\tau_{k,\ell}: P_{k} \rightarrow P_{\ell}$;
\item shifts $S^{+}_{k \rightarrow \ell}: R^{+}_{k} \rightarrow R^{+}_{\ell}$, where $k, \ell \geq 0$;
\item shifts $S^{-}_{k \rightarrow \ell}: R^{-}_{k} \rightarrow R^{-}_{\ell}$, where $k,\ell <0$.
\end{enumerate}
All of these are restrictions of suitable translations, and are bijections between the given domains and codomains.

We consider the product action of $nS_{H_{2}}$ on $\mathbb{Z}^{n}$.
A domain from a transformation $s \in nS_{H_{2}}$ is a polyhedral subset of $\mathbb{Z}^{n}$. For instance, letting $n=2$, we find nine domain types in $\mathcal{D}^{+}_{2S_{H_{2}}}$:  

\begin{enumerate}
\item quadrants: $R^{+}_{k} \times R^{+}_{\ell}$, $R^{+}_{k} \times R^{-}_{\ell}$,
$R^{-}_{k} \times R^{-}_{\ell}$, $R^{-}_{k} \times R^{+}_{\ell}$;
\item singletons: $P_{k} \times P_{\ell}$;
\item strips:  $R^{+}_{k} \times P_{\ell}$, $R^{-}_{k} \times P_{\ell}$, $P_{k} \times R^{+}_{\ell}$, $P_{k} \times R^{-}_{\ell}$.
\end{enumerate}

The set $\mathcal{D}^{+}_{nS_{H_{2}}}$ does not satisfy the compact ultrametric property when $n \geq 2$.
The associated groups $\Gamma_{nS_{H_{2}}}$ are examples from the class of groups considered in \cite{Bieri}.
\end{example}

\section{$S$-structures and pseudovertices} \label{section:Sstructures}

In this section, we will produce a partially ordered set upon which the group $\Gamma_{S}$ acts. The set and partial order depend on the choice of $S$-structure (Definition \ref{definition:sstructure}). Subsection \ref{subsection:sstructures} defines $S$-structures. Subsection \ref{subsection:po} contains basic information about the expansion relation, which determines the partial order on pseudovertices. Subsection \ref{subsection:theaction} defines the action of $\Gamma_{S}$ on various sets of pseudovertices. Subsection \ref{subsection:directed} shows that the pseudovertices form a directed set (under relevant hypotheses). The section concludes with Subsection \ref{subsection:sexamples}, which describes the expansion relation in a number of examples.

\subsection{$S$-structures} \label{subsection:sstructures}
 In this subsection, we define $S$-structures, which will be used to generate a partially ordered set upon which the group $\Gamma_{S}$ acts. The $S$-structures are therefore essential in everything that follows.

\begin{definition} (meet and restriction of partitions) \label{definition:meetrestrict}
Let $\mathcal{P}_{1}$ and $\mathcal{P}_{2}$ be partitions of a set $Z$. The \emph{meet} of $\mathcal{P}_{1}$ and $\mathcal{P}_{2}$, denoted
$\mathcal{P}_{1} \wedge \mathcal{P}_{2}$, is defined by the equation
\[ \mathcal{P}_{1} \wedge \mathcal{P}_{2} = \{ P_{1} \cap P_{2} \mid P_{1} \in \mathcal{P}_{1}, P_{2} \in \mathcal{P}_{2}, P_{1} \cap P_{2} \neq \emptyset \}. \]
In words, the meet is the coarsest common refinement of $\mathcal{P}_{1}$ and $\mathcal{P}_{2}$. It is a partition of $Z$.

If $\mathcal{P}$ is a partition of $Z$ and $Y \subseteq Z$, then 
\[ \mathcal{P}_{\mid Y} = \{ P \cap Y \mid P \in \mathcal{P}; P \cap Y \neq \emptyset \} \]
is the  \emph{restriction of $\mathcal{P}$ to $Y$}; it is a partition of $Y$.
\end{definition}

\begin{definition} ($S$-structure) \label{definition:sstructure}
Let $2^{S}$ denote the power set of $S$. 

An \emph{$S$-structure} is a pair $(\mathbb{S}, \mathbb{P})$, where  
 $\mathbb{S}: \mathcal{D}^{+}_{S} \times \mathcal{D}^{+}_{S} \rightarrow 2^{S}$ and $\mathbb{P}$ assigns to each domain $D$ a collection 
 $\mathbb{P}(D)$ of partitions of $D$. The functions $\mathbb{S}$ and $\mathbb{P}$ must satisfy the following properties:  
 
\begin{enumerate}
\item[(P1)] $\{ D \} \in \mathbb{P}(D)$, for each $D \in \mathcal{D}^{+}_{S}$;
\item[(P2)] each $\mathcal{P} \in \mathbb{P}(D)$ is a finite partition of $D$ into domains;

\item[(P3)] (restriction) If $\mathcal{P} \in \mathbb{P}(D)$ and $E$ is a non-empty domain contained in $D$, then $\mathcal{P}_{\mid E} \in \mathbb{P}(E)$.
\item[(P4)] (patchwork) if $\mathcal{P} = \{ D_{1}, \ldots, D_{m} \} \in \mathbb{P}(D)$ and, for $i=1, \ldots, m$, $\mathcal{P}_{i}\in \mathbb{P}(D_{i})$, then
$\mathcal{P}_{1} \cup \ldots \cup \mathcal{P}_{m} \in \mathbb{P}(D)$.
\item[(P5)] (cofinality) for each $D \in \mathcal{D}^{+}$ and each finite partition of $D$ into domains $\mathcal{P}'$, there is some $\mathcal{P} \in \mathbb{P}(D)$ refining $\mathcal{P}'$.
\item[(S1)] if $s \in \mathbb{S}(D_{1}, D_{2})$, then the domain of $s$ is $D_{1}$ and the image of $s$ is $D_{2}$;
\item[(S2)] (identities) for each $D \in \mathcal{D}^{+}$, $id_{D} \in \mathbb{S}(D, D)$;
\item[(S3)] (inverses) if $s \in \mathbb{S}(D_{1}, D_{2})$, then $s^{-1} \in \mathbb{S}(D_{2}, D_{1})$;
\item[(S4)] (compositions) if $s_{1} \in \mathbb{S}(D_{1}, D_{2})$ and $s_{2} \in \mathbb{S}(D_{2}, D_{3})$, then
$s_{2}s_{1} \in \mathbb{S}(D_{1}, D_{3})$.
\item[(S5)] (generation) if $D \in \mathcal{D}^{+}$ and $f: D \rightarrow X$ is in $\widehat{S}$ (see Definition \ref{definition:hatS}), then there is $\mathcal{P} \in \mathbb{P}(D)$
such that, for each $E \in \mathcal{P}$, there is $s \in S$ such that $E$ is the domain of $s$, $f_{\mid E}= s$, and $s \in \mathbb{S}(E, s(E))$.
\item[(S6)] ($\mathbb{S}$-invariance of $\mathbb{P}$) if $\mathcal{P}_{1} \in \mathbb{P}(D_{1})$ and $s \in \mathbb{S}(D_{1},D_{2})$, then 
$s(\mathcal{P}_{1}) \in \mathbb{P}(D_{2})$.
\end{enumerate}
The function $\mathbb{S}$ is the \emph{structure function} and $\mathbb{P}$ is the \emph{pattern function}. The sets $\mathbb{S}(D_{1},D_{2})$ are \emph{structure sets}. A member of $\mathbb{P}(D)$ is called a \emph{pattern of $D$}.
\end{definition}

\begin{remark} \label{remark:sstructuredescription} Thus, each pattern in $\mathbb{P}(D)$ must be a finite partition of $D$ into domains, and the patterns in question must contain the trivial partition and be closed under restriction to subdomains. Additionally, every finite partition of $D$ into domains must have a refinement in $\mathbb{P}(D)$ (cofinality). The patchwork condition ensures that if each piece of a pattern is itself replaced by a pattern, then the result is a pattern. 

The structure set $\mathbb{S}(D_{1},D_{2})$ consists of transformations from $S$ having domain $D_{1}$ and image $D_{2}$; the structure sets collectively satisfy ``groupoid-like'' properties (S1-S4). In particular, $\mathbb{S}(D,D)$ is a group, for each $D \in \mathcal{D}^{+}_{S}$. There is a straightforward compatibility requirement (S6). 

The generation property (S5) helps ensure that both $\mathbb{S}$ and $\mathbb{P}$ are ``sufficiently rich'', in an appropriate sense. For each embedding $f: D \rightarrow X$ that is locally determined by $S$, $\mathbb{P}(D)$ is large enough that we can find a pattern of $D$, such that the restriction of $f$ to each member of the pattern is an element of $S$; moreover, $\mathbb{S}$ is sufficiently rich in the sense that each such restriction is a member of the appropriate structure set. 
\end{remark}

\begin{example} \label{example:maximalSstructure} (the maximal $S$-structure)
Let $S$ be an inverse semigroup acting on $X$. For all $D_{1}, D_{2} \in \mathcal{D}^{+}_{S}$, we let 
\[ \mathbb{S}(D_{1},D_{2}) = \{ s \in S \mid dom(s) = D_{1}; im(s) = D_{2} \}, \]
and
\[ \mathbb{P}(D_{1}) = \{ \mathcal{P} \mid \mathcal{P} \text{ is a finite partition of }D_{1}\text{ into domains} \}. \]
If $\mathbb{S}$ and $\mathbb{P}$ are as above, then we refer to $(\mathbb{S},\mathbb{P})$ as the \emph{maximal $S$-structure}. It is easy to check both that $(\mathbb{S},\mathbb{P})$ is an $S$-structure and that it is the largest possible $S$-structure.

The maximal $S$-structure will be used in many of our applications. It has many advantages: it is simple and, very often, it is the only natural $S$-structure. We call attention to two potential disadvantages: 
\begin{enumerate}
\item The groups $\mathbb{S}(D,D)$ will contribute to the size of cell stabilizers (see Proposition \ref{proposition:stab}). In certain cases (see, for instance, Example \ref{example:Roever}), the maximal structure function will yield infinite groups $\mathbb{S}(D,D)$, and, therefore, infinite stabilizer groups. Even worse, these infinite groups may have bad finiteness properties, which will make it impossible to deduce useful finiteness properties for the larger group $\Gamma_{S}$.
\item When we consider product actions, the maximal pattern function would allow partitions like those from Example \ref{example:pathology}. Such examples make it difficult or impossible to establish finiteness properties of the groups $\Gamma_{S}$.
It will therefore be desirable to restrict the possible allowed partitions by specifying smaller pattern sets $\mathbb{P}(D)$.  
\end{enumerate}
\end{example}

\begin{example} \label{example:Brin} (Brin's patterns) 
Let $X$ denote the set of infinite binary strings, and carry over all other notation from Examples \ref{example:clarifying}, \ref{example:V} and \ref{example:nV}. Brin \cite{BrinHighD} specified a certain family of partitions of $X^{2}$, which he called ``patterns''. The \emph{patterns in $X^{2}$} are defined inductively as follows. The trivial partition $\{ X^{2} \}$ is a pattern. If $\mathcal{P}$ is a pattern and $B_{\alpha} \times B_{\beta} \in \mathcal{P}$, then
\[ (\mathcal{P} - \{ B_{\alpha} \times B_{\beta} \}) \cup \{ B_{\alpha0} \times B_{\beta}, B_{\alpha1} \times B_{\beta} \} \]
and
 \[ (\mathcal{P} - \{ B_{\alpha} \times B_{\beta} \}) \cup \{ B_{\alpha} \times B_{\beta0}, B_{\alpha} \times B_{\beta1} \} \]
are also patterns. 

Thus, in words, every pattern is a partition of $X^{2}$ into dyadic rectangles. The trivial partition is a pattern, and a new pattern can be obtained from another by dividing a given rectangle in half, either in the first or second coordinate. 

Brin similarly defines patterns in $X^{n}$, for all $n \geq 2$. It appears to the authors that patterns in $X^{2}$ are simply the partitions of $X^{2}$ into finitely many dyadic rectangles (or into finitely many domains, where the implied semigroup is $2S_{V}$). The patterns are, however, a special type of partition into finitely many domains when $n \geq 3$. For instance, the partitions $\mathcal{P}_{n}$ from Example \ref{example:pathology} are clearly not patterns. 

The definition of pattern extends easily to general products (see Definition \ref{definition:structuresonproducts}). The resulting definition rules out pathologies like the one from Example \ref{example:pathology} and proves to be useful in applications to finiteness properties.
\end{example}

\begin{remark} \label{remark:pip} (patterns in practice)
We are aware of two natural choices for the pattern function $\mathbb{P}$: (1) the unrestricted pattern function, which simply says that every partition of a domain $D$ into finitely many domains is a pattern (as in Example \ref{example:maximalSstructure}); (2) the pattern function that mimics the patterns from Example \ref{example:Brin} (see Example \ref{definition:structuresonproducts}).
We will use no other examples in this paper. 
\end{remark}

\begin{definition} \label{definition:domaintypes} (domain types) Let $D_{1}$, $D_{2} \in \mathcal{D}^{+}$. We say that $D_{1}$ and $D_{2}$ have the \emph{same domain type relative to $\mathbb{S}$} if $\mathbb{S}(D_{1},D_{2}) \neq \emptyset$. We will frequently omit mention of the structure function $\mathbb{S}$, and say that $D_{1}$ and $D_{2}$ have the same domain type, if $\mathbb{S}$ is understood.

We note that having the same domain type is an equivalence relation on $\mathcal{D}^{+}$. We let $[D]$ denote the equivalence class of $D$ under this relation.

We say that $\mathbb{S}$ has \emph{finitely many domain types} if the number of equivalence classes of domains is finite.
\end{definition}

\begin{definition} \label{definition:popattern} (a partial order on partitions)
Let $Y \subseteq X$ be a finite non-empty union of disjoint domains. Let $\mathcal{P}_{1} = \{ D_{1}, \ldots, D_{m} \}$ and $\mathcal{P}_{2}$ be finite partitions of $Y$ into domains.
We write $\mathcal{P}_{1} \preccurlyeq \mathcal{P}_{2}$ if 
\[ \mathcal{P}_{2} = \bigcup_{i=1}^{m} \widehat{\mathcal{P}}_{i}, \]
for some $\widehat{\mathcal{P}}_{i} \in \mathbb{P}(D_{i})$ ($i=1, \ldots, m$).

It is straightforward to check that $\preccurlyeq$ is a partial order on the set of partitions of $Y$ into finitely many domains. We note that the property of transitivity relies on the patchwork property (P4) from Definition \ref{definition:sstructure}
\end{definition}

\begin{proposition} \label{proposition:dspatterns} (A directed set property) Let $Y \subseteq X$ be a finite non-empty union of disjoint domains. Let 
$\mathcal{P}_{1}$ and $\mathcal{P}_{2}$ be finite partitions of $Y$ into domains.

There is $\mathcal{P}_{3}$ such that $\mathcal{P}_{1} \preccurlyeq \mathcal{P}_{3}$ and $\mathcal{P}_{2} \preccurlyeq \mathcal{P}_{3}$.
\end{proposition}

\begin{proof}
 First, we will show that there is a partition $\mathcal{P}'$ of $Y$ into finitely many domains such that $\mathcal{P}_{1} \preccurlyeq \mathcal{P}'$ and
$\mathcal{P}_{1} \wedge \mathcal{P}_{2} \preccurlyeq \mathcal{P}'$.

The meet $\mathcal{P}_{1} \wedge \mathcal{P}_{2}$ (see Definition \ref{definition:meetrestrict}) 
is a common refinement of both $\mathcal{P}_{1}$ and $\mathcal{P}_{2}$, and each member of the meet is a domain by Lemma \ref{lemma:domainproperties}(1). 
Assume that $\mathcal{P}_{1} = \{ D_{1}, \ldots, D_{m} \}$. For each $D_{i} \in \mathcal{P}_{1}$, there is a subset of $\mathcal{P}_{1} \wedge \mathcal{P}_{2}$ that partitions $D_{i}$. By the cofinality property (P5) from Definition \ref{definition:sstructure}, we can find a partition $\widehat{\mathcal{P}}_{i} \in \mathbb{P}(D_{i})$ that refines the latter partition. Letting
\[ \mathcal{P}' = \bigcup_{i=1}^{m} \widehat{\mathcal{P}}_{i}, \]
we find that $\mathcal{P}_{1} \preccurlyeq \mathcal{P}'$. 

We claim that, moreover, $\mathcal{P}_{1} \wedge \mathcal{P}_{2} \preccurlyeq \mathcal{P}'$. Indeed, let $\mathcal{P}_{1} \wedge \mathcal{P}_{2} = \{ E_{1}, \ldots, E_{n} \}$. For $i=1, \ldots, n$, there is some unique $\sigma(i) \in \{ 1, \ldots, m \}$ such that $E_{i} \subseteq D_{\sigma(i)}$. We have the equality
\[ \mathcal{P}'_{\mid E_{i}} = \left(\mathcal{P}'_{\mid D_{\sigma(i)}}\right)_{\mid E_{i}}, \]
by definition of the restriction (see Definition \ref{definition:meetrestrict}).  Since
\[ \mathcal{P}'_{\mid D_{\sigma(i)}} = \widehat{\mathcal{P}}_{\sigma(i)}, \]
and the latter is in $\mathbb{P}(D_{\sigma(i)})$, property (P3) implies that $\mathcal{P}'_{\mid E_{i}} \in \mathbb{P}(E_{i})$. The equality
\[ \mathcal{P}' = \bigcup_{i=1}^{n} \mathcal{P}'_{\mid E_{i}} \]
now implies that $\mathcal{P}_{1} \wedge \mathcal{P}_{2} \preccurlyeq \mathcal{P}'$, as desired.

Next, we simply apply the preceding argument with $\mathcal{P}_{2}$ playing the role of $\mathcal{P}_{1}$, and $\mathcal{P}'$ playing the role of $\mathcal{P}_{2}$. We conclude that there is $\mathcal{P}''$ such that $\mathcal{P}_{2} \preccurlyeq \mathcal{P}''$ and $\mathcal{P}_{2} \wedge \mathcal{P}' = \mathcal{P}' \preccurlyeq \mathcal{P}''$. Set $\mathcal{P}_{3} = \mathcal{P}''$.
\end{proof}

\subsection{The partially ordered set of pseudovertices} \label{subsection:po}

For the rest of this section, we fix an $S$-structure $(\mathbb{S},\mathbb{P})$.

\begin{definition}  \label{definition:pairs} (the fundamental equivalence relation)  Let  
\[ \mathcal{A} = \{ (f,D) \mid D \text{ is a domain and }f:D \rightarrow X \text{ is locally determined by }S \}. \]
We write 
\[ (f_{1}, D_{1}) \sim (f_{2}, D_{2}) \]
if there is some $h \in \mathbb{S}(D_{1}, D_{2})$ such that $f_{1} = f_{2} \circ h$. We note for future reference that 
$(f_{1}, D_{1}) \sim (f_{2}, D_{2})$ implies $f_{1}(D_{1}) = f_{2}(D_{2})$. 

The relation $\sim$ is an equivalence relation
on $\mathcal{A}$. We denote the equivalence class of $(f,D)$ by $[f,D]$, and the set of all such equivalence classes by $\mathcal{B}$.
\end{definition}

\begin{remark} \label{remark:convenient}
It will frequently be convenient to write $[f,D]$ instead of $[f_{\mid D}, D]$ when the domain $f$ is larger than $D$. We will freely do so in what follows, for the sake of simplicity in notation.
\end{remark}
 
\begin{definition} \label{definition:pseudovertex} (pseudovertex) 
A non-empty subset
\[ v = \{ [f_{1}, D_{1}], \ldots, [f_{m}, D_{m}] \} \]
of $\mathcal{B}$ is called a \emph{pseudovertex} if the images $f_{1}(D_1)$, $\ldots$, $f_{m}(D_{m})$ are pairwise disjoint. The \emph{image of $v$},  denoted $im(v)$, is the set
\[ f_{1}(D_{1}) \cup \ldots \cup f_{m}(D_{m}). \]
We let $\mathcal{PV}_{\mathbb{S}}$ denote the set of all pseudovertices. If $Y \subseteq X$ is expressible as a finite disjoint union of domains, then we let $\mathcal{PV}_{\mathbb{S},Y}$ denote the set of pseudovertices $v$ satisfying $im(v) = Y$. 

If $Y=X$, we may write $\mathcal{V}_{\mathbb{S}}$ instead of $\mathcal{PV}_{\mathbb{S}, X}$. 
\end{definition}

\begin{remark} \label{remark:vertices}
The set $\mathcal{V}_{\mathbb{S}}$ will ultimately be the vertices of the complexes of interest to us.
\end{remark}

\begin{remark} \label{remark:pedantic} We note that the definition of
$\mathcal{PV}_{\mathbb{S},Y}$ depends upon the structure sets $\mathbb{S}(D_{1},D_{2})$, not the pattern function $\mathbb{P}$, which justifies omitting $\mathbb{P}$ from the notation. However, $\mathbb{P}$ will affect the partial order  (Definition \ref{definition:expansion}). 
\end{remark}

\begin{definition} \label{definition:vt} (the type of a pseudovertex)
We say that two pseudovertices 
\[ v_{1} = \{ [f_{1},D_{1}], \ldots, [f_{m},D_{m}] \}
\text{ and } v_{2} = \{ [g_{1},E_{1}], \ldots, [g_{n}, E_{n}] \} \]
have the same \emph{type} if $m=n$ and $[D_{i}] = [E_{i}]$ for $i = 1, \ldots, m$ (possibly after reordering). Recall that $[D]$ denotes the domain type of $D$ (Definition \ref{definition:domaintypes}).
\end{definition}

\begin{remark}
It is straightforward to check that the type of a pseudovertex is well-defined; i.e., if
\[ \{ [f_{1},D_{1}], \ldots, [f_{m},D_{m}] \} = v =
\{ [g_{1},E_{1}], \ldots, [g_{n}, E_{n}] \}, \]
then $m=n$ and $[D_{i}] = [E_{i}]$ for all $i$, possibly after rearrangement.
\end{remark}

\begin{definition} \label{definition:vertexrank} (rank of a pseudovertex)
Let $v$ be a pseudovertex. Since $v \subseteq \mathcal{B}$, $|v|$ has its usual definition (as the cardinality of a set). We call $|v|$ the \emph{rank} of the pseudovertex $v$. 
\end{definition}

\begin{definition} \label{definition:expansion} (expansion) 
Let 
\[ v = \{ [f_{1}, D_{1}] , \ldots, [f_{m}, D_{m}] \} \]
be a pseudovertex. Let $[f_{i}, D_{i}] \in v$. Let $\widehat{D} \in \mathcal{D}^{+}$ have the same domain type as $D_{i}$, let $h \in \mathbb{S}(\widehat{D},D_{i})$, and let $\mathcal{P} \in \mathbb{P}(\widehat{D})$ be non-trivial (i.e., $|\mathcal{P}| \geq 2$).  
We say that
\[ v' = \left(v - \{ [f_{i}, D_{i}] \}\right) \cup \{ [f_{i}h, D] \mid D \in \mathcal{P} \} \]
is an \emph{expansion from $v$ at $[f_{i}, D_{i}]$} (or simply an \emph{expansion}). We write $v \nearrow v'$.  

If there is a (possibly empty) sequence of
expansions
\[ v = v_{0} \nearrow v_{1} \nearrow v_{2} \ldots \nearrow v_{n} = v', \]
then we write $v \leq v'$.
\end{definition}

\begin{remark} \label{remark:futurerefrank}
We note that expansion necessarily increases the rank of a pseudovertex.
\end{remark}

\begin{proposition} (The partial order induced by expansion) \label{proposition:expansion}
If $v$ is a pseudovertex and $v \nearrow v'$, then $v'$ is a pseudovertex having the same image.  Moreover:
\begin{enumerate}
\item The expansion relation $v \nearrow v'$ is well-defined; that is, if 
$v$ is a pseudovertex, $v = \hat{v}$, and $v \nearrow v'$, then 
$\hat{v} \nearrow v'$;
\item The relation $\leq$ is a partial order on $\mathcal{PV}_{\mathbb{S}}$. 
\end{enumerate}
\end{proposition}

\begin{proof}
If $v \nearrow v'$ and $v$ is a pseudovertex, then it is clear from the definition of the expansion relation (Definition \ref{definition:expansion}) that $v'$ is also a pseudovertex, and that $im(v) = im(v')$. 

We prove (1). Suppose that $v = \hat{v}$, where $v$ is a pseudovertex. Let $v = \{ [f_{1}, D_{1}], \ldots, [f_{n}, D_{n}] \}$ and
$\hat{v} = \{ [g_{1}, E_{1}], \ldots, [g_{n}, E_{n}] \}$, where $[f_{i}, D_{i}] = [g_{i}, E_{i}]$ for $i = 1, \ldots, n$. We assume that
$v \nearrow v'$ by an expansion at $[f_{1}, D_{1}]$. Thus, there is a domain $\widehat{D}$, an $h \in \mathbb{S}(\widehat{D}, D_{1})$ and a non-trivial $\mathcal{P}\in \mathbb{P}(\widehat{D})$ such that
\[ v' = \left( v - \{ [f_{1}, D_{1}] \} \right) \cup \{ [f_{1}h, D] \mid D \in \mathcal{P} \}. \]
Since $[f_{1}, D_{1}] = [g_{1}, E_{1}]$,
there is $h' \in \mathbb{S}(D_{1}, E_{1})$ such that $g_{1}h' = f_{1}$. It follows directly that $g_{1}h'h = f_{1}h$. Note that
$h'h \in \mathbb{S}(\widehat{D}, E_{1})$ by the compositions property of $S$-structures (property (S4) from Definition \ref{definition:sstructure}).  It follows that $\hat{v} \nearrow \hat{v}'$, where
\[
\hat{v}' 
= \left( \hat{v} - \{ [g_{1}, E_{1}] \} \right) \cup \{ [ g_{1}h'h, D] \mid D \in \mathcal{P} \}. \]
The equality $\hat{v}' = v'$ now follows from the assumption that $[f_{i},D_{i}] = [g_{i},E_{i}]$ for $i = 1, \ldots, n$, and from the fact that  $g_{1}h'h = f_{1}h$. It follows that $\hat{v} \nearrow v'$, as desired.

To prove (2) we first note that Definition \ref{definition:expansion} implies that $\leq$ is both reflexive and transitive. 
The fact that $\leq$ is also antisymmetric follows easily from the fact that expansion increases the cardinality of a pseudovertex.

\end{proof}

\begin{remark} \label{remark:refinement}
Given a pseudovertex $v = \{ [f_{1},D_{1}], \ldots, [f_{m},D_{m}] \}$, the set $\mathcal{P}_{v} = \{ f_{i}(D_{i}) \mid i \in \{ 1, \ldots, m \} \}$ is a partition of $im(v)$. 

We note for future reference that, when $v_{1} < v_{2}$, $\mathcal{P}_{v_{2}}$ is a proper refinement of $\mathcal{P}_{v_{1}}$. 
\end{remark}

\begin{proposition} \label{proposition:orderlocal} 
(the partial order is locally determined) Let $v_{1}$ and $v_{2}$ be pseudovertices and assume that $v_{1} \leq v_{2}$.
\begin{enumerate}
\item  For every pseudovertex $v'_{1} \subseteq v_{1}$, there is a unique pseudovertex $v'_{2} \subseteq v_{2}$ such that $im(v'_{1}) = im(v'_{2})$;
\item If $v'_{1}$ and $v'_{2}$ satisfy $v'_{i} \subseteq v_{i}$ ($i=1,2$) and $im(v'_{1}) = im(v'_{2})$, then $v'_{1} \leq v'_{2}$.
\end{enumerate}
\end{proposition}

\begin{proof}
Let $v_{1}$ and $v_{2}$ be pseudovertices such that $v_{1} \leq v_{2}$.

We prove (1). Let $v'_{1} \subseteq v_{1}$. We will first prove the existence of a suitable $v'_{2}$ and later consider uniqueness.
 By an easy induction, it suffices to consider the case in which   $v_{1} \nearrow v_{2}$. If the expansion in question occurs at a pair $[f,D] \notin v'_{1}$, then $v'_{1} \subseteq v_{2}$ and we can set $v'_{2} = v'_{1}$. If the expansion occurs at some $[f,D] \in v'_{1}$, then, by Definition \ref{definition:pseudovertex},
\[ v_{2} = (v_{1} - \{ [f,D] \}) \cup \{ [fh, E] \mid E \in \mathcal{P} \}, \]
where $\widehat{D} \in \mathcal{D}^{+}$, $h \in \mathbb{S}(\widehat{D}, D)$, and $\mathcal{P} \in \mathbb{P}(\widehat{D})$ is non-trivial. Since $\{ [f, D] \}$
and $\{ [fh, E] \mid E \in \mathcal{P} \}$ are pseudovertices with the same image, we can let 
\[ v'_{2} = (v'_{1} - \{ [f,D] \}) \cup \{ [fh, E] \mid E \in \mathcal{P} \}, \]
and $im(v'_{1}) = im(v'_{2})$. This demonstrates the existence of a pseudovertex $v'_{2} \subseteq v_{2}$ such that $im(v'_{2}) = im(v'_{1})$. 

The uniqueness of $v'_{2}$ is straightforward: in any pseudovertex 
\[ v = \{ [f_{1}, D_{1}], \ldots, [f_{m}, D_{m}] \} \]      
the images $f_{i}(D_{i})$ are non-empty and pairwise disjoint (see Definition \ref{definition:pseudovertex}), so there is at most one subset of $v$ having any given image. Uniqueness follows directly; this proves (1).

We now prove (2). Assume that $v'_{1} \subseteq v_{1}$, $v'_{2} \subseteq v_{2}$, and $im(v'_{1}) = im(v'_{2})$. We will assume further that $v_{1} \nearrow v_{2}$, since the general case follows from this one by an easy induction. Assume that the expansion in question occurs at $[f,D] \in v_{1}$. 

The proof of (1) shows that either $v'_{2} = v'_{1}$ or   
\[ v'_{2} = (v'_{1} - \{ [f,D] \}) \cup \{ [fh, E] \mid E \in \mathcal{P} \}, \]
where $\mathcal{P}$ and $h$ are as described above. In either case $v'_{1} \leq v'_{2}$. This proves (2).
\end{proof}

\subsection{The action on pseudovertices} \label{subsection:theaction}

Recall that $\widehat{S}$ denotes the collection of all partial bijections of $X$ that are locally determined by $S$ (see Definition \ref{definition:hatS}).

\begin{definition} \label{definition:hatSaction} 
(the partial action of $\widehat{S}$ on pseudovertices)  Let $\hat{s} \in \widehat{S}$, and let
$[f,D] \in \mathcal{B}$. If $f(D)$ is contained in the domain of $\hat{s}$, then we define
\[ \hat{s} \cdot [f,D] = [\hat{s}f,D]. \]
More generally, if
\[ v = \{ [f_{1}, D_{1}], \ldots, [f_{m}, D_{m}] \} \]
is a pseudovertex such that $im(v)$ is contained in the domain of $\hat{s}$, then 
we define
\[ \hat{s} \cdot v = \{ [ \hat{s} f_{1}, D_{1}], \ldots, [ \hat{s} f_{m}, D_{m}] \}. \]
\end{definition}

\begin{proposition} \label{proposition:hatSset} 
(the partial action on $\mathcal{PV}_{\mathbb{S}}$)
Let $v_{1}, v_{2} \in \mathcal{PV}_{\mathbb{S}}$ and suppose $v_{1} \leq v_{2}$. Assume that  $\hat{s} \in \widehat{S}$ and that the domain of $\hat{s}$ contains $im(v_{1}) = im(v_{2})$.
\begin{enumerate}
\item the action of $\hat{s}$ is well-defined on $\mathcal{B}$ and on $\mathcal{PV}$; i.e., if $v_{1} = \tilde{v}_{1} \in \mathcal{PV}$, then $\hat{s} \cdot v_{1} = \hat{s} \cdot \tilde{v}_{1}$;
\item the action of $\hat{s}$ is order-preserving:
$\hat{s} \cdot v_{1} \leq \hat{s} \cdot v_{2}$.
\item $im(\hat{s} \cdot v_{1}) = \hat{s}(im(v_{1}))$.
\end{enumerate}
\end{proposition}

\begin{proof}
We prove (1); the statement about $\mathcal{B}$ will be proved in the course of establishing the well-definedness of the action on $\mathcal{PV}$. Let 
\[ v_{1} = \{ [f_{1}, D_{1}], \ldots, [f_{m}, D_{m}] \} \quad \text{and} \quad \tilde{v}_{1} = \{ [g_{1}, E_{1}], \ldots, [g_{m}, E_{m}] \}, \]
where $[f_{i}, D_{i}] = [g_{i}, E_{i}]$, for $i=1, \ldots, m$. 

We choose a subscript $i$. Since $[f_{i},D_{i}] = [g_{i},E_{i}]$, there is $h \in \mathbb{S}(D_{i},E_{i})$ such that $g_{i}h = f_{i}$. Thus,
$\hat{s} g_{i} h = \hat{s} f_{i}$, which proves that $[\hat{s} f_{i}, D_{i}] = [\hat{s} g_{i}, E_{i}]$. 
Since $i$ was arbitrary, we find that $\hat{s} \cdot v_{1} = \hat{s} \cdot \tilde{v}_{1}$, as required.

We next prove (2). Assume that $v_{1} \leq v_{2}$.
It suffices to show that $\hat{s} \cdot v_{1} \leq \hat{s} \cdot v_{2}$ in the special case 
where $v_{1} \nearrow v_{2}$.

Let 
\[ v_{1} = \{ [f_{1}, D_{1}], \ldots, [f_{m}, D_{m}] \} \]
and suppose that $v_{2}$ is an expansion from $v_{1}$ at $[f_{1}, D_{1}]$. Thus there is $\widehat{D} \in \mathcal{D}$,
$h \in \mathbb{S}(\widehat{D}, D_{1})$, and a non-trivial $\mathcal{P} \in \mathbb{P}(\widehat{D})$ such that
\[ v_{2} = \{ [f_{2}, D_{2}], \ldots, [f_{m}, D_{m}] \} \cup \{ [ f_{1}h, E] \mid E \in \mathcal{P} \}. \]

We note that
\[ \hat{s} \cdot v_{1} = \{ [\hat{s} f_{1}, D_{1}], \ldots, [\hat{s} f_{m}, D_{m}] \} \]
and
\[ \hat{s} \cdot v_{2} = \{ [\hat{s} f_{2}, D_{2}], \ldots, [\hat{s} f_{m}, D_{m}] \} \cup \{ [\hat{s} f_{1}h, E] \mid E \in \mathcal{P} \}. \]
It follows that $\hat{s} \cdot v_{2}$ is an expansion from $\hat{s} \cdot v_{1}$ at $[\hat{s} f_{1}, D_{1}]$ via the 
same $\widehat{D}$, $h$, and $\mathcal{P}$ as above. Thus, $\hat{s} \cdot v_{1} \nearrow \hat{s} \cdot v_{2}$.

The final statement follows directly from the definition of the action 
(Definition \ref{definition:hatSaction}).
\end{proof}

\subsection{Directed sets of pseudovertices} \label{subsection:directed}

\begin{lemma} \label{lemma:expansionintocone} (directed set property)
\begin{enumerate}
\item Let 
\[ v = \{ [f_{1}, D_{1}], \ldots, [f_{m}, D_{m}] \} \] 
be a pseudovertex. There exists a pseudovertex $v'$ such that $v \leq v'$ and
\[ v' = \{ [id_{E_{1}}, E_{1}], \ldots, [id_{E_{n}}, E_{n}] \}, \]
where $\{ E_{1}, \ldots, E_{n} \}$ is a finite partition of $im(v)$ into domains.
\item The pseudovertices  
\[ v = \{ [f_{1},D_{1}], \ldots, [f_{m},D_{m}] \} \quad \text{ and } \quad
 v' = \{ [f'_{1},D'_{1}], \ldots, [f'_{n},D'_{n}] \} \]
have a common upper bound if and only if $im(v)=im(v')$. 
\end{enumerate}
\end{lemma}

\begin{proof}
We first prove (1).
The function $f_{1}: D_{1} \rightarrow X$ is locally determined by $S$. It follows from the generation property (property (S5) from Definition \ref{definition:sstructure}) that there is a non-trivial $\mathcal{P} \in \mathbb{P}(D)$ such that, for $E \in \mathcal{P}$, there is $s_{E} \in S$ such that $E$ is the domain of $s_{E}$, $f_{1 \mid E} = s_{E}$, and $s_{E} \in \mathbb{S}(E, s_{E}(E))$. It follows that $v \nearrow v_{1}$, where 
\[ v_{1} = \{ [f_{2}, D_{2}], \ldots, [f_{m}, D_{m}] \} \cup \{ [f_{1}, E] \mid E \in \mathcal{P} \}, \]
where we let $\widehat{D} = D$ and $h \in \mathbb{S}(\widehat{D},D)$ be $id_{D}$ in the definition of expansion (Definition \ref{definition:expansion}). We note that $[f_{1}, E] = [s_{E}, E]$ for each $E \in \mathcal{P}$, since $f_{1 \mid E} = s_{E}$. Since $s_{E} \in \mathbb{S}(E, s_{E}(E))$, we have the equality $[s_{E}, E] = [id, s_{E}(E)]$. Recall from Lemma \ref{lemma:domainproperties}(2) that the set of domains is closed under translation by elements of $S$. It follows that $s_{E}(E) = E'$, for some domain $E' \in \mathcal{D}^{+}$. We rewrite $v_{1}$ using this fact, and find that 
\[ v_{1} = \{ [f_{2}, D_{2}], \ldots, [f_{m}, D_{m}] \} \cup \{ [id, E'] \mid E' \in \mathcal{P}' \}, \]
where $\mathcal{P}'$ is a finite partition of $f_{1}(D_{1})$ into domains. 

Replacing the pairs $[f_{2}, D_{2}], \ldots, [f_{m}, D_{m}]$ in the same manner, we produce a finite sequence of pseudovertices
\[ v = v_{0} \nearrow v_{1} \nearrow v_{2} \nearrow \ldots \nearrow v_{m}, \]
 and
\[ v_{m} = \{ [id_{\widehat{E}}, \widehat{E}] \mid \widehat{E} \in \widehat{\mathcal{P}} \}, \]
where $\widehat{\mathcal{P}}$ is a finite partition of $im(v)$ into domains. This proves (1). 

Now we prove (2). We can assume, after applying (1), that 
\[ v = \{ [id,D] \mid D \in \mathcal{P}_{1} \} \quad \text{and} \quad
v' = \{ [id,E] \mid E \in \mathcal{P}_{2} \} \]
for some partitions $\mathcal{P}_{1}, \mathcal{P}_{2} \subseteq \mathcal{D}^{+}_{S}$. By Proposition \ref{proposition:dspatterns}, there is a partition $\mathcal{P}_{3}$ such that $\mathcal{P}_{1} \preccurlyeq \mathcal{P}_{3}$ and $\mathcal{P}_{2} \preccurlyeq \mathcal{P}_{3}$. Setting 
\[ v'' = \{ [id,F] \mid F \in \mathcal{P}_{3} \}, \]
we find that $v \leq v''$ and $v' \leq v''$, where the expansions leading from $v$ to $v''$ 
consist of replacing each singleton $\{ [id,D] \}$ ($D \in \mathcal{P}_{1}$) with  
\[ \{ [id,F] \mid F \in \left(\mathcal{P}_{1} \right)_{\mid D} \}, \]
and the expansions leading from $v'$ to $v''$ are defined analogously.
\end{proof}

\begin{corollary} \label{corollary:directedset} (a family of ranked directed sets)
If $v_{1}, v_{2} \in \mathcal{PV}_{\mathbb{S}}$ have identical images (i.e., $im(v_{1}) = im(v_{2})$), then there is $v' \in \mathcal{PV}_{\mathbb{S}}$ such that $v_{1} \leq v'$
and $v_{2} \leq v'$.

In particular, $\mathcal{PV}_{\mathbb{S},Y}$ is a ranked directed set whenever $Y$ is a non-empty finite disjoint union of domains.
\end{corollary}

\begin{proof}
This is now immediate from Lemma \ref{lemma:expansionintocone}, from Definitions \ref{definition:rankdir} and \ref{definition:vertexrank}, and from Remark \ref{remark:futurerefrank}.
\end{proof}

\subsection{Examples: $S$-structures and the expansion relation} \label{subsection:sexamples}

We will now offer some examples of $S$-structures and their associated expansion relations. We will concentrate on the case of maximal $S$-structures, although we will also consider a non-maximal structure for R\"{o}ver's group (see Example \ref{example:Roever}).

\subsubsection{The expansion relation in the maximal $S$-structure}

The expansion relation (Definition \ref{definition:expansion}) takes an especially simple form in the maximal $S$-structure ($\mathbb{S}$,$\mathbb{P}$) (see Example \ref{example:maximalSstructure}). Namely, if $[f,D] \in \mathcal{B}$, any expansion at
$[f,D]$ simply consists of replacing $[f,D]$ with various pairs $[f,E]$, where the domains $E$ range over an arbitrary finite partition of $D$ into domains. This is Proposition \ref{proposition:expansioninSmax}. 

We recall that, if $(\mathbb{S},\mathbb{P})$ is the maximal $S$-structure and $D$ is a domain, then $\mathcal{P} \in \mathbb{P}(D)$ exactly when $\mathcal{P}$ is a finite partition of $D$ into domains.

\begin{proposition}
\label{proposition:expansioninSmax}
(The expansion relation in the maximal $S$-structure) Assume that the inverse semigroup $S$ acts on $X$, and let $(\mathbb{S},\mathbb{P})$ be the maximal $S$-structure.
If $v_{1} = \{ [f_{1},D_{1}], \ldots, [f_{m},D_{m}] \}$ and $v_{1} \nearrow v_{2}$, then there is some $j \in \{1, \ldots, m \}$ and a finite partition $\mathcal{P}$ of $D_{j}$ into domains such that 
\[ v_{2} = \left( \{ [f_{1},D_{1}], \ldots, [f_{m},D_{m}] \} - \{ [f_{j},D_{j}] \} \right) \cup \{ [f_{j},E] \mid E \in \mathcal{P} \}. \]
\end{proposition}

\begin{proof}
Since $v_{1} \nearrow v_{2}$, there is some pair $[f_{j},D_{j}]$, a domain $\widehat{D}$, an $h \in \mathbb{S}(\widehat{D},D_{j})$ and a partition $\widehat{\mathcal{P}}$ of $\widehat{D}$ into finitely many domains such that
\[ v_{2} = \left( \{ [f_{1},D_{1}], \ldots, [f_{m},D_{m}] \} - \{ [f_{j},D_{j}] \} \right) \cup \{ [f_{j}h,E] \mid E \in \widehat{\mathcal{P}} \}. \]
It therefore suffices to show that 
\[ \{ [f_{j}h,E] \mid E \in \widehat{\mathcal{P}} \} = \{ [f_{j},E'] \mid E' \in \mathcal{P} \} \]
for some partition $\mathcal{P}$ of $D_{j}$ into finitely many domains. 

Note that by Corollary \ref{corollary:closureunderrestriction}, $h_{\mid E} \in S$ for each $E \in \widehat{\mathcal{P}}$. Since the domain of $h_{\mid E}$ is $E$ and the codomain of $h_{\mid E}$ is $h(E)$, we must have $h_{\mid E} \in \mathbb{S}(E,h(E))$ by the definition of $(\mathbb{S},\mathbb{P})$. Consider the partition $\mathcal{P} = \{ h(E) \mid E \in \widehat{\mathcal{P}} \}$ of $D_{j}$. We find that 
\[ [f_{j}h,E] = [f_{j},h(E)] \]
for each $E \in \widehat{\mathcal{P}}$, by the definition of the equivalence relation on $\mathcal{B}$ (Definition \ref{definition:pairs}). It follows directly that 
\[ \{ [f_{j}h,E] \mid E \in \widehat{\mathcal{P}} \} = \{ [f_{j}, E'] \mid
E' \in \mathcal{P} \} \]
for  $\mathcal{P}$ as defined above, completing the proof.
\end{proof}

\begin{remark} \label{remark:closedrestrictions}
We say that the structure sets are \emph{closed under restrictions} if, whenever $h \in \mathbb{S}(D_{1},D_{2})$ and $D$ is a domain contained in $D_{1}$, $h \in \mathbb{S}(D, h(D))$.
The above argument generalizes naturally to $S$-structures $(\mathbb{S},\mathbb{P})$ when the structure sets are closed under restrictions. Note that the finite partition 
$\mathcal{P}$ must still come from $\mathbb{P}(D_{j})$ in this more general setting.
\end{remark}

\begin{remark}
\label{remark:easyexpansion}
If we simply replace a pair $[f,D]$ with a list of pairs
\[ [f,E_{1}], \ldots, [f,E_{k}], \]
where $\{ E_{1}, \ldots, E_{k} \} \in \mathbb{P}(D)$, then the result is always an expansion, no matter which $S$-structure $(\mathbb{S},\mathbb{P})$ we are working with. Proposition \ref{proposition:expansioninSmax} asserts that this is the only kind of expansion when we are using the maximal $S$-structure.

There are other kinds of expansions in more general (i.e., more restricted) $S$-structures.
\end{remark}

\subsubsection{Expansions for domains satisfying the compact ultrametric property} \label{subsubsection:expansions}

\begin{definition}
Let $S$ be such that $\mathcal{D}^{+}_{S}$ satisfies the compact ultrametric property. Assume that $v_{1}$ and $v_{2}$ are pseudovertices, $v_{1} = \{ [f_{1},D_{1}], \ldots, [f_{m},D_{m}] \}$, and $v_{1} \nearrow v_{2}$.

We say that $v_{1} \nearrow v_{2}$ is a \emph{simple expansion} if there is a $j \in \{ 1, \ldots, m \}$, a domain $\widehat{D}$, and a transformation
$h \in \mathbb{S}(\widehat{D}, D_{j})$ such that
\[ v_{2} = \left( \{ [f_{1},D_{1}], \ldots, [f_{m},D_{m}] \} - \{ [f_{j},D_{j}] \} \right) \cup \{ [f_{j}h,E] \mid E \in \mathcal{P}_{\widehat{D}} \}. \]
Recall that $\mathcal{P}_{\widehat{D}}$ is the maximal partition of $\widehat{D}$ (Definition \ref{definition:maximal}). 
\end{definition}

\begin{proposition} \label{proposition:factorization} (factorization into simple expansions) Suppose that $\mathcal{D}^{+}_{S}$ satisfies the compact ultrametric property. Let $v_{1}, v_{2}$ be pseudovertices such that $v_{1} \nearrow v_{2}$. We can find a sequence
\[ v_{1} = w_{0} \nearrow w_{1} \nearrow \ldots \nearrow w_{n} = v_{2}, \]
where each expansion $w_{i} \nearrow w_{i+1}$ is simple.
\end{proposition}

\begin{proof}
It suffices to show that, whenever $v_{1} \nearrow v_{2}$, there is some pseudovertex $\hat{v}$ such that $v_{1} \nearrow \hat{v}$ is a simple expansion and $\hat{v} \leq v_{2}$. 

Let $v_{1} = \{ [f_{1},D_{1}], \ldots, [f_{m},D_{m}] \}$. Since
$v_{1} \nearrow v_{2}$, there is some $j \in \{1, \ldots, m \}$, a domain $\widehat{D} \in \mathcal{D}^{+}_{S}$, a partition $\mathcal{P}$ of $\widehat{D}$, and an 
$h \in \mathbb{S}(\widehat{D},D_{j})$ such that 
\[ v_{2} = \left( \{ [f_{1},D_{1}], \ldots, [f_{m},D_{m}] \} - \{ [f_{j},D_{j}] \} \right) \cup \{ [f_{j}h,E] \mid E \in \mathcal{P} \}. \]
Let $\mathcal{P}_{\widehat{D}}$ denote the maximal partition of $\widehat{D}$ (Definition \ref{definition:maximal}). Define $\hat{v}$ as follows:
\[ \hat{v} =  \left( \{ [f_{1},D_{1}], \ldots, [f_{m},D_{m}] \} - \{ [f_{j},D_{j}] \} \right) \cup \{ [f_{j}h,E] \mid E \in \mathcal{P}_{\widehat{D}} \}. \]
Since $\mathcal{P}$ is  necessarily a refinement of $\mathcal{P}_{\widehat{D}}$ (by Lemma \ref{lemma:maximalgeneration}), it follows that $\hat{v} \leq v_{2}$, by Remark \ref{remark:easyexpansion}. Finally, $v_{1} \nearrow \hat{v}$ is a simple expansion, completing the proof.
\end{proof}

\subsubsection{Product $S$-structures}

\begin{definition} \label{definition:structuresonproducts} ($S$-structures for products)
For $i=1, \ldots, n$, let $X_{i}$ be a set, and let $S_{i}$ be an inverse semigroup acting on $X_{i}$. Assume that $\mathcal{D}^{+}_{S_{i}}$ satisfies the compact ultrametric property for $i = 1, \ldots, n$.
We consider the product action of $S_{(1,\ldots,n)}$ on 
$X_{1} \times \ldots \times X_{n}$ (Definition \ref{definition:products}). 

We define an $S_{(1,\ldots,n)}$-structure $(\mathbb{S},\mathbb{P})$ as follows. The structure function $\mathbb{S}$ is the maximal one; i.e., 
\begin{align*}
\mathbb{S}(D_{1} \times \ldots \times D_{n}&,
D'_{1} \times \ldots \times D'_{n}) = \\ 
&\{ (s_{1}, \ldots, s_{n}) \mid dom(s_{i}) = D_{i} \text{ and } im(s_{i}) = D'_{i}, \text{ for }i=1, \ldots,n \}, 
\end{align*}
for all pairs $(D_{1} \times \ldots \times D_{n}, D'_{1} \times \ldots \times D'_{n})$.

The pattern function $\mathbb{P}$ is defined inductively, following the example set by Brin (see Example \ref{example:Brin}). The set $\mathbb{P}(D_{1} \times \ldots \times D_{n})$ contains the partition $\{ D_{1} \times \ldots \times D_{n} \}$. If $\mathcal{P} \in \mathbb{P}(D_{1} \times \ldots \times D_{n})$, $E_{1} \times \ldots \times E_{n} \in \mathcal{P}$, and $j \in \{ 1, \ldots, n \}$, then  
\[ (\mathcal{P}- \{E_{1} \times \ldots \times E_{n}\}) \cup \{ E_{1} \times \ldots \times E_{j-1} \times \widehat{E} \times E_{j+1} \times \ldots \times E_{n} \mid \widehat{E} \in \mathcal{P}_{E_{j}} \} \] 
is also in $\mathbb{P}(D_{1} \times \ldots \times D_{n})$. This completes the inductive definition.
\end{definition}

\begin{remark}
It is straightforward to check that the above $(\mathbb{S},\mathbb{P})$ is an $S_{(1,\ldots,n)}$-structure. We omit the proof.
\end{remark}

\subsubsection{The expansion relation in some examples from Subsection \ref{subsection:examples}}

\begin{example} \label{example:basicexpansion} (The expansion relation when $\mathcal{D}^{+}$  satisfies the compact ultrametric property and the $S$-structure is maximal)  
Propositions \ref{proposition:expansioninSmax} and \ref{proposition:factorization} offer an easy description of the expansion relation when the $S$-structure is maximal and the set of domains $\mathcal{D}^{+}$ satisfies the compact ultrametric property. Namely, a simple expansion simply consists of replacing a pair $[f,D]$ with the list
\[ [f,E_{1}], \ldots, [f,E_{k}], \]
where $\{ E_{1}, \ldots, E_{k} \} = \mathcal{P}_{D}$ is the maximal partition of $D$ (Definition \ref{definition:maximal}). Any expansion is obtainable by repeating simple expansions. 

Consider Thompson's group $V$ (Example \ref{example:V}). We use
 the maximal $S_{V}$-structure. This means that 
\[ \mathbb{S}(B_{\omega_{1}}, B_{\omega_{2}}) = \{ \sigma_{\omega_{1},\omega_{2}} \}, \]
for each pair of finite binary strings $\omega_{1}, \omega_{2}$, and $\mathbb{P}(B_{\omega})$ is the set of all partitions of $B_{\omega}$ into domains.
A typical simple expansion consists of replacing $[f,B_{\omega}] \in \mathcal{B}$
with the elements $[f,B_{\omega 0}]$ and $[f,B_{\omega 1}]$. (We note that
$\{ B_{\omega 0}, B_{\omega 1} \}$ is the maximal partition of $B_{\omega}$.) 
The above completely describes simple expansions and, by extension, the entire partial order on pseudovertices.

In the case of Houghton's group 
$H_{n}$ (Example \ref{example:Houghton}), it is also natural to use the maximal $S_{H_{n}}$-structure. The domains of $S_{H_{n}}$ are as follows: 
\begin{enumerate}
\item singleton sets $\{ (j,k) \}$, where $j \in \{ 1, \ldots, n \}$
and $k \in \mathbb{N}$, which we denote $P_{j,k}$;
\item rays $\{ \{ j \} \times \{ k, k+1, \ldots \} \mid (j,k) \in \{ 1, \ldots, n \} \times \mathbb{N} \}$, which we denote $R_{j,k}$. 
\end{enumerate}
The maximal structure function $\mathbb{S}$ is determined by the following assignments:
\begin{enumerate}
\item $\displaystyle \mathbb{S}(P_{j_{1},k_{1}},P_{j_{2},k_{2}}) = \{ \alpha_{(j_{1},k_{1})}^{(j_{2},k_{2})} \}$, for all pairs $(j_{1},k_{1}), (j_{2},k_{2})$;
\item $\displaystyle \mathbb{S}(R_{j,k}, R_{j,k'}) = \{ \beta_{(j,k)}^{k'-k} \}$, for all $j \in \{1, \ldots, n \}$ and $k,k' \in \mathbb{N}$.
\end{enumerate}
The structure sets $\mathbb{S}(D_{1},D_{2})$ are empty in all remaining cases.

Two rays $R_{j_{1},k_{1}}$ and $R_{j_{2},k_{2}}$ thus have the same domain type (in the sense of Definition \ref{definition:domaintypes}) if and only if $j_{1} = j_{2}$, while any two singleton sets have the same domain type. This makes $n+1$ domain types in all.

We note that the maximal partition of $R_{j,k}$ is $\{ P_{j,k}, R_{j,k+1} \}$, while the maximal partition of $P_{j,k}$ is $\{ P_{j,k} \}$. Simple expansions take a unique form; namely, one replaces $[f,R_{j,k}] \in \mathcal{B}$ with $[f,P_{j,k}]$ and $[f,R_{j,k+1}]$. Such simple expansions totally determine the partial order on pseudovertices. 
\end{example}
 
 \begin{example} \label{example:Roever} (The expansion relation in R\"{o}ver's group) In the case of R\"{o}ver's group, the maximal $S_{R}$-structure leads to infinite structure sets
 $\mathbb{S}(B_{\omega_{1}}, B_{\omega_{2}})$. It will be advantageous to work with a smaller $S_{R}$-structure, especially when we later consider finiteness properties.  
 
 We instead define a structure function $\mathbb{S}$ by the rule
 \[ \mathbb{S}(B_{\omega_{1}},B_{\omega_{2}}) = \{ \sigma_{\omega_{1},\omega_{2}}, b_{\omega_{1},\omega_{2}}, c_{\omega_{1},\omega_{2}}, d_{\omega_{1},\omega_{2}} \}, \]
where $\omega_{1}$ and $\omega_{2}$ are arbitrary finite binary strings. We let $\mathbb{P}$ be the maximal pattern function; i.e., $\mathcal{P} \in \mathbb{P}(B_{\omega})$ if and only if $\mathcal{P}$ is a finite partition of $B_{\omega}$ into domains. 
The verification that $(\mathbb{S},\mathbb{P})$ is indeed an $S_{R}$-structure uses the self-similarity and contracting properties of Grigorchuk's group $G$ \cite{Grigorchuk}, as well as the fact that the elements $\{ 1, b, c, d \}$ are a subgroup of $G$.

We note that the structure sets are not closed under restrictions, since, for instance, 
$b_{\epsilon,\epsilon} \in \mathbb{S}(B_{\epsilon},B_{\epsilon})$, but $b_{\epsilon,\epsilon \mid B_{0}} = a_{0,0} \not \in \mathbb{S}(B_{0},B_{0})$.

Next, we turn to a description of simple expansions. Let $[f,B_{\omega}] \in \mathcal{B}$. A simple expansion at $[f,B_{\omega}]$ is determined by a choice of domain $B_{\omega'}$ and an element $h \in \mathbb{S}(B_{\omega'},B_{\omega})$; the simple expansion is then performed by replacing $[f,B_{\omega}]$ with $[fh,B_{\omega'0}]$ and $[fh,B_{\omega'1}]$. We must now consider various cases, which are determined by the element $h$. By the above description of the set 
$\mathbb{S}(B_{\omega'},B_{\omega})$, we have 
$h \in \{ \sigma_{\omega',\omega}, b_{\omega',\omega}, c_{\omega',\omega}, d_{\omega',\omega} \}$. 

\begin{enumerate}
\item If $h = \sigma_{\omega',\omega}$, then the corresponding expansion is the standard one, as from Thompson's group $V$. Thus, the expansion produces the elements $[f,B_{\omega0}]$ and $[f,B_{\omega1}]$.
\item If $h = b_{\omega',\omega}$, then we have 
$[fb_{\omega',\omega},B_{\omega'0}]$ and $[fb_{\omega',\omega},B_{\omega'1}]$. A straightforward calculation shows that the restrictions of $b_{\omega',\omega}$ to $B_{\omega'0}$ and $B_{\omega'1}$ are $a_{\omega'0,\omega0}$ and $c_{\omega'1,\omega1}$, respectively. Thus, the simple expansion outputs the pairs $[fa_{\omega'0,\omega0},B_{\omega'0}]$ and $[fc_{\omega'1,\omega1},B_{\omega'1}]$. We can factor the elements $a_{\omega'0,\omega0}$ and $c_{\omega'1,\omega1}$ as 
\[ a_{\omega0,\omega0}\sigma_{\omega'0,\omega0} \quad \text{ and } \quad
c_{\omega1,\omega1}\sigma_{\omega'1,\omega1}, \] respectively. The equivalence relation on $\mathcal{B}$ now implies that the simple expansion produces the pairs $[fa_{\omega0,\omega0},B_{\omega0}]$
and $[fc_{\omega1,\omega1},B_{\omega1}] = [f,B_{\omega1}]$.
\item If $h = c_{\omega',\omega}$, then a near-identical analysis shows that the simple expansion outputs the pairs $[fa_{\omega0,\omega0},B_{\omega0}]$ and $[fd_{\omega1,\omega1},B_{\omega_{1}}] = [f,B_{\omega1}]$. Thus, the resulting simple expansion is identical to the one from the previous case.
\item If $h = d_{\omega',\omega}$, then a straightforward calculation shows that the restrictions of $d_{\omega',\omega}$ to $B_{\omega'0}$ and $B_{\omega'1}$ are $\sigma_{\omega'0,\omega0}$ and $b_{\omega'1,\omega1} = b_{\omega1,\omega1} \sigma_{\omega'1,\omega1}$, respectively. The resulting simple expansion therefore outputs 
$[f\sigma_{\omega'0,\omega0},B_{\omega'0}] = [f,B_{\omega0}]$ and
$[fb_{\omega1,\omega1} \sigma_{\omega'1,\omega1}, B_{\omega'1}] =
[f,B_{\omega1}]$. Thus, the result is the standard simple expansion
$[f,B_{\omega0}]$ and $[f,B_{\omega1}]$, exactly as from (1).
 \end{enumerate}
We therefore conclude that there are two simple expansions from 
$[f,B_{\omega}]$:
\begin{enumerate}
\item the ``standard'' one, which outputs $[f,B_{\omega0}]$ and $[f,B_{\omega1}]$, and 
\item a nonstandard one, which outputs $[fa_{\omega0,\omega0},B_{\omega0}]$ and $[f,B_{\omega1}]$.
\end{enumerate}
These two simple expansions completely determine the partial order on pseudovertices.
\end{example}

\section{Two basic constructions of $\Gamma_{S}$-complexes} \label{section:5}

In this section, we will consider two constructions of contractible $\Gamma_{S}$-complexes. Both arise as the order complexes associated to directed sets. 

The first complex, $\Delta(\mathcal{V}_{\mathbb{S}})$, to be considered in Subsection \ref{subsection:directedsets}, will be improved upon in Section \ref{section:expansionscheme}, and the results will ultimately be used when we consider finiteness properties of groups.

The second complex, $\Delta(\mathcal{V}^{ord}_{\mathbb{S}})$, will be considered in Subsection \ref{subsection:ordered}. It holds promise as a $\Gamma_{S}$-complex with smaller stabilizers, but we will not make direct use of it in what follows.   

\subsection{Complexes defined by the directed sets $\mathcal{V}_{\mathbb{S}}$} \label{subsection:directedsets}

Corollary \ref{corollary:directedset} shows that there is a certain natural directed set upon which $\Gamma_{S}$ acts, namely the set $\mathcal{V}_{\mathbb{S}}$ (recall Definition \ref{definition:pseudovertex}). The order complex of a directed set is well-known to be contractible, so this leads directly to the definition of a contractible $\Gamma_{S}$-simplicial complex. In this subsection, we will describe this complex, and also give some information about the orbits and stabilizers of the associated group action. 

In a similar way, we associate contractible simplicial complexes to the directed sets 
$\mathcal{PV}_{\mathbb{S},Y}$, where $Y$ is an arbitrary finite disjoint union of non-empty domains. These complexes will be vital to establishing finiteness properties in subsequent sections.

\begin{theorem} \label{theorem:Hdirectedsets} (directed sets)
The set $\mathcal{V}_{\mathbb{S}}$ is a directed $\Gamma_{S}$-set. 

If $\emptyset \neq Y \subseteq X$ is a finite disjoint union of  domains, then $\mathcal{PV}_{\mathbb{S},Y}$ is a directed set.
\end{theorem}

\begin{proof}
Corollary \ref{corollary:directedset} says that $\mathcal{V}_{\mathbb{S}}$ and $\mathcal{PV}_{\mathbb{S},Y}$ are directed sets. The fact that $\mathcal{V}_{\mathbb{S}}$ is a $\Gamma_{S}$-set follows by applying Proposition \ref{proposition:hatSset} to $\Gamma_{S} \subseteq \widehat{S}$.

\end{proof}

\begin{theorem} \label{theorem:contractiblecomplexes} (directed set constructions of contractible complexes) The order complex
$\Delta(\mathcal{V}_{\mathbb{S}})$ is a contractible $\Gamma_{S}$-complex. Each complex $\Delta(\mathcal{PV}_{\mathbb{S},Y})$ is contractible.
\end{theorem}

\begin{proof}
The order complex of a directed set is well-known to be contractible. (See Proposition 9.3.14 from \cite{BookbyRoss}, for instance.)  Otherwise, this is an entirely straightforward consequence of Theorem \ref{theorem:Hdirectedsets}.
\end{proof}

We next consider basic properties of the group action.

\begin{proposition} \label{proposition:stab} (A virtual description of vertex stabilizers)
Let 
\[ v = \{ [f_{1}, D_{1}], \ldots, [f_{m}, D_{m}] \}  \]
be a vertex of the order complex $\Delta(\mathcal{V}_{\mathbb{S}})$. We write $H$ in place of $\Gamma_{S}$.

The stabilizer group $H_{v}$ has a finite index subgroup $H'_{v}$ that is isomorphic to the group 
\[ \prod_{i=1}^{m} \mathbb{S}(D_{i}, D_{i}). \]

In particular, if $|\mathbb{S}(D,D)| < \infty$ for all domains $D$, then $H_{v}$ is always finite.
\end{proposition}

\begin{proof}
Let $v = \{ [f_{1}, D_{1}], \ldots, [f_{m}, D_{m}] \}$. Since the stabilizer group $H_{v}$ fixes $v$, we have, for a given $\gamma \in H_{v}$, the equality
\[ \{ [f_{1}, D_{1}], \ldots, [f_{m}, D_{m}] \} = \{ [\gamma f_{1}, D_{1}], \ldots, [\gamma f_{m}, D_{m}] \}, \]
from which it follows that $[\gamma f_{i}, D_{i}] = [f_{\sigma_{\gamma}(i)}, D_{\sigma_{\gamma}(i)}]$,
for $i = 1, \ldots, m$, where $\sigma_{\gamma}$ is an element of the symmetric group $S_{m}$. The assignment $\psi: H_{v} \rightarrow S_{m}$ sending $\gamma$ to $\sigma_{\gamma}$ is a homomorphism. Denote the kernel of $\psi$ by $H'_{v}$. It follows that $H'_{v}$ has finite index in $H_{v}$, and, for all $\gamma \in H'_{v}$, $[\gamma f_{i}, D_{i}] = [f_{i}, D_{i}]$, for $i=1, \ldots, m$. Fix $\gamma \in H'_{v}$ for what follows. For $i = 1, \ldots, m$, there is an $h_{i} \in \mathbb{S}(D_{i}, D_{i})$ such that $\gamma f_{i} = f_{i}h_{i}$, by the definition of $\sim$ (see Definition \ref{definition:pairs}). Solving for $\gamma$, we find that $\gamma = f_{i}h_{i}f_{i}^{-1}$ on the 
subset $f_{i}(D_{i})$ of $X$. Since the $f_{i}(D_{i})$ partition $X$, the elements $h_{i}$ completely determine $\gamma$. The desired embedding sends $\gamma$ to $(h_{1}, \ldots, h_{m})$, where the $h_{i}$ are as defined above (and, of course, depend on $\gamma$); this assignment is also a homomorphism.

One easily observes that the above embedding is, in fact, also a surjection.
\end{proof}

\begin{corollary} (construction of $E_{fin}(\Gamma_{S})$) \label{theorem:classifyingspace}
Let $(\mathbb{S}, \mathbb{P})$ be an $S$-structure such that the structure set $\mathbb{S}(D,D)$ is finite, for all 
$D \in \mathcal{D}^{+}$. The complex
$\Delta(\mathcal{V}_{\mathbb{S}})$ is a model for $E_{fin}(\Gamma_{S})$. 
\end{corollary}

\begin{proof}
The statement that $\Delta(\mathcal{V}_{\mathbb{S}})$ is a model for $E_{fin}(\Gamma_{S})$ simply means that $\Delta(\mathcal{V}_{\mathbb{S}})$ is contractible and $\Gamma_{S}$ acts with finite cell stabilizers.
These properties follow immediately from Theorem \ref{theorem:contractiblecomplexes} and Proposition \ref{proposition:stab}.
\end{proof}

\begin{proposition} \label{proposition:vertexorbits} (description of orbits in $\Delta(\mathcal{V}_{\mathbb{S}})$ and $\Delta(\mathcal{PV}_{\mathbb{S}})$) 
Two vertices $v_{1}, v_{2} \in \Delta(\mathcal{V}_{\mathbb{S}})$ 
are in the same $\Gamma_{S}$-orbit if and only if they have the same type. Two vertices $v_{1}, v_{2} \in \Delta(\mathcal{PV}_{\mathbb{S}})$ are in the same $\widehat{S}$-orbit if and only if they have the same type.
\end{proposition}

\begin{proof}
The type of a vertex is invariant under the action of $\Gamma_{S}$, by Definition \ref{definition:hatSaction}. Thus, it suffices to show that if two vertices have the same type, then they are in the same orbit.

Assume that $v_{1} = \{ [f_{1},D_{1}], \ldots, [f_{m},D_{m}] \}$ and $v_{2} = \{ [g_{1},E_{1}], \ldots, [g_{m},E_{m}] \}$, where $[D_{i}] = [E_{i}]$ for $i=1, \ldots, m$. Since $D_{i}$ and $E_{i}$ have the same domain type, $\mathbb{S}(D_{i},E_{i}) \neq \emptyset$, for $i=1, \ldots, m$; for each $i$, we choose $h_{i} \in \mathbb{S}(D_{i},E_{i})$. 

Define $\gamma : X \rightarrow X$ by letting $\gamma_{\mid f_{i}(D_{i})} = g_{i}h_{i}f_{i}^{-1}$. The fact that $\{ f_{i}(D_{i}) \mid i = 1, \ldots, m \}$ is a partition of $X$ shows that $\gamma \in \Gamma_{S}$. For each $i$, we then have
\begin{align*}
\gamma \cdot [f_{i},D_{i}] &= [g_{i}h_{i}, D_{i}] \\
&= [g_{i}, h_{i}(D_{i})] \\
&= [g_{i}, E_{i}]. 
\end{align*}
(The second-to-last inequality appeals to the definition of the equivalence relation on $\mathcal{A}$; see Definition \ref{definition:pairs}.)     
It follows that $\gamma \cdot v_{1} = v_{2}$.

Given two pseudovertices $v_{1}$ and $v_{2}$
of the same type, essentially the same method produces an element $\hat{s} \in \widehat{S}$ such that $\hat{s} \cdot v_{1} = v_{2}$. 
\end{proof}

\begin{proposition} \label{proposition:isomorphism}
(isomorphisms between $\Delta(\mathcal{PV}_{\mathbb{S},Y_{1}})$ and $\Delta(\mathcal{PV}_{\mathbb{S},Y_{2}})$)
If $Y_{1}$ and $Y_{2}$ are non-empty finite disjoint unions of domains, and $\hat{s} \in \widehat{S}$ is such that $\hat{s}(Y_{1}) = Y_{2}$, then $\hat{s}$ induces a simplicial isomorphism between the complexes
$\Delta(\mathcal{PV}_{\mathbb{S},Y_{1}})$ and $\Delta(\mathcal{PV}_{\mathbb{S},Y_{2}})$. 
\end{proposition}

\begin{proof}
The fact that $\hat{s}$ determines an isomorphism between the directed sets $\mathcal{PV}_{\mathbb{S},Y_{1}}$ and $\mathcal{PV}_{\mathbb{S},Y_{2}}$ is a simple consequence of Proposition \ref{proposition:hatSset}. It then follows immediately that $\hat{s}$ induces the required simplicial isomorphism.
\end{proof}

\subsection{Complexes defined by ordered vertices} \label{subsection:ordered}

Write $\mathcal{V}$ in place of $\mathcal{V}_{\mathbb{S}}$. Our goal in this subsection will be to build a class of contractible $\Gamma_{S}$-complexes that will have smaller vertex stabilizers than the complexes from Theorem \ref{theorem:contractiblecomplexes}. In some cases, we will obtain free actions of $\Gamma_{S}$.

In what follows, we will need to fix a linear ordering $\preceq$ on the set $\{ [D] \mid D \in \mathcal{D}^{+} \}$
of domain types (see Definition \ref{definition:domaintypes}).

\begin{definition} (the ordered vertices $\mathcal{V}^{ord}$) 
Let $\mathcal{V}^{ord}$ be the set of all $m$-tuples  
\[ ([f_{1},D_{1}], \ldots, [f_{m},D_{m}]) \] 
where $m \in \mathbb{N}$, $\{ [f_{1},D_{1}], \ldots,
[f_{m},D_{m}] \} \in \mathcal{V}$, and $[D_{i}] \preceq [D_{j}]$  if  $i<j$.
\end{definition}

\begin{remark} \label{remark:explainordervertices} 
An element of $\mathcal{V}^{ord}$ is simply the result of ordering the members of a vertex 
\[ \{ [f_{1}, D_{1}], \ldots, [f_{m}, D_{m}] \} \in \mathcal{V} \]
into an $m$-tuple. The additional constraint imposed by $\preceq$ is unnecessary for the theory, but including it helps to minimize the number of orbits of vertices under the action of $\Gamma_{S}$ (respectively, $H$).
\end{remark}

\begin{definition} (order-forgetting function) 
Define $F: \mathcal{V}^{ord} \rightarrow \mathcal{V}$ by the rule
\[ F( [f_{1}, D_{1}], \ldots, [f_{m}, D_{m}]) = \{ [f_{1}, D_{1}], \ldots, [f_{m}, D_{m}] \}. \]
The function $F$ is the \emph{order-forgetting function}.
\end{definition}
    
\begin{definition} (the partial order on $\mathcal{V}^{ord}$) 
Let $v_{1}, v_{2} \in \mathcal{V}^{ord}$. We write $v_{1} < v_{2}$ if $F(v_{1}) < F(v_{2})$. (I.e., $v_{1} \leq v_{2}$ if $v_{1} = v_{2}$ or if $F(v_{2})$ is obtainable from $F(v_{1})$ by a finite, non-empty sequence of expansions.) It is routine to check that $\leq$ is a partial order on $\mathcal{V}^{ord}$.
\end{definition}

\begin{definition} (the action on $\mathcal{V}^{ord}$)
If $\mathcal{V} = \mathcal{V}_{\mathbb{S}}$, then $\mathcal{V}^{ord}$ is a $\Gamma_{S}$-set. 
The action of a given $\gamma \in \Gamma_{S}$ is as follows:
\[ \gamma \cdot ([f_{1},D_{1}], \ldots, [f_{m},D_{m}]) = (\gamma \cdot [f_{1},D_{1}], \ldots,
\gamma \cdot [f_{m}, D_{m}]). \]
The proof that the above assignment determines an action $\cdot : \Gamma_{S} \times
\mathcal{V}^{ord} \rightarrow \mathcal{V}^{ord}$ is a minor modification of the proof of Proposition \ref{proposition:hatSset} and will be omitted.
\end{definition}

\begin{theorem} (the complex $\Delta(\mathcal{V}^{ord})$) Assume that each vertex $v \in \mathcal{V}$ admits at least one expansion.
The $\Gamma_{S}$-complex  $\Delta(\mathcal{V}^{ord})$ is  contractible 
\end{theorem}

\begin{proof}
We first show that $(\mathcal{V}^{ord}, \leq)$ is a directed set. Let $v_{1}, v_{2} \in \mathcal{V}^{ord}$. Since $(\mathcal{V}, \leq)$ is a directed set by Theorem \ref{theorem:Hdirectedsets}, there is some $\hat{v} \in \mathcal{V}$ such that
$F(v_{1}) \leq \hat{v}$ and $F(v_{2}) \leq \hat{v}$. After expansion from $\hat{v}$ if necessary (here we use the assumption that $\hat{v}$ admits at least one expansion), we can assume without loss of generality that $F(v_{1}) < \hat{v}$ and $F(v_{2}) < \hat{v}$. Since the order-forgetting map is clearly surjective, we have $\hat{v} = F(v)$, for some $v \in \mathcal{V}^{ord}$. It follows that $v_{1}, v_{2} < v$, so $(\mathcal{V}^{ord}, \leq)$ is a directed set. It follows that
$\Delta(\mathcal{V}^{ord})$ is contractible.

We now show that $\Gamma_{S}$ acts in an order-preserving way. Let $v_{1} < v_{2}$, where $v_{i} \in \mathcal{V}^{ord}$ for $i=1,2$. It follows that $F(v_{1}) < F(v_{2})$. For $\gamma \in \Gamma_{S}$, we have $\gamma \cdot F(v_{1}) < \gamma \cdot F(v_{2})$, 
since the action of $\Gamma_{S}$ on $\mathcal{V}$ is order-preserving (Proposition \ref{proposition:hatSset} and Theorem \ref{theorem:contractiblecomplexes}). It is clear from the definition of $F$ that $\gamma \cdot F(v) = F( \gamma \cdot v)$, for all $\gamma \in \Gamma_{S}$ and $v \in \mathcal{V}^{ord}$, so $F(\gamma \cdot v_{1}) < F(\gamma \cdot v_{2})$. It then follows that $\gamma \cdot v_{1} < \gamma \cdot v_{2}$. Thus, $\Gamma_{S}$ preserves the order on $\mathcal{V}^{ord}$.
\end{proof}

 \begin{proposition} \label{proposition:stabord} (A virtual description of vertex stabilizers in $\Delta(\mathcal{V}^{ord})$)
Write $H$ in place of $\Gamma_{S}$.
Let 
\[ v = ([f_{1}, D_{1}], \ldots, [f_{m}, D_{m}]) \in \mathcal{V}^{ord} \]
be a vertex in $\mathcal{V}^{ord}$.

The stabilizer group $H_{v}$ is isomorphic to
\[ \prod_{i=1}^{m} \mathbb{S}(D_{i}, D_{i}). \]

In particular, if $|\mathbb{S}(D,D)| < \infty$ for all domains $D$, then $H_{v}$ is always finite. If $\mathbb{S}(D,D)$ is the trivial group for all domains $D$, then the action of $H$ on $K(\mathcal{V}^{ord})$ is free.
\end{proposition}

\begin{proof}
This is just like the proof of Proposition \ref{proposition:stab}, except that a given $\gamma \in H_{v}$ fixes each individual  coordinate of $v$. Thus,
\[ \gamma \cdot [f_{i},D_{i}] = [\gamma f_{i},D_{i}], \]
for each $i$. The proof of Proposition \ref{proposition:stab} therefore embeds $H_{v}$ itself into $\prod \mathbb{S}(D_{i},D_{i})$, rather than a finite index subgroup $H'_{v}$
of $H_{v}$.
\end{proof}

\section{Complexes defined by expansion schemes} \label{section:expansionscheme}

In this section, we introduce expansion schemes, a device for simplifying the directed set construction from Subsection \ref{subsection:directedsets}. We denote an arbitrary expansion scheme by the letter $\mathcal{E}$; the ``$\mathcal{E}$'' could just as easily stand for ``elementary interval'', since the expansion scheme idea generalizes the latter idea from \cite{Stein}. Roughly speaking, the complexes $\Delta(\mathcal{V}_{\mathbb{S}})$ of Subsection\ref{subsection:directedsets} mirror those of \cite{Brown}, while the complexes $\Delta^{\mathcal{E}}(\mathcal{V}_{\mathbb{S}})$ are analogous to those from \cite{Stein}. Thus, the complexes defined in this section are often locally finite, for instance. In subsequent sections, we will establish finiteness properties of the groups 
$\Gamma_{S}$ using these simplified complexes. 

The section is structured as follows. Subsection \ref{subsection:E} defines expansion schemes and the simplicial complexes $\Delta^{\mathcal{E}}(\mathcal{V}_{\mathbb{S}})$ associated to them. The next several subsections help to set up the application of Brown's Finiteness Criterion (Theorem \ref{theorem:Brown}). Subsection \ref{subsection:nconnectedE} offers a sufficient condition for a given complex $\Delta^{\mathcal{E}}(\mathcal{V}_{\mathbb{S}})$ to be $n$-connected. Subsection \ref{subsection:Efilt} defines a natural filtration, and summarizes its properties. Subsection \ref{subsection:fnstab} states a sufficient condition for simplex stabilizer groups to have type $F_{n}$. 

The section concludes with a subsection about the generation of expansion schemes (Subsection \ref{subsection:Egeneration}) and a subsection containing some general classes of expansion schemes (Subsection \ref{subsection:Eexamples}).

\subsection{Expansion schemes} \label{subsection:E}

We write $\mathcal{V}$ in place of $\mathcal{V}_{\mathbb{S}}$. If $Y \subseteq X$ is a non-empty finite disjoint union of domains, then we write $\mathcal{PV}_{Y}$ in place of  $\mathcal{PV}_{\mathbb{S},Y}$. We let $\mathcal{PV}$ denote the union of all $\mathcal{PV}_{Y}$, as $Y$ ranges over all non-empty finite disjoint unions of domains.

In what follows, $2^{X}$ denotes the power set of $X$.

\begin{definition} \label{definition:scheme} 
($\mathcal{E}$-expansion; expansion scheme)
Assume that  $\mathcal{E}: \mathcal{B} \rightarrow 2^{\mathcal{PV}}$ satisfies (1)-(3), for each $[f,D] \in \mathcal{B}$:
\begin{enumerate}
\item Each $w \in \mathcal{E}([f,D])$ is the result of a sequence of expansions from $\{[f,D]\}$; i.e., for each $w \in \mathcal{E}([f,D])$, $\{ [f,D] \} \leq w$;
\item $\{ [f,D] \} \in \mathcal{E}([f,D])$;
\item ($\widehat{S}$-invariance) for each $\hat{s} \in \widehat{S}$, and each $b \in \mathcal{B}$ for which $\hat{s} \cdot b$ is defined (see Proposition \ref{definition:hatSaction}), $\hat{s} \cdot \mathcal{E}(b) = \mathcal{E}(\hat{s} \cdot b)$. 
\end{enumerate}
Let $v \in \mathcal{PV}$; we write $v = \{ b_{1}, \ldots, b_{m} \}$, where $b_{1}$, $\ldots$, $b_{m} \in \mathcal{B}$. We say that $v'$ is a result of \emph{$\mathcal{E}$-expansion from $v$} if there are $v'_{i} \in \mathcal{E}(b_{i})$, for $i = 1, \ldots, m$, such that 
\[ v' = \bigcup_{i=1}^{m} v'_{i}. \]
We say that $\mathcal{E}$ is an \emph{expansion scheme} if
\begin{enumerate}
\item[(4)] for every $[f,D] \in \mathcal{B}$ and every $w_{1}, w_{2} \in \mathcal{E}([f,D])$ such that $w_{1} \leq w_{2}$, $w_{2}$ is the result of $\mathcal{E}$-expansion from $w_{1}$.
\end{enumerate}
\end{definition}

\begin{definition} ($\mathcal{E}$-chains; $\Delta^{\mathcal{E}}$) Let $\mathcal{E}$ be an expansion scheme. A sequence of pseudovertices
\[ v_{1} \lneq v_{2} \lneq v_{3} \lneq \ldots \lneq v_{m} \]
is called an \emph{$\mathcal{E}$-chain} if each $v_{i}$ is the result of an $\mathcal{E}$-expansion from $v_{1}$.

Let $\Delta^{\mathcal{E}} = ( \mathcal{V}^{\mathcal{E}}, \mathcal{S}^{\mathcal{E}})$, where
$\mathcal{V}^{\mathcal{E}} = \mathcal{V}$ and $\mathcal{S}^{\mathcal{E}}$ is the set of all finite non-empty $\mathcal{E}$-chains in $\mathcal{V}$.

In an analogous way, we can define the complexes $\Delta^{\mathcal{E}}(\mathcal{PV})$ and $\Delta^{\mathcal{E}}(\mathcal{PV}_{Y})$.
\end{definition}

\begin{lemma} \label{lemma:Esubset} (subset of an $\mathcal{E}$-chain is an $\mathcal{E}$-chain) Let $\mathcal{E}$ be an expansion scheme. If 
\[ v_{1} \lneq v_{2} \lneq \ldots \lneq v_{m} \]
is an $\mathcal{E}$-chain and $m \geq 2$, then so is
\[ v_{1} \lneq v_{2} \lneq \ldots \lneq \hat{v}_{i} \lneq \ldots \lneq v_{m}, \]
for $i \in \{ 1, \ldots, m \}$, where $\hat{v}_{i}$ indicates that $v_{i}$ is to be omitted.
\end{lemma}

\begin{proof}
This is immediate from the definition if $i>1$. It therefore suffices to consider the sequence
\[ v_{2} \lneq v_{3} \lneq \ldots \lneq v_{m}, \]
where $m \geq 2$. To prove that the above is an $\mathcal{E}$-chain, it suffices to show that each $v_{k}$ ($2 \leq k \leq m$) is obtained by $\mathcal{E}$-expansion from $v_{2}$.

Let $v_{1} = \{ b_{1}, \ldots, b_{n} \}$, where the $b_{i}$ are from $\mathcal{B}$. Since $v_{2}$ is the result of $\mathcal{E}$-expansion from $v_{1}$, we can write
\[ v_{2} = \bigcup_{i=1}^{n} v'_{i}, \]
where $v'_{i} \in \mathcal{E}(b_{i})$ for each $i$. Similarly, we can write
\[ v_{k} = \bigcup_{i=1}^{n} v''_{i}, \]
where $v''_{i} \in \mathcal{E}(b_{i})$ for each $i$. Since the pseudovertices $v'_{i}$ and $v''_{i}$ are obtained by (possibly repeated) expansion from $\{ b_{i} \}$ (see Definition 
\ref{definition:scheme}(1)), $im(v'_{i}) = im(\{ b_{i} \}) = im(v''_{i})$, for each $i$, by Proposition \ref{proposition:expansion}. Proposition \ref{proposition:orderlocal}(2) now implies that $v'_{i} \leq v''_{i}$ for each $i$, since $v_{2} \leq v_{k}$. Property (4) of Definition \ref{definition:scheme} implies that $v''_{i}$ is the result of $\mathcal{E}$-expansion from $v'_{i}$, for $i = 1, \ldots, n$. It follows that $v_{k}$ is itself the result of $\mathcal{E}$-expansion from $v_{2}$, completing the proof. 
\end{proof}

\begin{theorem}
(the simplicial complexes $\Delta^{\mathcal{E}}$)
\label{theorem:Ecomplex}
The pair $\Delta^{\mathcal{E}} = (\mathcal{V}^{\mathcal{E}}, \mathcal{S}^{\mathcal{E}})$ is an abstract $\Gamma_{S}$-simplicial complex if $\mathcal{E}$ is an expansion scheme. 

If $Y_{1}, Y_{2} \subseteq X$ can be expressed as non-empty finite disjoint unions of domains, and there is some $\hat{s} \in \widehat{S}$ such that $Y_{1}$ is contained in the domain of $\hat{s}$ and $\hat{s}(Y_{1}) = Y_{2}$, then $\hat{s}$ induces an isomorphism between the complexes 
$\Delta^{\mathcal{E}}(\mathcal{PV}_{Y_{1}})$ and 
$\Delta^{\mathcal{E}}(\mathcal{PV}_{Y_{2}})$.
\end{theorem} 

\begin{proof}
Lemma \ref{lemma:Esubset} easily shows that the set $\mathcal{S}^{\mathcal{E}}$ is closed under taking non-empty subsets. This implies that $\Delta^{\mathcal{E}}$ is an abstract simplicial complex. 

It is now clear that $\Delta^{\mathcal{E}}$ is a subcomplex of $\Delta(\mathcal{V})$. Thus, $\Delta^{\mathcal{E}}$ is contained in a complex with  a $\Gamma_{S}$-action. The $\Gamma_{S}$-invariance of $\Delta^{\mathcal{E}}$ follows directly from property (3) from
Definition \ref{definition:scheme}. This completes the proofs of the statements about $\Delta^{\mathcal{E}}$.

We obtain an isomorphism between the complexes
$\Delta^{\mathcal{E}}(\mathcal{PV}_{Y_{1}})$ and $\Delta^{\mathcal{E}}(\mathcal{PV}_{Y_{2}})$
by combining Proposition \ref{proposition:hatSset} with $\widehat{S}$-invariance of $\mathcal{E}$ (property (3) from Definition \ref{definition:scheme}).
 \end{proof}

\begin{example} (the trivial and maximal expansion schemes)
We can define $\mathcal{E}: \mathcal{B} \rightarrow 2^{\mathcal{PV}}$ by the 
rule: 
\[ \mathcal{E}([f,D]) = \{ \{ [f,D] \} \}. \]
This assignment necessarily satisfies (1)-(4) from Definition \ref{definition:scheme}. The resulting simplicial complex $\Delta^{\mathcal{E}}$ is simply the discrete set of vertices $\mathcal{V}$.

At the opposite extreme, we can let
\[ \mathcal{E}([f,D]) = \{ w \in \mathcal{PV} \mid \{ [f,D] \} \leq w \}. \]
The resulting simplicial complex $\Delta^{\mathcal{E}}$ is 
isomorphic to $\Delta(\mathcal{V})$.

Thus, we will need to choose the expansion scheme $\mathcal{E}$ carefully in order to guarantee that $\Delta^{\mathcal{E}}$ has useful topological properties, and also that $\Delta^{\mathcal{E}}$ is more economical than the construction from Subsection \ref{subsection:directedsets}. 
\end{example}

\subsection{$n$-connected expansion schemes} \label{subsection:nconnectedE}

\begin{definition} \label{definition:nconnectedE}
Let $\mathcal{E}$ be an expansion scheme. We say that $\mathcal{E}$ is \emph{$n$-connected}
if, for each $b \in \mathcal{B}$ and each pseudovertex $v$ satisfying $\{ b \} \lneq v$, 
\[ lk(\{b\}, \Delta(\mathcal{E}(b))_{[\{b\},v]}) \]
is $(n-1)$-connected.
\end{definition}

\begin{remark}
It may be useful to unpack the above definition somewhat. For each $b \in \mathcal{B}$, the set $\mathcal{E}(b)$ is partially ordered by $\leq$. We can therefore consider the order complex $\Delta(\mathcal{E}(b))$. If $v$ is an arbitrary pseudovertex obtained by a non-empty sequence of expansions from $\{ b \}$, then we consider the subcomplex determined by the interval $[\{b\}, v]$; i.e., 
\[ \Delta(\mathcal{E}(b))_{[\{b\},v]}. \]
(We emphasize that $v$ is  truly arbitrary. For instance, $v$ need not be a member of $\mathcal{E}(b)$.)
We then consider the link of $\{ b\}$ in the above complex. The expansion scheme $\mathcal{E}$ is said to be \emph{$n$-connected} if the resulting link is $(n-1)$-connected, for every possible choice of $b$ and $v$.
\end{remark}

\begin{remark}
The property of $\mathcal{E}$ being $n$-connected is dependent upon the $S$-structure $(\mathbb{S},\mathbb{P})$, even though we have suppressed any explicit mention of the $S$-structure in the definition.
\end{remark}

\begin{theorem} \label{theorem:bigone}
($n$-connectedness of $\Delta^{\mathcal{E}}$)
If $\mathcal{E}$ is an $n$-connected expansion scheme, then the complexes
$\Delta^{\mathcal{E}}$ and $\Delta^{\mathcal{E}}(\mathcal{PV}_{Y})$ are $n$-connected. 
\end{theorem}

\begin{proof}
We will prove the theorem only in the case of $\Delta^{\mathcal{E}}$, the case of $\Delta^{\mathcal{E}}(\mathcal{PV}_{Y})$ being similar.

Recall that $\mathcal{V}$ is a ranked directed set by Corollary \ref{corollary:directedset}, where $r:\mathcal{V} \rightarrow \mathbb{N}$ defined by $r(v) = |v|$ is a ranking function. We note, furthermore, that $\Delta^{\mathcal{E}}$ is a simplicial complex on $\mathcal{V}$.

We now try to follow the approach of Lemma \ref{lemma:nconnectedness}. Let $v_{1}$ and $v_{2}$ be arbitrary vertices of $\mathcal{V}$, where $v_{1} < v_{2}$. It suffices, by Lemma \ref{lemma:nconnectedness}, to prove that 
$lk(v_{1},\Delta^{\mathcal{E}}_{[v_{1},v_{2}]})$ is $(n-1)$-connected. 

We write $v_{1} = \{ b_{1}, \ldots, b_{m} \}$, where the $b_{i}$ are elements of $\mathcal{B}$. By Proposition \ref{proposition:orderlocal}(1), there are pseudovertices $w_{1}, \ldots,
w_{m}$  such that
\begin{itemize}
\item $v_{2} = \bigcup_{k=1}^{m} w_{k}$;
\item for each $k$, $\{ b_{k} \} \leq w_{k}$.
\end{itemize}
There is a natural isomorphism of simplicial complexes
\[ lk\left((\{b_{1}\}, \ldots, \{b_{m}\}), \prod_{k=1}^{m} \Delta(\mathcal{E}(b_{k}))_{[\{b_{k}\}, w_{k}]}\right) \cong lk(v_{1}, \Delta^{\mathcal{E}}_{[v_{1},v_{2}]}).\]
(Note that the product in the above formula is the simplicial product; see Definition \ref{definition:product}.) By Theorem \ref{theorem:product}, 
\[ lk\left((\{b_{1}\}, \ldots, \{b_{m}\}), \prod_{k=1}^{m} \Delta(\mathcal{E}(b_{k}))_{[\{b_{k}\}, w_{k}]}\right) \cong \Asterisk_{k=1}^{m} lk\left(\{b_{k}\}, \Delta(\mathcal{E}(b_{k}))_{[\{b_{k}\}, w_{k}]}\right). \]
At least one of the links in the latter join is non-empty, since $v_{1}<v_{2}$. Since the join operation is associative up to isomorphism of simplicial complexes, the join of an $n_{1}$-connected
simplicial complex with an $n_{2}$-connected simplicial complex is $(n_{1}+n_{2}+2)$-connected, and each non-empty factor in the join is $(n-1)$-connected by hypothesis, we conclude that the latter join is at least $(n-1)$-connected. It now follows from Lemma \ref{lemma:nconnectedness} that $\Delta^{\mathcal{E}}$ is $n$-connected.
\end{proof}

\subsection{A filtration of $\Delta^{\mathcal{E}}(\mathcal{V})$ by $\Gamma_{S}$-finite subcomplexes} \label{subsection:Efilt}

We continue to write $\mathcal{V}$ in place of $\mathcal{V}_{\mathbb{S}}$. We will define a natural filtration $\{ \Delta^{\mathcal{E}}(\mathcal{V})_{n} \mid n \in \mathbb{N} \}$ of any $\Delta^{\mathcal{E}}(\mathcal{V})$ by $\Gamma_{S}$-invariant subcomplexes. We will also give a sufficient condition for the action of $\Gamma_{S}$ to be cocompact on each subcomplex in the filtration.

\begin{definition} (a filtration of $\Delta^{\mathcal{E}}(\mathcal{V})$) \label{definition:thefiltration}
We let $\Delta^{\mathcal{E}}(\mathcal{V})_{n}$ denote the subcomplex of $\Delta^{\mathcal{E}}(\mathcal{V})$ spanned by vertices of rank less than or equal to $n$.
\end{definition}

\begin{proposition} \label{proposition:findiminv}
Each $\Delta^{\mathcal{E}}(\mathcal{V})_{n}$ is a finite-dimensional, $\Gamma_{S}$-invariant subcomplex of $\Delta^{\mathcal{E}}(\mathcal{V})$, and 
\[ \Delta^{\mathcal{E}}(\mathcal{V}) = \bigcup_{n=1}^{\infty} \Delta^{\mathcal{E}}(\mathcal{V})_{n}. \] 
\end{proposition}

\begin{proof}
It is clear that $\Delta^{\mathcal{E}}(\mathcal{V})_{n}$ is a subcomplex of $\Delta^{\mathcal{E}}(\mathcal{V})$ for each $n$. A $j$-simplex in 
$\Delta^{\mathcal{E}}(\mathcal{V})_{n}$
is an $\mathcal{E}$-chain $v_{0} \lneq v_{1} \lneq \ldots \lneq v_{j}$. Since the ranks of the vertices are positive integers satisfying 
$|v_{0}| < |v_{1}| < |v_{2}| < \ldots < |v_{j}|$, $|v_{j}|$ is at least $j+1$. Thus, $j+1 \leq n$. It follows easily that each $\Delta^{\mathcal{E}}(\mathcal{V})_{n}$ is at most $(n-1)$-dimensional.

To show that $\Delta^{\mathcal{E}}(\mathcal{V})_{n}$ is $\Gamma_{S}$-invariant, it suffices to show that the action of $\Gamma_{S}$ on the vertices is rank-preserving. But this is an immediate consequence of the definition of the action (see Definition \ref{definition:hatSaction}). 

Finally, we establish the inclusion $\Delta^{\mathcal{E}}(\mathcal{V}) \subseteq \bigcup_{n=1}^{\infty} \Delta^{\mathcal{E}}(\mathcal{V})_{n}$, the reverse inclusion being trivial. Each point $x$ of $\Delta^{\mathcal{E}}(\mathcal{V})$ is contained in some simplex \[ \sigma = v_{0} < v_{1} < \ldots < v_{j}.\] We have the corresponding inequalities $|v_{0}| < \ldots < |v_{j}|$, so $x \in \sigma \subseteq \Delta^{\mathcal{E}}(\mathcal{V})_{|v_{j}|}$. 

\end{proof}



\begin{definition} ($\mathbb{S}$-finite expansion scheme)
Let $\mathcal{E}$ be an expansion scheme. Let $[f,D] \in \mathcal{B}$. The group $\mathbb{S}(D,D)$ acts on $\mathcal{E}([f,D])$ by the rule 
\[ h \star v = (fhf^{-1}) \cdot v, \]
where $h \in \mathbb{S}(D,D)$. This action extends to a natural action of
$\mathbb{S}(D,D)$ on the simplicial complex $\Delta(\mathcal{E}([f,D]))$. 

We say that $\mathcal{E}$ is \emph{$\mathbb{S}$-finite} if the above action is cocompact, for each $[f,D] \in \mathcal{B}$.
\end{definition}

\begin{proposition} \label{proposition:cocompact} (sufficient condition for cocompactness)
If the expansion scheme $\mathcal{E}$ is $\mathbb{S}$-finite and $\mathbb{S}$ has finitely many domain types, then the action
of $\Gamma_{S}$ on each subcomplex $\Delta^{\mathcal{E}}(\mathcal{V})_{n}$ is cocompact.
\end{proposition}

\begin{proof}
Since $\mathbb{S}$ has finitely many domain types, there can be only finitely many types of vertices of a fixed height. It therefore follows from Proposition \ref{proposition:vertexorbits} that there are only finitely many $\Gamma_{S}$-orbits of vertices in $\Delta^{\mathcal{E}}(\mathcal{V})_{n}$, for any given $n$. 

To prove that the action of $\Gamma_{S}$ on each $\Delta^{\mathcal{E}}(\mathcal{V})_{n}$ is cocompact, it will suffice to show that $\Delta^{\mathcal{E}}(\mathcal{V})_{n}$ has only finitely many $\Gamma_{S}$-orbits of simplices in each dimension, since $\Delta^{\mathcal{E}}(\mathcal{V})_{n}$ is finite-dimensional by 
Proposition \ref{proposition:findiminv}. We suppose, for a contradiction, that there is a dimension $j$ and an infinite collection $\Sigma$ of $j$-simplices in $\Delta^{\mathcal{E}}(\mathcal{V})_{n}$, no two of which are in the same $\Gamma_{S}$-orbit. Each simplex in $\Sigma$, as a finite chain, must contain a least vertex. Since there are only finitely many $\Gamma_{S}$-orbits of vertices in $\Delta^{\mathcal{E}}(\mathcal{V})_{n}$, there must be an infinite collection $\Sigma'$ of simplices in $\Sigma$ whose least vertices all lie in a single $\Gamma_{S}$-orbit. After passing to $\Sigma'$, and then replacing the simplices in $\Sigma'$ with suitable $\Gamma_{S}$-translates, we get a new infinite collection of $j$-simplices, all in different $\Gamma_{S}$-orbits and having a common least vertex; we will again denote this collection by $\Sigma$. Let
\[ v = \{ [f_{1}, D_{1}], \ldots, [f_{m}, D_{m}]\} \]
be the common least vertex of the simplices in $\Sigma$. We will write $b_{i}$ in place of $[f_{i},D_{i}]$, for brevity's sake. 

For $i = 1, \ldots, m$, let $X_{i}$ denote the set of all sequences
\[  b_{i} \leq b_{1i} \leq \ldots \leq b_{ji}, \]
where $b_{\ell i} \in \mathcal{E}(b_{i})$, for $\ell = 1, \ldots, j$. (We note that the sequences in $X_{i}$ are not necessarily strictly increasing.) Since the action of $\mathbb{S}(D_{i},D_{i})$ on $\Delta(\mathcal{E}(b_{i}))$ is cocompact and $j$ is fixed, the natural action of $\mathbb{S}(D_{i},D_{i})$ on $X_{i}$ has finitely many orbits. We let $\mathcal{O}_{i}$ denote the (finite) set of orbits under this action. 

A given simplex $\sigma \in \Sigma$ determines an element of
$\prod_{i=1}^{m} X_{i}$ as follows. Suppose that $\sigma$ is the chain
\[ v= v_{0} \lneq v_{1} \lneq v_{2} \lneq \ldots \lneq v_{j}. \]
For $i = 1, \ldots, m$, we can uniquely determine an element
\[ b_{0i} \leq b_{1i} \leq \ldots \leq b_{ji} \]
 of $X_{i}$ by letting $b_{\ell i}$ ($\ell = 0, \ldots, j$) denote the subset of $v_{\ell}$ having the same image as $b_{i}$; such a $b_{\ell i}$ exists and is unique by Proposition \ref{proposition:orderlocal}(1). The fact that $b_{\ell i} \in \mathcal{E}(b_{i})$ is a consequence of the definition of $\Delta^{\mathcal{E}}(\mathcal{V})$. 

We directly get a function $f: \Sigma \rightarrow \prod_{i=1}^{m} \mathcal{O}_{i}$. Since the domain of $f$ is infinite and the codomain is finite, $f$ cannot be injective, so we must have two simplices $\sigma, \sigma' \in \Sigma$ such that $f(\sigma) = f(\sigma')$. Thus the associated sequences 
\[ b_{0i} \leq b_{1i} \leq \ldots \leq b_{ji} \quad \text{ and }
\quad b'_{0i} \leq b'_{1i} \leq \ldots \leq b'_{ji} \]
are in the same $\mathbb{S}(D_{i},D_{i})$-orbit, for $i=1, \ldots, m$. Thus, we can find, for each $i$, an element $h_{i} \in \mathbb{S}(D_{i},D_{i})$ such that $f_{i}h_{i}f^{-1}_{i}$ sends the first sequence to the second. If we define $\gamma \in \Gamma_{S}$ by setting
$\gamma_{\mid f_{i}(D_{i})} = f_{i}h_{i}f^{-1}_{i}$ for each $i$, then we have $\gamma \cdot \sigma = \sigma'$, a contradiction to the definition of $\Sigma$. The action of $\Gamma_{S}$ on $\Delta^{\mathcal{E}}(\mathcal{V})_{n}$ is therefore cocompact.  
\end{proof}

\subsection{Finiteness properties of stabilizer subgroups} \label{subsection:fnstab}

Here we will show that, under appropriate hypotheses, the stabilizer group of a simplex in $\Delta^{\mathcal{E}}(\mathcal{V})$ has finite index in a vertex stabilizer group (Proposition \ref{proposition:finiteindex}). Thus, simplex stabilizer groups have good finiteness properties exactly when vertex stabilizer groups have these properties (Corollary \ref{corollary:Fnstabilizers}).

\begin{proposition} \label{proposition:finiteindex}
Let $\mathcal{E}$ be an expansion scheme. Assume that the action of $\mathbb{S}(D,D)$ on $\mathcal{E}([f,D])$ has finite orbits, for each $[f,D] \in \mathcal{B}$. 

If $\sigma$ is a simplex of $\Delta^{\mathcal{E}}(\mathcal{V})$, and $v$ is the least vertex of $\sigma$ in the partial order, then the stabilizer $\left(\Gamma_{S}\right)_{\sigma}$ has finite index in $\left(\Gamma_{S}\right)_{v}$.
\end{proposition}
 
\begin{proof}
Let 
$ v = \{ [f_{1},D_{1}], \ldots, [f_{m},D_{m}] \}$, 
and let $\sigma$ be the $j$-simplex 
\[ v = v_{0} \lneq v_{1} \lneq \ldots \lneq v_{j}. \]
 
We define $X_{i}$ (for $i = 1, \ldots, m$) exactly as in the proof of Proposition \ref{proposition:cocompact}. 
We let $\Gamma'_{v}$ denote the finite index subgroup of $(\Gamma_{S})_{v}$ that fixes $v$ pointwise; i.e.,
for each $\gamma \in \Gamma'_{v}$ and for each $[f_{i},D_{i}] \in v$, 
\[ \gamma \cdot [f_{i},D_{i}] = [f_{i},D_{i}]. \]
It follows that $[\gamma f_{i},D_{i}] = [f_{i},D_{i}]$, so there is $h \in \mathbb{S}(D_{i},D_{i})$ such that $\gamma f_{i} = f_{i} h$. (See the definition of the equivalence relation on $\mathcal{B}$ (Definition \ref{definition:pairs}) and the definition of the action
(Definition \ref{definition:hatSaction}).) It follows that $\gamma_{\mid f_{i}(D_{i})} = (f_{i}hf_{i}^{-1})_{\mid f_{i}(D_{i})}$. We use this equality coordinate-by-coordinate, for $i = 1, \ldots, m$, to determine an action of $\gamma$ on $\prod_{i=1}^{m} X_{i}$, and, thus, an action of $\Gamma'_{v}$ on $\prod_{i=1}^{m} X_{i}$. Since the action of each $\mathbb{S}(D_{i},D_{i})$ on $X_{i}$ has finite orbits, and the action of $\Gamma'_{v}$ factors through the action of $\prod_{i=1}^{m} \mathbb{S}(D_{i},D_{i})$ on $\prod_{i=1}^{m} X_{i}$, the action of $\Gamma'_{v}$ on $\prod_{i=1}^{m} X_{i}$ has finite orbits.  

We let $(x_{1}, x_{2}, \ldots, x_{m}) \in \prod_{i=1}^{m} X_{i}$ be the $m$-tuple that corresponds to $\sigma$. (I.e., $x_{i}$ is the weakly increasing sequence 
\[ b_{0i} \leq b_{1i} \leq b_{2i} \leq \ldots \leq b_{ji}, \] 
where $b_{\ell i}$ is the subset of $v_{\ell}$ such that $im(b_{\ell i}) = f_{i}(D_{i})$, exactly as in the proof of Proposition \ref{proposition:cocompact}.)
Since the action of $\Gamma'_{v}$ on $\prod_{i=1}^{m} X_{i}$ has finite orbits, there is a finite index subgroup $\Gamma''_{v}$ of $\Gamma'_{v}$ that fixes $(x_{1}, \ldots, x_{m})$, and, thus, $\sigma$. It now follows that 
\[ \Gamma''_{v} \leq (\Gamma_{S})_{\sigma} \leq (\Gamma_{S})_{v}, \]
where $[(\Gamma_{S})_{v}: \Gamma''_{v} ] < \infty$. This proves the proposition.\end{proof}

\begin{corollary} \label{corollary:Fnstabilizers}
Let $\mathcal{E}$ be an expansion scheme. Assume that
\begin{enumerate}
\item the action of $\mathbb{S}(D,D)$ on $\mathcal{E}([f,D])$ has finite orbits, for each $[f,D] \in \mathcal{B}$, and
\item each group $\mathbb{S}(D,D)$ has type $F_{n}$.
\end{enumerate}
For every cell $\sigma \subseteq \Delta^{\mathcal{E}}(\mathcal{V})$,  $(\Gamma_{S})_{\sigma}$ is of type $F_{n}$.
\end{corollary}

\begin{proof}
We first consider vertex stabilizers. Let $v = \{ [f_{1},D_{1}], \ldots, [f_{m},D_{m}] \}$ be a vertex. We define a map $\phi: \prod_{i=1}^{m} \mathbb{S}(D_{i},D_{i}) \rightarrow (\Gamma_{S})_{v}$ by sending $(h_{1}, \ldots, h_{m})$ to the bijection of $X$ determined by the rule $\gamma_{\mid f_{i}(D_{i})} = (f_{i}h_{i}f_{i}^{-1})_{\mid f_{i}(D_{i})}$. This assignment is an injective homomorphism. We let $(\Gamma_{S})'_{v}$ denote the image of this homomorphism. We note that $(\Gamma_{S})'_{v}$ has type $F_{n}$, since each $\mathbb{S}(D_{i},D_{i})$ has type $F_{n}$.

By Proposition \ref{proposition:stab}, the index of $(\Gamma_{S})'_{v}$ in $(\Gamma_{S})_{v}$ is finite. Thus, the vertex stabilizer $(\Gamma_{S})_{v}$ has type $F_{n}$, for all $v$. By Proposition \ref{proposition:finiteindex}, every cell stabilizer $(\Gamma_{S})_{\sigma}$ has finite index in some vertex stabilizer, and therefore must have type $F_{n}$ as well.
\end{proof}

\subsection{Generation of expansion schemes} \label{subsection:Egeneration}

In applications, we would like to have a rapid way of producing expansion schemes. Our approach in this subsection will be to consider generating sets for expansion schemes, which we call ``expansion preschemes''. These are analogous to generating sets for groups, or bases for vector spaces. The idea will be to define the expansion scheme on a few pairs of the form $[id_{D},D]$, and then extend uniquely to all pairs in $\mathcal{B}$, but the ``expansion prescheme'' idea will also help to handle all related questions of well-definedness.

\begin{definition} \label{definition:transversal} (transversal of $\mathcal{B}$)
Let $\{ [D_{i}] \mid i \in \mathcal{I} \}$ be the set of all domain types relative to $\mathbb{S}$ (Definition \ref{definition:domaintypes}). Assume that
$\widehat{\mathcal{T}}=\{ D_{i} \mid i \in \mathcal{I} \}$ consists of a selection of exactly one domain $D_{i}$ from each equivalence class $[D_{i}]$. 

Let $\mathcal{T} = \{ [id_{D},D] \mid D \in \widehat{\mathcal{T}} \}$
We say that $\mathcal{T}$ is a \emph{transversal of $\mathcal{B}$}.
\end{definition}

\begin{definition} \label{definition:prescheme} (expansion preschemes)
Let $\mathcal{T}$ be a transversal of $\mathcal{B}$. We say that $\mathcal{E}': \mathcal{T} \rightarrow 2^{\mathcal{PV}}$ is an \emph{expansion prescheme} if it satisfies (1)-(4), for each $[id_{D},D] \in \mathcal{T}$:
\begin{enumerate}
\item Each $w \in \mathcal{E}'([id_{D},D])$ is the result of a sequence of expansions from $\{[id_{D},D]\}$; i.e., for each $w \in \mathcal{E}'([id_{D},D])$, $\{ [id_{D},D] \} \leq w$;
\item $\{ [id_{D},D] \} \in \mathcal{E}'([id_{D},D])$;
\item ($\mathbb{S}$-invariance) for each $h \in \mathbb{S}(D,D)$, $h \cdot \mathcal{E}'([id_{D},D]) = \mathcal{E}'([id_{D},D])$. 
\item Let $w_{1}, w_{2} \in \mathcal{E}'([id_{D},D])$ be arbitrary pseudovertices such that $w_{1} < w_{2}$. Write $w_{1} = \{ b_{1}, \ldots, b_{m} \} \subseteq \mathcal{B}$ and
\[ w_{2} = \bigcup_{i=1}^{m} w_{2i}, \]
where $w_{2i}$ is a pseudovertex with the same image as $\{ b_{i} \}$, for $i=1, \ldots, m$. There exist $g_{i} \in \widehat{S}$ such that $g_{i} \cdot b_{i}\in \mathcal{T}$ 
and
$g_{i} \cdot w_{2i} \in \mathcal{E}'(g_{i} \cdot b_{i})$ for $i=1, \ldots, m$.
\end{enumerate}
\end{definition}

\begin{remark} \label{remark:transversal} We note that, for each $b \in \mathcal{B}$, there is a unique $\hat{b} \in \mathcal{T}$ such that $\hat{s} \cdot b = \hat{b}$, for some $\hat{s} \in \widehat{S}$. (The element $\hat{s}$ is not unique, however.)
\end{remark}

\begin{proposition}
\label{proposition:generationofschemes} (expansion preschemes generate expansion schemes)
Let $\mathcal{E}'$ be an expansion prescheme. For each $[g,E] \in \mathcal{B}$, choose some $\alpha \in \widehat{S}$ such that $\alpha \cdot [g,E] = [id_{D},D] \in \mathcal{T}$. Define $\mathcal{E}: \mathcal{B} \rightarrow 2^{\mathcal{PV}}$ by the rule
\[ \mathcal{E}([g,E]) = \alpha^{-1} \cdot \mathcal{E}'([id_{D},D]).\]
The assignment $\mathcal{E}$ is an expansion scheme and does not depend upon the choices of $\alpha$.
\end{proposition}

\begin{proof}
We first check that the assignment $\mathcal{E}$ is well-defined.
Thus, assume that $[g_{1},E_{1}] = [g_{2},E_{2}]$, and choose elements $\alpha, \beta \in \widehat{S}$ such that
$\alpha \cdot [g_{1},E_{1}] = [id_{D},D] = \beta \cdot [g_{2},E_{2}]$, where $[id_{D},D] \in \mathcal{T}$. (We note that $E_{1}$ and $E_{2}$ have the same domain type, so if
$\alpha \cdot [g_{1},E_{1}] \in \mathcal{T}$ and
$\beta \cdot [g_{2},E_{2}] \in \mathcal{T}$, then the latter two must be equal.) It follows from the definition of the equivalence relation and the action (Definitions \ref{definition:pairs} and \ref{definition:hatSaction}, respectively) that $\alpha g_{1} = h$, $\beta g_{2} = j$, and $g_{1}=g_{2}k$, for some $h \in \mathbb{S}(E_{1},D)$, $j \in \mathbb{S}(E_{2},D)$, and $k \in \mathbb{S}(E_{1},E_{2})$.

We must show that $\alpha^{-1} \cdot \mathcal{E}'([id_{D},D]) = \beta^{-1} \cdot \mathcal{E}'([id_{D},D])$, or, equivalently, that $(\beta \alpha^{-1}) \cdot \mathcal{E}'([id_{D},D]) = \mathcal{E}'([id_{D},D])$. This follows from a direct calculation:
\begin{align*}
\beta \alpha^{-1} &= j g_{2}^{-1}g_{1}h^{-1} \\
&= j g_{2}^{-1} g_{2}k h^{-1} \\
&= jkh^{-1}. \
\end{align*}
By the definition of $j$, $k$, and $h$, and since the structure sets are closed under compositions and inverses (properties (S4) and (S3) from Definition \ref{definition:sstructure}, respectively), $jkh^{-1} \in \mathbb{S}(D,D)$. The required equality $(\beta \alpha^{-1}) \cdot \mathcal{E}'([id_{D},D]) = \mathcal{E}'([id_{D},D])$ is now a consequence of property (3) from Definition \ref{definition:prescheme}. This proves that $\mathcal{E}$ is well-defined and does not depend upon the specific choice of $\alpha$.

Properties (1) and (2) from Definition \ref{definition:scheme} hold trivially.   
We check property (3). Thus, suppose that $[g,E] \in \mathcal{B}$ and that the domain of $\hat{s} \in \widehat{S}$ contains $g(E)$. Let $\alpha \in \widehat{S}$ be such that $\alpha \cdot [g,E] = [id_{D},D]$, where $[id_{D},D] \in \mathcal{T}$.   
We have the equalities
\[ \mathcal{E}(\hat{s} \cdot [g,E]) = \mathcal{E}([\hat{s}g,E]) =
(\hat{s}\alpha^{-1}) \cdot \mathcal{E}'([id_{D},D]) = \hat{s} \cdot \mathcal{E}([g,E]), \]
where the second equality follows by letting $\alpha \hat{s}^{-1}$ play the role of $\alpha$ in the definition of $\mathcal{E}([\hat{s}g,E])$. Property (3) follows.

We next check property (4) from Definition \ref{definition:scheme}. Thus, suppose that $w_{1} \leq w_{2}$, where $w_{1}, w_{2} \in \mathcal{E}([g,E])$. Clearly, we can assume that $w_{1} < w_{2}$. We must show that $w_{2}$ is the result of $\mathcal{E}$-expansion from $w_{1}$. We write $w_{1} = \{ b_{1}, \ldots, b_{m} \} \subseteq \mathcal{B}$. By Proposition \ref{proposition:orderlocal}(1), we can write
\[ w_{2} = \{ w_{21}, \ldots, w_{2m} \}, \]
where each $w_{2i}$ is a pseudovertex with the same image as $\{ b_{i} \}$. For $i=1, \ldots, m$, there exist $g_{i} \in \widehat{S}$ such that $g_{i} \cdot b_{i} \in \mathcal{T}$ and 
$g_{i} \cdot w_{2i} \in \mathcal{E}'(g_{i} \cdot b_{i})$. It follows directly that
\[ w_{2i} \in g_{i}^{-1} \cdot \mathcal{E}'(g_{i} \cdot b_{i}) = \mathcal{E}(b_{i}), \]
for $i=1, \ldots, m$, establishing property (4).
\end{proof}

\begin{remark} \label{remark:nconnected} We note that the property of $n$-connectedness (Definition \ref{definition:nconnectedE}) can be established by checking the relevant property on the sets $\mathcal{E}'([id_{D},D])$. The details are straightforward and will be omitted.
\end{remark}

\subsection{Examples of expansion schemes} \label{subsection:Eexamples}

In this subsection, we offer some general classes of expansion schemes. In \ref{subsub:um}, we consider expansion schemes in the case when $\mathcal{D}_{S}$ satisfies the compact ultrametric condition (Definition \ref{definition:umetricproperty}), provided also that the $S$-structure in question is the maximal one. In \ref{subsub:Eprod}, we consider expansion schemes on compact ultrametric products. Finally, in \ref{subsub:R}, we give an expansion scheme for R\"{o}ver's group. 

\subsubsection{The case in which $X$ is a compact ultrametric space and $(\mathbb{S},\mathbb{P})$ is maximal} \label{subsub:um}

\begin{proposition} \label{proposition:expansionforcompactumetric} Let $S$ be an inverse semigroup acting on $X$ such that the set $\mathcal{D}^{+}_{S}$ satisfies the compact ultrametric property. Let $(\mathbb{S},\mathbb{P})$ be the maximal $S$-structure.

The assignment 
\[ \mathcal{E}([f,D]) = \{ \{ [f,D] \}, 
\{ [f,E] \mid E \in \mathcal{P}_{D} \} \} \]
is an $n$-connected expansion scheme, for all $n$. Note that 
$\mathcal{P}_{D}$ is the maximal partition of $D$ (Definition \ref{definition:maximal}).
\end{proposition}

\begin{proof}
Choose a transversal $\mathcal{T} \subseteq \mathcal{B}$. For each $[id,D] \in \mathcal{T}$, define
\[ \mathcal{E}'([id,D]) = \{ \{ [id,D] \}, \{ [id,E] \mid E \in \mathcal{P}_{D} \} \}. \]
We check that the assignment $\mathcal{E}'$ is an expansion prescheme. Properties (1) and (2) from Definition \ref{definition:prescheme} are clear.  If $h \in \mathbb{S}(D,D)$, then
\begin{align*}
h \cdot \{ [id,E] \mid E \in \mathcal{P}_{D} \} &=
\{ [h,E] \mid E \in \mathcal{P}_{D} \} \\
&= \{ [id, h(E)] \mid E \in \mathcal{P}_{D} \} \\
&= \{ [id,E] \mid E \in \mathcal{P}_{D} \}.
\end{align*}
Here the second-to-last equality is due to the fact that $h_{\mid E} \in \mathbb{S}(E,h(E))$ (by maximality of the $S$-structure), and the final equality is due to Proposition \ref{proposition:Sinvarianceofmaximalsubdivision}. Clearly 
\[ h \cdot \{ [id,D] \} = \{ [id,D] \} \in \mathcal{E}'([id,D]);\] this proves (3).

Next we prove (4) from Definition \ref{definition:prescheme}. Let $w_{1} \leq w_{2}$, where $w_{1}, w_{2} \in \mathcal{E}'([id,D])$ and $[id,D] \in \mathcal{T}$. We may assume that $w_{1} \neq w_{2}$, since there is nothing to prove otherwise. Thus, $w_{1} = \{ [id,D] \}$
and $w_{2} = \{ [id,E] \mid E \in \mathcal{P}_{D} \}$. We must find an $f \in \widehat{S}$ such that $f \cdot [id,D] \in \mathcal{T}$ and
\[ f \cdot \{ [id,E] \mid E \in \mathcal{P}_{D} \} \in \mathcal{E}'(f \cdot [id,D]). \]
We can let $f = id_{D}$. This proves (4).

Next, we extend equivariantly to obtain an expansion scheme, as in Proposition \ref{proposition:generationofschemes}. Thus, for a given $[\hat{f}, \widehat{D}] \in \mathcal{B}$, we define
\[ \mathcal{E}([\hat{f}, \widehat{D}]) = \alpha^{-1} \cdot \mathcal{E}'([id,\widetilde{D}]), \]
where $[id,\widetilde{D}] \in \mathcal{T}$ and $\alpha \in \widehat{S}$ is such that $\alpha \cdot [\hat{f},\widehat{D}] = [id,\widetilde{D}]$. We must next show that this expansion scheme is the same as the one defined in the statement of the current proposition. Note that, by the definition of the equivalence relation on $\mathcal{B}$, there is some $h \in \mathbb{S}(\widetilde{D},\widehat{D})$ such that $\hat{f} h = \alpha^{-1}$. We compute:

\begin{align*}
\mathcal{E}([\hat{f},\widehat{D}]) &= \alpha^{-1} \cdot
\mathcal{E}'([id,\widetilde{D}]) \\
&= \{ \{ [\alpha^{-1}, \widetilde{D}] \}, 
\{ [\alpha^{-1}, \widetilde{E}] \mid \widetilde{E} \in \mathcal{P}_{\widetilde{D}} \} \} \\
&= \{ \{ [\hat{f}h, \widetilde{D}] \}, 
\{ [\hat{f}h, \widetilde{E}] \mid \widetilde{E} \in \mathcal{P}_{\widetilde{D}} \} \} \\
&= \{ \{ [\hat{f}, \widehat{D}] \}, 
\{ [\hat{f}, \widehat{E}] \mid \widehat{E} \in \mathcal{P}_{\widehat{D}} \} \}. \\
\end{align*} 
The final equality uses Proposition \ref{proposition:Sinvarianceofmaximalsubdivision}. This proves that the assignment from the statement of the proposition is an expansion scheme.

It remains to show that the expansion scheme is $n$-connected, for each $n \geq 0$. Thus, we let $w_{2} > \{ [f,D] \} = w_{1}$, and consider the link of $w_{1}$ in the simplicial complex
\[ \Delta(\mathcal{E}([f,D]))_{[w_{1},w_{2}]}. \]
The simplicial complex
$\Delta(\mathcal{E}([f,D]))$ consists of two vertices connected by an edge. Since every expansion from $\{ [f,D] \}$
factors through $\{ [f,E] \mid E \in \mathcal{P}_{D} \}$ (Example \ref{example:basicexpansion}), we have that
$\{ [f,E] \mid E \in \mathcal{P}_{D} \} \leq w_{2}$, so 
\[ \Delta(\mathcal{E}([f,D])_{[w_{1},w_{2}]} = 
\Delta(\mathcal{E}([f,D]). \]
It follows that the link in question is always a point, which is $(n-1)$-connected for all $n$. 
\end{proof}

\subsubsection{Expansion schemes on compact ultrametric products}
\label{subsub:Eprod}

\begin{definition} \label{definition:specialpartitions}(some special partitions) For $i=1, \ldots, n$, let $S_{i}$ be an inverse semigroup acting on $X_{i}$.
Assume that each $\mathcal{D}^{+}_{S_{i}}$ satisfies the compact ultrametric property.
 
Fix a domain $D = D_{1} \times \ldots \times D_{n} \in \mathcal{D}^{+}_{S_{(1,\ldots,n)}}$. For a subset $U \subset \{ 1, \ldots, n \}$, define
\[ \mathcal{P}_{D_{U}} = \{ E_{1} \times \ldots \times E_{n} \mid
E_{i} = D_{i} \text{ if }i \notin U; E_{i} \in \mathcal{P}_{D_{i}} \text{ if }i\in U \}. \]
Recall that $\mathcal{P}_{D}$ denotes the maximal partition of $D$; see Definition \ref{definition:maximal}.
\end{definition}

\begin{proposition} \label{proposition:expansionschemeproduct}
(an expansion scheme for compact ultrametric products) 
Let $X_{i}$ be a set, and let $S_{i}$ be an inverse semigroup acting on $X_{i}$, for $i=1, \ldots, n$. We consider the action of $S_{(1,\ldots,n)}$ on the product $X_{1} \times \ldots \times X_{n}$. Let the $S_{(1,\ldots,n)}$-structure $(\mathbb{S},\mathbb{P})$ be as given in Definition 
\ref{definition:structuresonproducts}.

We define $\mathcal{E}: \mathcal{B} \rightarrow 2^{\mathcal{PV}}$ as follows:
\[ \mathcal{E}([f,D]) = \{ \{ [f,E] \mid E \in \mathcal{P}_{D_{U}}\} \mid U \subseteq \{ 1, \ldots, n \} \}. \]
The map $\mathcal{E}$ is an $m$-connected expansion scheme, for all $m$.
\end{proposition}

\begin{proof}
The proof is similar to that of Proposition \ref{proposition:expansionforcompactumetric}, which is a special case of the current proposition. We will therefore omit most details, considering only the question of $m$-connectivity for $m \geq 0$.

The complexes $\Delta(\mathcal{E}([f,D]))$ are subdivided cubes of dimension at most $n$. (The dimension is exactly $n$ unless there is a factor $D_{i}$ such that $D_{i}$ has no proper non-empty subdomains; such a factor does not contribute to the dimension. We will assume that the dimension is exactly $n$ for the sake of this discussion.)  The pseudovertices in $\mathcal{E}([f,D])$ label the corners of an $n$-cube by the following rule: a pseudovertex
\[ v_{U} = \{ [f,E] \mid E \in \mathcal{P}_{D_{U}} \} \]
 corresponds to the $n$-tuple $(a_{1}, \ldots, a_{n})$, where $a_{i} = 1$ if $i \in U$ and $a_{i} = 0$ if not.   

If $v_{\emptyset} = \{ [f,D] \} \leq v_{2}$, then 
\[ v_{2} = \{ [f,E] \mid E \in \mathcal{P} \}, \]
for some $\mathcal{P} \in \mathbb{P}(D)$. It is straightforward to check that 
if $v_{U_{1}} \leq v_{2}$ and $v_{U_{2}} \leq v_{2}$, then 
$v_{U_{1} \cup U_{2}} \leq v_{2}$. Moreover, there is at least one $U \neq \emptyset$ such that $v_{U} \leq v_{2}$, by the choice of $\mathbb{P}$. (Here the careful choice of $\mathbb{P}$ avoids the pathologies indicated in Example \ref{example:pathology}.)
It follows directly that the link of $v_{\emptyset}$ in 
\[ K = \Delta(\mathcal{E}([f,D]))_{[v_{\emptyset},v_{2}]} \]
is a directed set. (In fact, we can say more: $K$ is a subdivided face of the $n$-dimensional cube, and the link in question is the link of $v_{\emptyset}$ in that face.) 

Thus, the link in question is always contractible, completing the proof that $\mathcal{E}$ is $m$-connected for all $m$.
\end{proof}

\subsubsection{An expansion scheme for R\"{o}ver's group} \label{subsub:R}

\begin{proposition} \label{proposition:Roeverexpansion}
(an expansion scheme for R\"{o}ver's group)
Let $X$ be the set of infinite binary strings and let $S_{R}$ be the semigroup defined in Example \ref{example:Rover}. We let $\mathbb{S}$ be the $S_{R}$-structure defined in Example \ref{example:Roever}. 

The assignment
\[ \mathcal{E}([f,B_{\omega}]) = 
\{ \{[f,B_{\omega}]\}, \{ [f,B_{\omega0}], [f,B_{\omega1}] \}, \{ [fa_{\omega0,\omega0},B_{\omega0}], [f,B_{\omega1}] \}, \] \[ \{ [f,B_{\omega00}], [f,B_{\omega01}], [f,B_{\omega1}] \} \}, \]
is an $n$-connected expansion scheme for all $n$.
\end{proposition}

\begin{proof}
The singleton set $\{ [id_{X},X] \}$ is a transversal. We will first show that the assignment
\begin{align*}
\mathcal{E}'([id_{X},X]) = &\{ \{ [id,X] \}, \{ [id,B_{0}], [id,B_{1}] \}, \{ [a_{0,0},B_{0}], [id,B_{1}] \}, \\ 
 &\{ [id,B_{00}], [id,B_{01}], [id,B_{1}] \} \} 
\end{align*}
is an expansion prescheme.

Indeed, properties (1) and (2) from Definition \ref{definition:prescheme} are clear. (Refer to the discussion of the expansion relation in Example \ref{example:Roever}.) Property (3) follows from an easy calculation: we must verify that the group
\[ \mathbb{S}(X,X) = \{ 1, b, c, d \} \]
leaves the set $\mathcal{E}'([id,X])$ invariant. We omit the details of the calculation, but summarize the results: 
\begin{itemize}
\item the pseudovertex $\{ [id,X] \}$ is stabilized by $\{ 1, b, c, d \}$;
\item the pseudovertices $\{ [id,B_{0}], [id,B_{1}] \}$ 
and $\{ [a_{0,0},B_{0}], [id,B_{1}] \}$ are interchanged by the elements $b$, and $c$, but are each stabilized by $1$ and $d$;
\item the pseudovertex $\{ [id,B_{00}], [id,B_{01}], [id,B_{1}] \}$ is also stabilized by $\{ 1, b, c, d \}$, but $b$ and $c$ non-trivially permute the individual elements.
\end{itemize}
This establishes property (3).

Property (4) is also easy to check. Given $w_{1} < w_{2}$, where $w_{1}, w_{2} \in \mathcal{E}'([id,X])$, we must produce the relevant $g_{i} \in \widehat{S}$ from Definition \ref{definition:prescheme}(4). This is trivial if $w_{1} = \{ [id,X] \}$ (we can simply let $g_{1} = id$), which leaves only two cases to consider:
\begin{enumerate}
\item $w_{1} = \{ [id,B_{0}], [id,B_{1}] \}$
and $w_{2} = \{ [id,B_{00}], [id,B_{01}], [id,B_{1}] \}$;
\item $w_{1} = \{ [a_{0,0},B_{0}], [id,B_{1}] \}$
and $w_{2} = \{ [id,B_{00}], [id,B_{01}], [id,B_{1}] \}$.
\end{enumerate}
If we assign the labels $b_{1}$ and $b_{2}$ to the elements of $w_{1}$ (respectively, in the order that they are listed), then we can let $g_{1} = \sigma_{0,\epsilon}$ and $g_{2} = \sigma_{1,\epsilon}$ in (1), and let $g_{1} = a_{0,\epsilon}$ and $g_{2} = \sigma_{1,\epsilon}$ in (2). An easy check then establishes (4) from Definition \ref{definition:prescheme}. It follows that the function $\mathcal{E}'$
is an expansion prescheme.

We conclude that the assignment 
\[ \mathcal{E}([f,B_{\omega}]) = f \sigma_{\epsilon,\omega} \cdot
\mathcal{E}'([id,X]) \]
is an expansion scheme, by Proposition \ref{proposition:generationofschemes}. The latter assignment is easily seen to be equivalent to the one in the statement of the proposition.

Next we turn to a proof that the expansion scheme is $n$-connected for all $n$. Note that the simplicial complex $\Delta(\mathcal{E}'([id,X]))$ is a subdivided square $[0,1]^{2} \subseteq \mathbb{R}^{2}$, in which we may take $\{ [id,X] \}$ to label $(0,0)$, the two pseudovertices of rank two to label $(1,0)$ and $(0,1)$, and the remaining pseudovertex $v_{T}$ to label $(1,1)$. If $\{ [id,X] \} < v_{2}$, then $v_{2}$ is the result of performing a sequence of simple expansions from $\{ [id,X] \}$. The first of these expansions must result in either $v_{L} = \{ [id,B_{0}], [id,B_{1}] \}$
or $v_{R} = \{ [a_{0,0},B_{0}], [id,B_{1}] \}$, by Example \ref{example:Roever}. It follows that the link of $v_{1} = \{ [id,X] \}$ in
\[ \Delta(\mathcal{E}([id,X]))_{[v_{1},v_{2}]} \} \]
is non-empty. 

To prove that $\mathcal{E}$ is $n$-connected for all $n$, it is now sufficient to show that if $v_{L} \leq v_{2}$ and $v_{R} \leq v_{2}$, then $v_{T} \leq v_{2}$. This is a consequence of the following claim.

\begin{claim}
For every $[f,B_{\omega}] \in \mathcal{B}$ and for every common upper bound $v_{1}$ of $\{ [f,B_{\omega}] \}$ and $\{ [fa_{\omega,\omega},B_{\omega}] \}$, there is another common upper bound $\hat{v}$
such that $\hat{v} \leq v_{1}$ and $\hat{v}$ may be obtained from 
$\{ [f,B_{\omega}] \}$ using only standard simple expansions. 
\end{claim}

\begin{proof}[Proof of claim]
We prove this by induction on the rank of the common upper bound $v_{1}$. The case $r(v_{1}) = 1$ is vacuous; the case $r(v_{1}) = 2$ is trivial, since the only common upper bound of that rank is the standard simple expansion from $\{ [f,B_{\omega}] \}$. 

Assume the inductive hypothesis. We can write $v_{1} = v_{L} \cup v_{R}$, where $v_{L}$ is the subset of $v_{1}$ whose image is $f(B_{\omega0})$ and $v_{R}$ is the subset of $v_{1}$ whose image is $f(B_{\omega1})$. Clearly, we can find simple expansions from $\{ [f,B_{\omega}] \}$ and $\{ [fa_{\omega,\omega},B_{\omega}] \}$ that are less than $v_{1}$; one or the other of these simple expansions must be non-standard. This leads to three separate cases; we will consider the case in which both simple expansions are non-standard, the other two cases being similar. 

Thus, we have
\[ \{ [fa_{\omega0,\omega0},B_{\omega0}], [f,B_{\omega1}] \} \leq v_{1} \]
and
\[ \{ [fa_{\omega,\omega}a_{\omega0,\omega0},B_{\omega0}], [fa_{\omega,\omega},B_{\omega1}] \} \leq v_{1}. \]
We note that $a_{\omega,\omega}$ restricted to $B_{\omega1}$ is
$\sigma_{\omega1,\omega0}$ and $a_{\omega,\omega}a_{\omega0,\omega0}$ restricted to $B_{\omega0}$ is $a_{\omega0,\omega1} = a_{\omega1,\omega1}\sigma_{\omega0,\omega1}$, so, using the definition of the equivalence relation we can rewrite the latter as
\[ \{ [fa_{\omega1,\omega1},B_{\omega1}], [f,B_{\omega0}] \} \leq v_{1}. \]
The inductive hypothesis now implies that there is a common upper bound 
$v'_{L} \leq v_{L}$ of $\{ [f,B_{\omega0}] \}$ and $\{ [fa_{\omega0,\omega0},B_{\omega0}] \}$ that is obtained by standard simple expansions from $\{ [f,B_{\omega0}] \}$ and (likewise) a common upper bound $v'_{R} \leq v_{R}$ of $\{ [f,B_{\omega1}] \}$ and $\{ [fa_{\omega1,\omega1},B_{\omega1}]$ that is obtained by simple expansions from $\{[f,B_{\omega1}]\}$. Let $v' = v'_{L} \cup v'_{R}$; clearly $v'$ is obtained from $\{ [f,B_{\omega}] \}$ via standard simple expansions and $v' \leq v_{1}$, proving the claim.
\end{proof}
 
\end{proof}

\section{Preliminaries for finiteness properties} \label{section:7}

In this section, we will prepare for the finiteness results of the paper, the proofs of which are completed in Section \ref{section:8}. In Subsection \ref{subsection:Brown}, we recall Brown's Finiteness Criterion, which is the foundation of all of our finiteness results. Subsection \ref{subsection:gendesclink} recalls some very standard results about the descending link; in particular, we relate the increasing connectivity of the descending links with the increasing connectivity of the complexes in the natural filtration. Subsections \ref{subsection:analysisI} and \ref{subsection:analysisII} develop the necessary definitions that are to be used in studying the descending links in the complexes $\Delta^{\mathcal{E}}(\mathcal{V}_{\mathbb{S}})$. The section ends with a combinatorial sufficient condition for the descending links to be $n$-connected (Proposition \ref{proposition:downwardlinkconnectivity}).


\subsection{Brown's Finiteness Criterion} \label{subsection:Brown}
We will now briefly review Brown's Finiteness Criterion. First, a basic definition:

\begin{definition} (Properties $F_{n}$ and $F_{\infty}$) Let $G$ be a group. By a \emph{$K(G,1)$-complex} we mean a CW-complex $X$ with fundamental group $G$ and contractible universal cover. We say that $G$ has \emph{type $F_{n}$} if $G$ admits a $K(G,1)$-complex with finite $n$-skeleton. We say that $G$ has \emph{type $F_{\infty}$} if $G$ admits a $K(G,1)$-complex with finite $n$-skeleton for each $n$.
\end{definition}

\begin{remark} The above definition appears to suggest that the $F_{\infty}$ property is strictly stronger than the property of being $F_{n}$ for all $n$,
since the former condition requires a single complex with a finite $n$-skeleton for each $n$, while the latter condition allows a different complex for each $n$. The two properties are nevertheless equivalent; a proof may be found, for instance, in \cite{BookbyRoss}. 

Additionally, we note that $F_{0}$ is a property of every group, $F_{1}$ is equivalent to finite generation, and 
$F_{2}$ is equivalent to finite presentability. These facts are also standard, and may be found in \cite{BookbyRoss}.
\end{remark}
  
\begin{theorem} \label{theorem:Brown} (Brown's Finiteness Criterion)
Let $X$ be a CW-complex. Let $G$ be a group acting on $X$. If
\begin{enumerate}
\item $X$ is $(n-1)$-connected; 
\item $G$ acts cellularly on $X$, and
\item there is a filtration $X_{1} \subseteq X_{2} \subseteq \ldots \subseteq X_{k} \subseteq \ldots \subseteq X$
such that
\begin{enumerate}
\item $X = \bigcup_{k=1}^{\infty} X_{k}$;
\item $G$ leaves each $X_{k}^{(n)}$ invariant and acts cocompactly on each $X_{k}^{(n)}$; 
\item each $p$-cell stabilizer has type $F_{n-p}$, and
\item for sufficiently large $k$, $X_{k}$ is $(n-1)$-connected,
\end{enumerate}
\end{enumerate}
then $G$ is of type $F_{n}$. \qed
\end{theorem}

\begin{remark}
Our goal will be to apply Brown's criterion to the action of $\Gamma_{S}$ on its complex $\Delta^{\mathcal{E}}(\mathcal{V})$. The required connectivity of the complexes $\Delta^{\mathcal{E}}(\mathcal{V})$ can be established using the results of Subsection \ref{subsection:nconnectedE}. 
The finiteness properties of cell stabilizers were considered in Subsection \ref{subsection:fnstab}, while the cocompactness of the action of $\Gamma_{S}$ on $\Delta^{\mathcal{E}}(\mathcal{V})_{n}$ was handled in Subsection
\ref{subsection:Efilt}.

Thus, it remains only to consider the connectivity properties of the subcomplexes $\Delta^{\mathcal{E}}(\mathcal{V})_{n}$, which is the subject of the rest of this section.
\end{remark}

\subsection{Generalities about the descending link} \label{subsection:gendesclink}

We next consider the general problem of establishing the $n$-connectivity of the subcomplexes $\Delta^{\mathcal{E}}(\mathcal{V}_{\mathbb{S}})_{k}$, for suitable
$n$ and $k$. Here we will follow a well-known strategy, which involves reducing the entire question to an analysis of the descending links in the complex $\Delta^{\mathcal{E}}(\mathcal{V}_{\mathbb{S}})$. We offer a complete treatment in order to make our account self-contained. 
 Note that the actual analysis of specific descending links will be pursued later, under additional hypotheses.
 
 We will write $\Delta$ and $\Delta_{k}$ in place of $\Delta^{\mathcal{E}}(\mathcal{V}_{\mathbb{S}})$ and $\Delta^{\mathcal{E}}(\mathcal{V}_{\mathbb{S}})_{k}$, respectively. See Definition \ref{definition:linkandstar} for the definition of descending link.

\begin{lemma} \label{lemma:homology}
If $lk_{\downarrow}(v,\Delta)$ is $n$-connected for all vertices $v$ of rank $k$, then the map between homology groups
\[ \iota_{j}: H_{j}(\Delta_{k-1}) \rightarrow H_{j}(\Delta_{k}) \]
is an isomorphism, for $j=0, \ldots, n$.
\end{lemma}

\begin{proof}
Consider the long exact sequence in homology of the pair $(\Delta_{k}, \Delta_{k-1})$. If the relative groups $H_{i}(\Delta_{k}, \Delta_{k-1})$ are all $0$ for $i=0, \ldots, n+1$, then, by exactness, the map $\iota_{j}$ is an isomorphism, for 
$j=0, \ldots, n$, as desired. It therefore suffices to show that the relative groups are all $0$ through the given range.

Around each vertex $v$ of rank $k$, choose an open $\epsilon$-ball $B_{\epsilon}(v) \subseteq \Delta_{k}$ that contains no other vertices. The boundary of this ball is homeomorphic to $lk_{\downarrow}(v,\Delta)$, so the closed ball $\overline{B}_{\epsilon}(v)$ is homeomorphic to the cone on $lk_{\downarrow}(v,\Delta)$. Let $L$ denote the complement of the union of the open balls. The inclusion $(\Delta_{k}, \Delta_{k-1}) \rightarrow (\Delta_{k}, L)$ is a homotopy equivalence of pairs, since $L$ strong deformation retracts on $\Delta_{k-1}$ by radial projection from the vertices of rank $k$. The inclusion of pairs
\[ ( \coprod \overline{B}_{\epsilon}(v), \coprod \dot{B}_{\epsilon}(v)) \rightarrow (\Delta_{k}, L), \] 
obtained by removing the interior of $L$, is an excision, so we get isomorphisms
\[ \bigoplus_{v} H_{\ast}(\overline{B}_{\epsilon}(v), \dot{B}_{\epsilon}(v)) \cong H_{\ast}(\Delta_{k}, \Delta_{k-1})\] 
in all dimensions. (Here $\dot{B}$ denotes the boundary of $\overline{B}$; the direct sum is over all vertices of rank $k$.) Next consider the long exact sequence in reduced homology for the pair $(\overline{B}_{\epsilon}(v), \dot{B}_{\epsilon}(v))$. Since $\overline{B}_{\epsilon}(v)$ is contractible, we get isomorphisms
\[ H_{j}(\overline{B}_{\epsilon}(v), \dot{B}_{\epsilon}(v)) \cong \widetilde{H}_{j-1}(\dot{B}_{\epsilon}(v)) \cong
\widetilde{H}_{j-1}(lk_{\downarrow}(v,\Delta)) \]
for all $j \geq 1$. Since $lk_{\downarrow}(v,\Delta)$ is $n$-connected, it follows that 
$H_{j}(\overline{B}_{\epsilon}(v), \dot{B}_{\epsilon}(v)) = 0$, 
for $j = 0, \ldots, n+1$. It now follows that $H_{j}(\Delta_{k}, \Delta_{k-1}) = 0$ for $j=0, \ldots, n+1$ by the above direct sum decomposition, completing the proof. 
\end{proof}

\begin{proposition} \label{proposition:connectfilt}
(connectivity of the filtration) Assume that $\Delta$ is $n$-connected.
If the descending links $lk_{\downarrow}(v,\Delta)$ of all vertices of rank at least $k$ are $n$-connected, then the maps on homotopy groups
\[ \pi_{j}(\Delta_{\ell - 1}) \rightarrow \pi_{j}(\Delta_{\ell}) \] 
are isomorphisms, for $j=0, \ldots, n$ and $\ell \geq k$. 

In particular, $\Delta_{\ell-1}$ is $n$-connected, for $\ell \geq k$.
\end{proposition}

\begin{proof}
We prove both statements by induction on $n$. If $n=0$, then Lemma \ref{lemma:homology} implies that each inclusion
$\Delta_{\ell-1} \rightarrow \Delta_{\ell}$ is bijective at the level of path components, provided $\ell \geq k$. This establishes the first conclusion. It follows that
\[ \pi_{0}(\Delta_{k-1}) \rightarrow \pi_{0}(\Delta_{k}) \rightarrow \ldots \rightarrow \pi_{0}(\Delta) \]
is a sequence of bijections. Thus, the inclusion $\Delta_{\ell} \rightarrow \Delta$ determines a bijection of path components for $\ell \geq k$, so $\Delta_{\ell}$ is path connected. This completes the base case of the induction. 

If $n=1$, then, by induction, $\Delta_{\ell-1}$ is path-connected, for all $\ell \geq k$. We recall from the proof of Lemma \ref{lemma:homology} that the complex $\Delta_{\ell}$ is obtained by attaching cones to $\Delta_{\ell-1}$ along their bases, where each such base is the descending link of a vertex of rank $\ell$. By hypothesis, all such bases are $1$-connected for $\ell \geq k$, so van Kampen's theorem and a direct limit argument yield an isomorphism
\[ \pi_{1}(\Delta_{\ell-1}) \rightarrow \pi_{1}(\Delta_{\ell}), \]
for each $\ell \geq k$. The fact that $\Delta$ itself is $1$-connected now implies that $\Delta_{\ell-1}$ is $1$-connected, for $\ell \geq k$, as desired.

The general case $n>1$ follows easily by induction, Lemma \ref{lemma:homology}, and the Hurewicz theorem.
\end{proof}

\subsection{The descending link: the partitioned downward star and link} \label{subsection:analysisI}

For the remainder of the section, we will let $\Delta$ denote any of the complexes $\Delta^{\mathcal{E}}(\mathcal{V}_{\mathbb{S}})$ or $\Delta^{\mathcal{E}}(\mathcal{PV}_{\mathbb{S},Y})$, for $Y$ an arbitrary finite disjoint union of domains. The precise identity of $\Delta$ can be readily determined from the context. In a few cases (e.g., Proposition \ref{proposition:partitionedstar} and Corollary \ref{corollary:partitionedlink}), we suppress any mention of $\Delta$, in order to avoid using $\Delta$ to mean two different things within the same formula.

Here we introduce the partitioned downward stars and links, which are fundamental to our analysis of the downward links. The partitioned downward stars have natural product decompositions (Proposition \ref{proposition:partitionedstar}) and, therefore, the partitioned downward links have natural join structures (Corollary \ref{corollary:partitionedlink}).  

\begin{definition}  \label{definition:downwardpartitionedcomplex} (The downward star and downward link of a partitioned pseudovertex) Let 
\[ v = \{ [f_{1},D_{1}], \ldots, [f_{m},D_{m}] \} \]
 be a pseudovertex.
Let $\mathcal{P} = \{ p_{1}, \ldots, p_{\ell} \}$ be a partition of $v$,
where each $p_{i}$ is a pseudovertex. We denote such a choice of $v$ and partition $\mathcal{P}$ by $v_{\mathcal{P}}$, and call $v_{\mathcal{P}}$ a \emph{partitioned pseudovertex}.

We let $st_{\downarrow}(v_{\mathcal{P}},\Delta)$ denote the full subcomplex of
$st_{\downarrow}(v,\Delta)$ spanned by pseudovertices 
\[ v' = \{ [g_{1},E_{1}], \ldots, [g_{t},E_{t}] \}, \]
such that, for $j = 1, \ldots, t$, $g_{j}(E_{j}) \subseteq im(p_{i})$ for some $i \in \{ 1, \ldots, \ell \}$. The subcomplex $st_{\downarrow}(v_{\mathcal{P}},\Delta)$ is the \emph{downward star of the partitioned pseudovertex $v_{\mathcal{P}}$ in $\Delta$}. The 
\emph{downward link of $v_{\mathcal{P}}$ in $\Delta$}, denoted $lk_{\downarrow}(v_{\mathcal{P}},\Delta)$, is the link of $v$ in $st_{\downarrow}(v_{\mathcal{P}},\Delta)$.
\end{definition}

\begin{proposition} \label{proposition:partitionedstar} (product decomposition of the partitioned downward link)
Let $v \in \Delta$ and let $\mathcal{P} = \{ p_{1}, \ldots, p_{\ell} \}$ be a partition of $v$, where the $p_{i}$ are pseudovertices, for $i=1, \ldots, \ell$. We have 
\[ st_{\downarrow}(v_{\mathcal{P}}) \cong \prod_{i=1}^{\ell} st_{\downarrow}(p_{i}), \]
where the equivalence is a simplicial isomorphism when the right side is given the simplicial product structure (Definition \ref{definition:product}). 
\end{proposition}

\begin{proof} 
We will show that the two sides of the equivalence are abstractly isomorphic.

Let $\sigma$ denote an arbitrary simplex in $st_{\downarrow}(v_{\mathcal{P}})$. Thus, $\sigma$ is an $\mathcal{E}$-chain
\[ v_{0} < v_{1} < v_{2} < \ldots < v_{j} = v, \]
for some $j \geq 0$, and each $v_{k}$ ($k=0, \ldots, j$) is such that, for each $[f,D] \in v_{k}$, there is some $i \in \{1, \ldots, \ell \}$ with the property that $f(D) \subseteq im(p_{i})$.

For each pair $(k,i) \in \{ 0, \ldots, j \} \times \{ 1, \ldots, \ell \}$, we let
\[ v_{ki} = \{ [f,D] \mid [f,D] \in v_{k} \text{ and } f(D) \subseteq im(p_{i}) \}. \]
It follows easily that $v_{ki}$ is the (unique) subcollection of $v_{k}$ whose image is precisely $im(p_{i})$. Proposition \ref{proposition:orderlocal}(2) now implies that, for $i = 1, \ldots, \ell$, we have a weakly increasing sequence of pseudovertices
\[ v_{0i} \leq v_{1i} \leq \ldots \leq v_{ji} = p_{i}. \]
We note that the latter sequence is an $\mathcal{E}$-chain (and thus a simplex) in $st_{\downarrow}(p_{i})$,  after removing any repetitions.

We define a map 
\[ (v_{0} < \ldots < v_{j} = v) \mapsto ((v_{01}, \ldots, v_{0\ell}) < (v_{11}, \ldots, v_{1\ell}) < \ldots  < (p_{1}, \ldots, p_{\ell})). \]
The right side of the above formula is a simplex in the product $\prod_{i=1}^{m} st_{\downarrow}(p_{i})$ by Definitions \ref{definition:linkandstar} and \ref{definition:product}. The proof is completed by noticing that a given $v_{k}$ uniquely determines, and, conversely, is determined by, its $\ell$-tuple $(v_{k1}, \ldots, v_{k\ell})$.
\end{proof}

\begin{corollary} \label{corollary:partitionedlink} (join structure of the partitioned descending link)
Let $v \in \Delta$ and $\mathcal{P}$ be as in Proposition \ref{proposition:partitionedstar}. We have a homeomorphism
\[ lk_{\downarrow}(v_{\mathcal{P}}) \cong \bigast_{j=1}^{\ell} lk_{\downarrow}(p_{j}). \]
\end{corollary}

\begin{proof}
This follows directly from Proposition \ref{proposition:partitionedstar} and Theorem \ref{theorem:product}(2).
\end{proof}

\begin{proposition} \label{proposition:intersectionsoflinksandstars}
(intersections of partitioned descending links and stars)
Let $v$ be a pseudovertex and let $\mathcal{P}_{1}, \ldots, \mathcal{P}_{k}$ be partitions of $v$. We have the equalities
\begin{enumerate}
\item $\displaystyle \cap_{j=1}^{k} st_{\downarrow}(v_{\mathcal{P}_{j}},\Delta) =
st_{\downarrow}(v_{\wedge_{j=1}^{k} \mathcal{P}_{j}},\Delta)$.
\item $\displaystyle \cap_{j=1}^{k} lk_{\downarrow}(v_{\mathcal{P}_{j}},\Delta) = lk_{\downarrow}(v_{\wedge_{j=1}^{k} \mathcal{P}_{j}},\Delta)$.
\end{enumerate}
\end{proposition}

\begin{proof}
This is a simple consequence of Definition \ref{definition:downwardpartitionedcomplex} and the definition of meet (Definition \ref{definition:meetrestrict}).
\end{proof}

\subsection{The descending link: the standard cover} \label{subsection:analysisII}
We are now almost ready to state a useful inductive principle for proving that the descending link is highly connected.
The main result here is 
Proposition \ref{proposition:downwardlinkconnectivity}, stating easy sufficient conditions for the downward link to be highly connected, in terms of ``contracting pseudovertices''.  

Recall that $\Delta$ denotes any of the complexes $\Delta^{\mathcal{E}}(\mathcal{V}_{\mathbb{S}})$ or $\Delta^{\mathcal{E}}(\mathcal{PV}_{\mathbb{S},Y})$, for $Y$ an arbitrary finite disjoint union of domains.

\begin{definition} \label{definition:contractingpseudovertex}  (contracting pseudovertex)
A pseudovertex $p$ is called \emph{contracting} if there is some $b \in \mathcal{B}$ and some pseudovertex $p' \in \mathcal{E}(b) - \{ b \}$ such that $p$ and $p'$ have the same type.
\end{definition}

\begin{lemma} \label{lemma:contractingpseudovertex} (contracting pseudovertices admit contractions) If $p$ is a contracting pseudovertex, then there is some $[f,D] \in \mathcal{B}$ such that $p \in \mathcal{E}([f,D])$.
\end{lemma}

\begin{proof} 
Since $p$ is a contracting pseudovertex, we can find some $[g,E] \in \mathcal{B}$ and some pseudovertex $p' \in \mathcal{E}([g,E])$ such that $p$ and $p'$ have the same type. By Proposition \ref{proposition:vertexorbits}, there exists some $\hat{s} \in \widehat{S}$ such that $\hat{s} \cdot p' = p$. It follows from $\widehat{S}$ invariance of $\mathcal{E}$ (Definition \ref{definition:scheme}) that 
$p \in \mathcal{E}([\hat{s}g,E])$.
\end{proof}

\begin{definition} \label{definition:standardcover} (standard cover)
If $v$ is a pseudovertex in $\Delta$ and $p \subseteq v$ is a contracting pseudovertex, then we let $\mathcal{P}_{p} = \{ p, v-p \}$. We let 
\[ \mathcal{C}_{v} = 
\{ lk_{\downarrow}(v_{\mathcal{P}_{p}}, \Delta) \mid p \subseteq v \text{ is a contracting pseudovertex} \}. \]
We call $\mathcal{C}_{v}$ the \emph{standard cover of $lk_{\downarrow}(v,\Delta)$}.
\end{definition}

\begin{proposition} \label{proposition:standardcoverisacover}
Let $v \in \Delta$ be a pseudovertex. The collection $\mathcal{C}_{v}$ is a cover of $lk_{\downarrow}(v,\Delta)$.
\end{proposition}

\begin{proof}
Let 
\[ v_{0} < v_{1} < v_{2} < \ldots < v_{n-1} \]
be a simplex in the downward link $lk_{\downarrow}(v,\Delta)$, where $n \geq 1$. We must show that this simplex lies in $lk_{\downarrow}(v_{\mathcal{P}_{p}}, \Delta)$, for some contracting pseudovertex $p \subseteq v$.

Note that
\[c:= (v_{0} < v_{1} < v_{2} < \ldots < v_{n-1} < v_{n}) \]
is a simplex in the downward star $st_{\downarrow}(v,\Delta)$.
Let us write $v_{0} = \{ b_{1}, \ldots, b_{m}\}$, where each $b_{i} \in \mathcal{B}$. Since $c$ is an $\mathcal{E}$-chain, each $v_{i}$ is the result of $\mathcal{E}$-expansion from $v_{0}$ (see Definition \ref{definition:scheme}). It follows that
\[ v_{k} = \bigcup_{j=1}^{m} v_{k,j}, \]
for $k=0, \ldots, n$, where $v_{k,j} \in \mathcal{E}(b_{j})$. Proposition \ref{proposition:orderlocal} implies that, for all $j$, $v_{k_{1},j} \leq v_{k_{2},j}$ when $k_{1} \leq k_{2}$.

There is some $j$ such that $v_{n,j} \gneq v_{0,j}$, since $v_{n} \gneq v_{0}$. We assume, without loss of generality, that $j=1$, and set $p = v_{n,1}$. (We note that $p \in \mathcal{E}(b_{1}) - \{ b_{1} \}$, so $p$ is a contracting pseudovertex.) It follows directly that, for $k<n$, $v_{k} \in lk_{\downarrow}(v_{\mathcal{P}_{p}},\Delta)$ 
(since $v_{k,1} \leq p$ and $\cup_{j=2}^{m} v_{k,j} \leq v-p$). Thus,
\[ v_{0} < v_{1} < \ldots < v_{n-1} \]
is a simplex in $lk_{\downarrow}(v_{\mathcal{P}_{p}},\Delta)$, completing the proof.  
\end{proof}

\begin{proposition} \label{proposition:nerveofstandardcover} (connectedness of the nerve of the standard cover) Let $v \in \Delta$ be a pseudovertex and let $\mathcal{C}_{v}$ be the standard cover of $lk_{\downarrow}(v,\Delta)$.  If for every 
choice of contracting pseudovertices $p_{0}, \ldots, p_{k} \subseteq v$, there is some contracting pseudovertex $p' \subseteq v-(p_{0} \cup \ldots \cup p_{k})$, then the nerve
$\mathcal{N}(C_{v})$ is $(k-1)$-connected.
\end{proposition}

\begin{proof}
We claim that every $(k+1)$-element subcollection of $\mathcal{C}_{v}$ spans a simplex in $\mathcal{N}(\mathcal{C}_{v})$; i.e., that every such subcollection has non-empty intersection. 

Let $p_{0}, \ldots, p_{k} \subseteq v$ be an arbitrary choice of contracting pseudovertices. By Proposition \ref{proposition:intersectionsoflinksandstars}
\[ \cap_{j=0}^{k} lk_{\downarrow}(v_{\mathcal{P}_{j}}, \Delta) = lk_{\downarrow}(v_{\wedge_{j=0}^{k} \mathcal{P}_{j}}, \Delta). \]
We note that
\[ v-(p_{0} \cup \ldots \cup p_{k}) \in \wedge_{j=0}^{k} \mathcal{P}_{j}. \]
Since there is a contracting pseudovertex $p' \subseteq v-(p_{0} \cup \ldots \cup p_{k})$, we have
\[ lk_{\downarrow}( v_{\wedge_{j=0}^{k} \mathcal{P}_{j}}, \Delta) \neq \emptyset. \]
This proves the claim.  

It now follows that $\mathcal{N}(\mathcal{C}_{v})$ consists of the entire $k$-skeleton of a high-dimensional (or infinite-dimensional) simplex. (If $\mathcal{N}(\mathcal{C}_{v})$ has $k+1$ or fewer vertices, then it is a simplex.) It follows that $\mathcal{N}(\mathcal{C}_{v})$ is $(k-1)$-connected.   
\end{proof}

\begin{proposition} ($n$-connectivity of the descending link: a sufficient condition) \label{proposition:downwardlinkconnectivity}  Let $v \in \Delta$ be a pseudovertex and let $n \in \{ 0 \} \cup \mathbb{N}$. Assume that the descending link of $v$ is non-empty.
If
\begin{enumerate}
\item for every contracting pseudovertex $p$ contained in $v$,
\[ lk_{\downarrow}(v-p, \Delta) \]
is $(n-1)$-connected, and
\item for every $j \in \{ 2, \ldots, n+2 \}$ and for every choice $p_{1}, \ldots, p_{j}$ of contracting pseudovertices contained in $v$,  
\[ lk_{\downarrow}(v-(p_{1} \cup \ldots \cup p_{j}),\Delta) \] is $(n+1-j)$-connected,
\end{enumerate}
 then $lk_{\downarrow}(v,\Delta)$ is $n$-connected.
\end{proposition}

\begin{proof}
We will apply the Nerve Theorem (Theorem \ref{theorem:Nerve}) to the nerve $\mathcal{N}(\mathcal{C}_{v})$.

We will first show that $\mathcal{N}(\mathcal{C}_{v})$ is $n$-connected. Let $j = n+2$. For every choice $p_{1}, \ldots, p_{n+2} \subseteq v$ 
of contracting pseudovertices, we have
\[ lk_{\downarrow}(v-(p_{1} \cup \ldots \cup p_{n+2}), \Delta) \neq \emptyset \]
by hypothesis.
Since $v-(p_{1} \cup \ldots \cup p_{n+2}) \in \wedge_{\ell=1}^{n+2} \mathcal{P}_{\ell}$,
\[ lk_{\downarrow}(v_{\wedge_{\ell =1}^{n+2} \mathcal{P}_{\ell}}, \Delta) 
= \cap_{\ell =1}^{n+2} lk_{\downarrow}(v_{\mathcal{P}_{\ell}},\Delta), \]
is also non-empty. It now follows from Proposition \ref{proposition:nerveofstandardcover} that $\mathcal{N}(\mathcal{C}_{v})$ is $n$-connected.

Next, consider any $j \in \{ 2, \ldots, n+1 \}$; let $p_{1}, \ldots, p_{j} \subseteq v$ be contracting pseudovertices. We have that
\[ lk_{\downarrow}(v-(p_{1} \cup \ldots \cup p_{j})) \]
is $(n+1-j)$-connected by hypothesis. It follows that
\[ \cap_{\ell=1}^{j} lk_{\downarrow}(v_{\mathcal{P}_{\ell}},\Delta) \]
is $(n+1-j)$-connected, by essentially the same reasoning as above.

Now consider the case $j=1$. Let $p_{1} \subseteq v$ be a contracting pseudovertex. We have
\[ lk_{\downarrow}(v_{\mathcal{P}_{p_{1}}}) \cong lk_{\downarrow}(p_{1}) \ast
lk_{\downarrow}(v-p_{1}). \]
Note that $lk_{\downarrow}(p_{1})$ is non-empty since $p_{1}$ is a contracting pseudovertex, and $lk_{\downarrow}(v-p_{1})$ is $(n-1)$-connected by hypothesis. It follows that $lk_{\downarrow}(v_{\mathcal{P}_{p_{1}}})$ is $n$-connected. 

It follows from Theorem \ref{theorem:Nerve} that $lk_{\downarrow}(v,\Delta)$ is $n$-connected.
\end{proof}

\section{Finiteness properties of groups} \label{section:8}

In this section, we will establish finiteness properties for a number of the groups. In Subsection \ref{subsection:Finfbasic}, we prove Theorem \ref{theorem:easyFinfinitygroups}, which states that the group $\Gamma_{S}$ has type $F_{\infty}$ when the associated expansion scheme $\mathcal{E}$ is ``rich in contractions'' (Definition \ref{definition:rich}). This result greatly generalizes the main theorem of \cite{FarleyHughes}. In Subsection \ref{subsection:inductive}, we establish an inductive procedure that enables one to prove that a given group has type $F_{n}$. The latter procedure can be applied in various cases where the ``rich in contractions'' property fails. 

\subsection{A basic sufficient condition for type $F_{\infty}$} \label{subsection:Finfbasic}

We can now offer a simple sufficient condition for the group $\Gamma_{S}$ to have type $F_{\infty}$ (Theorem \ref{theorem:easyFinfinitygroups}). The most important new ingredient is the ``rich in contractions'' property (Definition \ref{definition:rich}), a generalization of the ``rich in simple contractions'' property (Definition 5.11) from 
\cite{FarleyHughes}. 

\begin{definition} \label{definition:rich} (Rich in contractions) \label{definition:rich}
Let $\mathcal{E}$ be an expansion scheme. We say that $\mathcal{E}$ is \emph{rich in contractions} if there is some constant $C_{1}$ such that, if $v \in \mathcal{PV}_{\mathbb{S}}$ is any pseudovertex of rank at least $C_{1}$, then there is some contracting pseudovertex $v' \subseteq v$.
\end{definition}

\begin{theorem} \label{theorem:easyFinfinitygroups}(Groups of type $F_{\infty}$) Let $\mathbb{S}$ be an $S$-structure with finitely many domain types, such that the group $\mathbb{S}(D,D)$ has type $F_{\infty}$ for each $D \in \mathcal{D}^{+}$. 
Let $\mathcal{E}$ be an expansion scheme such that 
\begin{enumerate}
\item $\mathcal{E}$ is $n$-connected for all $n$;
\item $\mathcal{E}$ is rich in contractions;
\item each set $\mathcal{E}(b)$ ($b \in \mathcal{B}$) is finite.
\end{enumerate}
The group $\Gamma_{S}$ has type $F_{\infty}$.
\end{theorem}

\begin{proof}
We will show that the hypotheses of Theorem \ref{theorem:Brown} are satisfied for all $n$.

We first note that, since $\mathcal{E}$ is $n$-connected for all $n$, $\Delta^{\mathcal{E}}$ is $n$-connected for all $n$ by Theorem \ref{theorem:bigone}, and thus contractible by Whitehead's Theorem. The action of $\Gamma_{S}$ on $\Delta^{\mathcal{E}}$ is clearly cellular. 

By Proposition \ref{proposition:findiminv}, we have the equality
\[ \Delta^{\mathcal{E}} = \bigcup_{k=1}^{\infty} \Delta^{\mathcal{E}}_{n}, \]
where each subcomplex $\Delta^{\mathcal{E}}_{k}$ is $\Gamma_{S}$-invariant.
Moreover, the finiteness of the sets $\mathcal{E}(b)$ easily implies that $\mathcal{E}$ is $\mathbb{S}$-finite. Since $\mathbb{S}$ also has finitely many domain types, Proposition \ref{proposition:cocompact} implies that the action of $\Gamma_{S}$ on each $\Delta^{\mathcal{E}}_{k}$ is cocompact.

Since each set $\mathcal{E}(b)$ is finite and each group $\mathbb{S}(D,D)$ has type $F_{\infty}$, both hypotheses of Corollary \ref{corollary:Fnstabilizers} are satisfied for each $n$. Thus, each cell stabilizer in $\Delta^{\mathcal{E}}$ is of type $F_{\infty}$.

It is now enough to show that, for each $n$, $\Delta^{\mathcal{E}}_{k}$ is $n$-connected for sufficiently large $k$. By Proposition \ref{proposition:connectfilt}, it suffices to show that the descending links $lk_{\downarrow}(v, \Delta^{\mathcal{E}})$ are always $n$-connected, provided that the rank of the vertex $v$ is sufficiently large. We will in fact show this for all pseudovertices $v$ of sufficiently large rank.

Let $C_{1}$ be the constant from Definition \ref{definition:rich}. Any pseudovertex $v$ of rank at least $C_{1}$ thus contains a contracting pseudovertex. It follows from Lemma \ref{lemma:contractingpseudovertex} that $lk_{\downarrow}(v, \Delta^{\mathcal{E}})$ is non-empty if $r(v) \geq C_{1}$.
We note also that the standard cover $\mathcal{C}_{v}$ is a cover of $lk_{\downarrow}(v,\Delta^{\mathcal{E}})$ if $r(v) \geq C_{1}$, by Proposition \ref{proposition:standardcoverisacover}. 

Since there are only finitely many domain types and each set 
$\mathcal{E}(b)$ is finite (for $b \in \mathcal{B}$), there is a constant $C_{0}$ such that the rank of each contracting pseudovertex is less than or equal to $C_{0}$. 

Let $n \geq 0$ and let $v$ be a pseudovertex. We claim that if $r(v) \geq (2n+2)C_{0} + C_{1}$, then $lk_{\downarrow}(v, \Delta^{\mathcal{E}})$ is $n$-connected.
The proof is by induction on $n$, beginning with the case $n=0$. We will use the sufficient condition for $n$-connectivity given in Proposition \ref{proposition:downwardlinkconnectivity}. Thus, assume that $r(v) \geq 2C_{0} + C_{1}$. If $p \subseteq v$ is any contracting pseudovertex, we have $r(v-p) \geq C_{0} + C_{1}$, and therefore $lk_{\downarrow}(v-p,\Delta^{\mathcal{E}})$ is non-empty. This establishes condition (1) from Proposition \ref{proposition:downwardlinkconnectivity}. Now let $p_{1}, p_{2} \subseteq v$ be contracting pseudovertices.
Clearly, 
\[ r(v - (p_{1} \cup p_{2})) \geq C_{1},\] 
so $lk_{\downarrow}(v - (p_{1} \cup p_{2}))$ is non-empty, establishing (2). This proves the claim when $n=0$.

Now let $n$ be arbitrary, and assume that the claim holds for smaller $n$. Let $v$ be a pseudovertex of rank at least $(2n+2)C_{0} + C_{1}$. We check condition (1) from Proposition \ref{proposition:downwardlinkconnectivity}; thus, let $p \subseteq v$ be a contracting pseudovertex. Clearly, 
\[ r(v-p) \geq (2n+1)C_{0} + C_{1} \geq (2n)C_{0} + C_{1}, \] so $lk_{\downarrow}(v-p, \Delta^{\mathcal{E}})$ is $(n-1)$-connected, as required. Now we check (2); let $p_{1}, \ldots, p_{j} \subseteq v$ be contracting pseudovertices, for some $j \in \{ 2, \ldots, n+2 \}$. Clearly,
\begin{align*}
 r(v - (p_{1} \cup \ldots \cup p_{j})) &\geq (2n+2-j)C_{0} + C_{1} \\
 &\geq [2(n+1-j) + 2]C_{0} + C_{1}, 
 \end{align*}
 which shows that 
 \[ lk_{\downarrow}(v-(p_{1} \cup \ldots \cup p_{j}), \Delta^{\mathcal{E}}) \]
 is $(n+1-j)$-connected, as required. 
 
 This proves the claim and completes the proof of the theorem.
 \end{proof}
 
 \begin{example} \label{example:VhasFinfinity} (The $F_{\infty}$ property for the generalized Thompson groups $V_{n,r}$)
We first consider Thompson's group $V$. Recall that the associated set $\mathcal{D}_{S_{V}}^{+}$ of domains satisfies the compact ultrametric property E (Example \ref{example:V}). We use the maximal $S_{V}$-structure (Example \ref{example:maximalSstructure}). The resulting structure function $\mathbb{S}$ assigns a singleton to each pair of domains:
\[ \mathbb{S}(B_{\omega_{1}}, B_{\omega_{2}}) = \{ \sigma_{\omega_{1},\omega_{2}} \}. \]
(In particular, we note that there is only one domain type.)
It follows directly that every simplex in $\Delta(\mathcal{V}_{\mathbb{S}})$ has a finite stabilizer (Proposition \ref{proposition:stab}).
We use the expansion scheme $\mathcal{E}$ from Proposition \ref{proposition:expansionforcompactumetric}; the associated complex $\Delta^{\mathcal{E}}$ is $m$-connected for all $m$. For any domain $B_{\omega} \in \mathcal{D}_{S_{V}}^{+}$, the maximal partition of $B_{\omega}$ is as follows:
\[ \mathcal{P}_{B_{\omega}} = \{ B_{\omega0}, B_{\omega1} \}. \]
It follows easily from the description of $\mathcal{E}$ that any pseudovertex of rank two is a contracting pseudovertex. Thus, the expansion scheme $\mathcal{E}$ is rich in contractions with constant $C_{1}=2$.  Clearly, the sets $\mathcal{E}(b)$ are also finite, so Theorem \ref{theorem:easyFinfinitygroups} implies that $V$ has type $F_{\infty}$.

More generally, we can consider the group $V_{n}$ that acts on the $n$-ary Cantor set
\[ \mathcal{C}_{n} = \prod_{k=1}^{\infty} \{0, 1, \ldots, n-1 \}. \]
We analogously define transformations $\sigma_{\omega_{1},\omega_{2}}$, where $\omega_{1}$ and $\omega_{2}$ are finite strings over the alphabet $\{ 0, \ldots, n-1 \}$. As in the case of $V$, a transformation $\sigma_{\omega_{1},\omega_{2}}$ removes the prefix $\omega_{1}$ from an infinite $n$-ary string, and attaches the prefix $\omega_{2}$ to the resulting string. (If the string $a_{1}a_{2}\ldots$ does not begin with the prefix $\omega_{1}$, then $\sigma_{\omega_{1},\omega_{2}}(a_{1}a_{2}\ldots)$ is undefined.) Letting 
\[ S_{V_{n}} = \{ \sigma_{\omega_{1},\omega_{2}} \mid \omega_{1}, \omega_{2} \text{ are finite }n\text{-ary strings} \} \cup \{ 0 \}, \]
we find that 
\[ \mathcal{D}^{+}_{S_{V_{n}}} = \{ B_{\omega} \mid \omega \text{ is a finite }n\text{-ary string} \}, \]
where $B_{\omega}$ is the collection of all infinite $n$-ary strings that begin with the prefix $\omega$. For any $\omega$,
\[ \mathcal{P}_{B_{\omega}} = \{ B_{\omega0}, \ldots, B_{\omega n-1} \}. \]
If we use the expansion scheme $\mathcal{E}$ from 
Proposition \ref{proposition:expansionforcompactumetric}, then every pseudovertex of rank $n$ is a contracting pseudovertex. Thus, the expansion scheme $\mathcal{E}$ is rich in contractions with constant $C_{1} = n$. The remaining conditions from Theorem \ref{theorem:easyFinfinitygroups} are easily checked, so we conclude that
$V_{n}$ has type $F_{\infty}$ for all $n \geq 2$.

The group $V_{n,r}$ ($n \geq 2$; $r \geq 1$) acts on the set
\[ X = \mathcal{C}_{n} \coprod \mathcal{C}_{n} \coprod \ldots \coprod \mathcal{C}_{n}, \]
where there are $r$ terms in the disjoint union. Elements of $V_{n,r}$ are locally determined by transformations between sets of the form $B_{\omega}$. (The domains and images of these transformations can be inside different copies of $\mathcal{C}_{n}$.) Essentially the same line of argument as those given above shows that $V_{n,r}$ has type $F_{\infty}$.   
\end{example}
 
 \begin{example} \label{example:nVFinfinity} (The $F_{\infty}$ property for the Brin-Thompson groups $nV$ and R\"{o}ver's group) Next we will  show that the groups $nV$ and R\"{o}ver's group all have type $F_{\infty}$. We will use the expansion schemes described in Propositions \ref{proposition:expansionschemeproduct} and  \ref{proposition:Roeverexpansion}, respectively, as well as all previously-established conventions related to these groups. Recall that all of these expansion schemes are $m$-connected, for all $m$.
 
 Note that all of the $S$-structures in question have only a single domain type. The expansion schemes are alike in that any two-element subset 
 $\{ [f_{1},D_{1}], [f_{2},D_{2}] \}$ of any pseudovertex is a contracting pseudovertex. It follows that the expansion schemes in question are rich in contractions with constant $C_{1} = 2$ (see Definition \ref{definition:rich}). 
It is clear from the descriptions of $\mathcal{E}$ that each set $\mathcal{E}(b)$ is finite. Finally, we note that the group $\mathbb{S}(D,D)$ is either trivial  (in the case of $nV$) or of order four (in the case of R\"{o}ver's group). It follows from Theorem \ref{theorem:easyFinfinitygroups} that $nV$ and the R\"{o}ver group have type $F_{\infty}$. Note that this proof also covers Thompson's group $V = 1V$.
\end{example}

\begin{example} \label{example:more}
(more examples based on products)
For $n \geq 1$, let $\bar{T}_{n}$ denote the rooted ordered infinite $n$-ary tree. 
Thus, $\bar{T}_{1}$ is a cellulated ray, $\bar{T}_{2}$ is the rooted ordered infinite binary tree, and so forth. We let $S_{\bar{n}}$ denote the inverse semigroup generated by two types of partial transformations of $\bar{T}_{n}$:
\begin{enumerate}
\item singleton transformations, whose domains and images are both singleton sets, and
\item transformations between subtrees, which move one rooted subtree to another, without otherwise permuting leaves or branches.
\end{enumerate}
Thus, $S_{\bar{1}}$ is the same as the inverse semigroup $S_{H_{1}}$ (Example \ref{example:Houghton}), and 
$S_{\bar{2}} = S_{QV}$ (Example \ref{example:QV}). If $n \geq 3$, then 
$S_{\bar{n}}$ is the straightforward $n$-ary generalization of $S_{QV}$. 

For $n \geq 2$, we will let $S_{n}$ denote the inverse semigroup associated to the generalized Thompson groups $V_{n}$. Thus, $S_{2} = S_{V}$ and $S_{n} = S_{V_{n}}$ for $n \geq 3$. 

We will consider product actions of the above semigroups. Let
\[ \Sigma = (\bar{a}_{1}, \bar{a}_{2},\ldots, \bar{a}_{j}, a_{j+1}, \ldots, a_{k}),\] where the $a_{\ell}$ are all positive integers, $a_{\ell} \geq 1$ when $1 \leq \ell \leq j$, and $a_{\ell} \geq 2$ when $j+1 \leq \ell \leq k$. We assume, furthermore, that
$a_{1} \leq \ldots \leq a_{j}$ and $a_{j+1} \leq \ldots \leq a_{k}$. We define
\[ S_{\Sigma}
:= S_{\bar{a}_{1}} \times S_{\bar{a}_{2}} \times
\ldots \times S_{\bar{a}_{j}} \times S_{a_{j+1}} \times \ldots \times S_{a_{k}}. \]

We claim that the group $\Gamma_{S_{\Sigma}}$ has type $F_{\infty}$ if $j<k$ (i.e., if there is at least one integer without a bar). The proof is very similar to the proof that $nV$ has type $F_{\infty}$, and we in fact choose the $S_{\Sigma}$-structure and expansion scheme $\mathcal{E}$ exactly as in that case. We mention the main difference: the structure function $\mathbb{S}$ associated to $\Gamma_{S_{\Sigma}}$ has $2^{j}$ distinct domain types. (This is because each of the factors has either two or one domain types, according to whether the subscript has a bar or not (respectively).) If a pseudovertex $v$ has rank at least $(a_{j+1} - 1) 2^{j} + 1$, then there will necessarily be at least $a_{j+1}$ pairs $[f,D]$ having the same type. It follows from the definition of $\mathcal{E}$ that it will then be possible to perform a contraction on $v$. Thus, the expansion scheme $\mathcal{E}$ is rich in contractions with constant $(a_{j+1} - 1)2^{j} + 1$. The remaining hypotheses of Theorem \ref{theorem:easyFinfinitygroups} are straightforward to check, completing the proof.    
 \end{example}

\begin{example} \label{example:FSS} (FSS groups) 
Let $X$ be a compact ultrametric space with ultrametric $d$. A \emph{finite similarity structure} \cite{FarleyHughes} associates to each pair of balls $(B_{1},B_{2})$ a finite set $\mathrm{Sim}_{X}(B_{1},B_{2})$ of surjective similarities; i.e. bijections $h: B_{1} \rightarrow B_{2}$ that stretch distances by a constant factor $\lambda$ that depends only upon $h$. The sets $\mathrm{Sim}_{X}(B_{1},B_{2})$ are required to satisfy properties (S2)-(S4) from Definition \ref{definition:sstructure} and to be closed under restrictions (in the sense of Remark
\ref{remark:closedrestrictions}). Let $\Gamma_{\mathrm{Sim}_{X}}$ denote the set of bijections of $X$ that are locally determined by the sets $\mathrm{Sim}_{X}(B_{1},B_{2})$. The authors showed that if $\mathrm{Sim}_{X}$ is rich in simple contractions and has finitely many ball types, then $\Gamma_{\mathrm{Sim}_{X}}$ has type $F_{\infty}$ \cite{FarleyHughes}. 

If we let 
\[ S = \bigcup_{(B_{1},B_{2})} \mathrm{Sim}_{X}(B_{1},B_{2}) \]
then $S$ is an inverse semigroup acting on $X$. The set of domains $\mathcal{D}_{S}^{+}$ is precisely the set of all metric balls in $X$. If we use the maximal $S$-structure, then we have the identity $\mathbb{S}(B_{1},B_{2}) = \mathrm{Sim}_{X}(B_{1},B_{2})$, for all pairs of metric balls $(B_{1},B_{2})$. We let the expansion scheme $\mathcal{E}$ be defined as in Proposition \ref{proposition:expansionforcompactumetric}. If we assume that $\mathcal{E}$ is rich in contractions and $\mathbb{S}$ has finitely many domain types, then Theorem \ref{theorem:easyFinfinitygroups} shows that $\Gamma_{S}$ has type $F_{\infty}$. This recovers the main result from \cite{FarleyHughes} (as described above).  
\end{example}  
 
\subsection{An inductively-defined sufficient condition for type $F_{n}$} \label{subsection:inductive}

In some cases, the ``rich in contractions" condition (Definition \ref{definition:rich}) is too restrictive. In this subsection, we will compute the connectivity of the descending link by inductive means. The induction will be done over the collection of ``type vectors".  

\begin{definition} \label{definition:typevectors} (type vectors; contracting vectors)
Assume that there are only finitely many domain types relative to  $\mathbb{S}$ (Definition \ref{definition:domaintypes}). If there are $t$ different domain types in all, then choose a numbering $1, \ldots, t$ of these domain types.

Let $p = \{ [f_{1}, D_{1}], \ldots, [f_{m}, D_{m}] \}$ be a pseudovertex. The \emph{type vector of $p$}, denoted $\vec{w}_{p}$,  
is the vector 
\[ (a_{1}, \ldots, a_{t}) \in (\mathbb{N} \cup \{ 0 \})^{t}, \]
where $a_{i}$ is the number of subscripts $j \in \{ 1, \ldots, m \}$ such that $D_{j}$ has type $i$. (I.e.,  $\vec{w}_{p}$ counts the number of domains $D_{j}$ having each of the $t$ domain types.) 
Conversely, 
we say that $v$ is \emph{of type $\vec{w}_{v}$}.

A vector $\vec{w} \in (\mathbb{N} \cup \{ 0 \})^{t}$ is called a \emph{contracting vector} if it is the type vector of some contractive pseudovertex. If $p \in \mathcal{E}(b)- \{ b \}$ is a contracting pseudovertex and $\vec{w} = \vec{w}_{p}$, then we write $\vec{w} \rightarrow \vec{w}_{\{ b \}}$. 
\end{definition}

\begin{remark} \label{remark:stableconn}
Clearly, $\ell_{sc}(\vec{w}_{1}) \geq \ell_{sc}(\vec{w}_{2})$
when $\vec{w}_{1} \succcurlyeq \vec{w}_{2}$. We will use this fact without further comment in what follows.
\end{remark}

\begin{definition} (connectivity length; stable connectivity length)
Let $\vec{w}$ be a type vector. The \emph{connectivity length of $\vec{w}$}, denoted $\ell_{c}(\vec{w})$, is the largest $n$ such that $lk_{\downarrow}(p)$
is $n$-connected, for some (equivalently, any) pseudovertex $p$ having type vector $\vec{w}$. 

If $\vec{w}_{1} = (a_{1}, \ldots, a_{t})$ and $\vec{w}_{2} = (b_{1}, \ldots, b_{t})$ are type vectors, then we write 
$\vec{w}_{1} \preccurlyeq \vec{w}_{2}$ if $a_{i} \leq b_{i}$ for all $i \in \{ 1, \ldots, t \}$. 

A type vector $\vec{w}$ has \emph{stable connectivity length at least $n$}
if, for every type vector $\vec{w}_{1} \succcurlyeq \vec{w}$, $\ell_{c}(\vec{w}_{1}) \geq n$. In this case, we write $\ell_{sc}(\vec{w}) \geq n$. We say that $\vec{w}$ has \emph{stable connectivity length $n$} if
\[ n = \mathrm{max}\{ j  \mid \ell_{sc}(\vec{w}) \geq j \}. \]
\end{definition}
 
\begin{proposition} \label{proposition:Finfind} (A sufficient condition for type $F_{n}$) Let $\mathbb{S}$ be an $S$-structure with finitely many domain types, such that the group $\mathbb{S}(D,D)$ has type $F_{n}$ for each $D \in \mathcal{D}^{+}$. 
Let $\mathcal{E}$ be an $(n-1)$-connected expansion scheme such that 
\begin{enumerate}
\item each set $\mathcal{E}(b)$ ($b \in \mathcal{B}$) is finite, and
\item there is a constant $C$ such that, whenever $v$ is a vertex satisfying $r(v) \geq C$, $\ell_{sc}(\vec{w}(v)) \geq n-1$.
\end{enumerate}
The group $\Gamma_{S}$ has type $F_{n}$.
\end{proposition}

\begin{proof}
We note first that $\Delta^{\mathcal{E}}$ is $(n-1)$-connected by Theorem \ref{theorem:bigone}. The action of $\Gamma_{S}$ on $\Delta^{\mathcal{E}}$ is cellular by Theorem \ref{theorem:Ecomplex}. The complex $\Delta^{\mathcal{E}}$ is filtered by the $\Gamma_{S}$-complexes $\Delta^{\mathcal{E}}_{k}$ (see Definition \ref{definition:thefiltration} and Proposition \ref{proposition:findiminv}). The action of $\Gamma_{S}$ on each $\Delta^{\mathcal{E}}_{k}$ is cocompact by Proposition \ref{proposition:cocompact}. Each cell stabilizer has type $F_{n}$, by Corollary \ref{corollary:Fnstabilizers}. 

It therefore suffices to show that each subcomplex $\Delta^{\mathcal{E}}_{k}$ is $(n-1)$-connected, for $k$ sufficiently large.  For this, it is sufficient, by Proposition \ref{proposition:connectfilt}, to prove, for some $k$, that the descending links $lk_{\downarrow}(v,\Delta^{\mathcal{E}})$ of all vertices of rank at least $k$ are $(n-1)$-connected. The latter follows immediately from (2) by letting $k=C$. 
\end{proof}

\begin{proposition} \label{proposition:inductiveconnectivity} (inductively computing $\ell_{sc}(\vec{w})$)
Let $\vec{w} \in (\mathbb{N} \cup \{ 0 \})^{t}$ be a type vector and let $n \in \mathbb{N} \cup \{ 0 \}$. If
\begin{enumerate}
\item for each contracting vector $\vec{c}_{1}$ such that $\vec{c}_{1} \preccurlyeq \vec{w}$, 
\[ \ell_{sc}(\vec{w} - \vec{c}_{1}) \geq n-1,\]
and
\item for each $j \in \{ 2, \ldots, n+2 \}$ and for every choice of (not necessarily distinct) contracting vectors $\vec{c}_{1}, \ldots, \vec{c}_{j} \preccurlyeq \vec{w}$, 
\[ \ell_{sc}\left(\vec{w} - \sum_{k=1}^{j} \vec{c}_{k}\right) \geq n-j+1, \]
where any negative entries in the above vector are to be interpreted as $0$s,
\end{enumerate}
then $\ell_{sc}(\vec{w}) \geq n$.
\end{proposition}

\begin{proof}
Let $\vec{w}$ satisfy the given conditions; we let $v$ be a pseudovertex such that $\vec{w}_{v} \succcurlyeq \vec{w}$. We apply Proposition \ref{proposition:downwardlinkconnectivity}. 

If $p\subseteq v$ is a contracting pseudovertex, then the type vector of 
$v-p$ is $\vec{w} - \vec{p}$. Our hypothesis says that $\ell_{sc}(\vec{w}-\vec{p}) \geq n-1$, which implies that $lk_{\downarrow}(v-p, \Delta)$ is $(n-1)$-connected. This establishes the first part of the hypothesis from Proposition \ref{proposition:downwardlinkconnectivity}.

Now suppose that $p_{1}, \ldots, p_{j}$ are contracting pseudovertices, each contained in $v$. It follows that
\[ \vec{w}_{v} - \left( \sum_{i=1}^{j} \vec{w}_{p_{i}} \right) \preccurlyeq \vec{w}_{v - (p_{1} \cup \ldots \cup p_{j})}, \]
where any negative entries in the vector on the left may be interpreted as zeroes. Since the stable connectivity length of the vector on the left is at least $n-j+1$ by hypothesis, it follows that $\vec{w}_{v-(p_{1} \cup \ldots \cup p_{j})}$ has connectivity length at least $n-j+1$. Thus,
$lk_{\downarrow}(v-(p_{1} \cup \ldots \cup p_{j}))$ is at least $(n-j+1)$- connected. 

It now follows from Proposition \ref{proposition:downwardlinkconnectivity} that $lk_{\downarrow}(v)$ is $n$-connected. Thus, $\ell_{c}(\vec{w}_{v}) \geq n$ if $\vec{w}_{v} \succcurlyeq \vec{w}$, so
$\ell_{sc}(\vec{w}) \geq n$.
\end{proof}

\begin{example}
\label{example:QVisFinf}
(the group $QV$ is of type $F_{\infty}$) Consider the group $QV$ and the associated semigroup $S_{QV}$ from Example \ref{example:QV}. The set $\mathcal{D}^{+}_{S_{QV}}$ satisfies the compact ultrametric property. We use the maximal $S_{QV}$-structure and the expansion scheme $\mathcal{E}$ from Proposition \ref{proposition:expansionforcompactumetric}.
There are two domain types: singleton sets $\{ \omega \}$ and the sets $\mathcal{T}^{0}_{\omega}$, where $\omega$ is an arbitrary finite binary string. The expansion scheme $\mathcal{E}$ is not rich in contractions, since pseudovertices of the form
\[ \{ [f_{1}, \{ \omega_{1} \}], \ldots, [f_{m}, \{ \omega_{m} \}] \} \]
contain no contracting pseudovertices, and we can clearly let the rank of such pseudovertices become arbitrarily large.

We claim that $QV$ is of type $F_{\infty}$. Since Theorem \ref{theorem:easyFinfinitygroups} does not apply, we proceed inductively and try to apply Proposition \ref{proposition:Finfind}. As noted above, there are just two domain types. Each structure set of the form $\mathbb{S}(D,D)$ contains only the identity transformation, and therefore has type $F_{n}$ for all $n$.   
The expansion scheme $\mathcal{E}$ is $(n-1)$-connected for all $n$ (Proposition \ref{proposition:expansionforcompactumetric}). Each set $\mathcal{E}(b)$ is clearly finite. Thus, it remains only to check condition (2) from Proposition \ref{proposition:Finfind}. 

We order the domain types, letting the singleton sets be first. Thus, a type vector $(a,b)$ describes a pseudovertex
\[ \{ [f_{1}, D_{1}], \ldots, [f_{m},D_{m}] \}, \]
where $a+b=m$ and precisely $a$ of the domains $D_{1}, \ldots, D_{m}$ are singletons. With this convention, there is
only one contracting vector, namely $(1,2)$. (This is because the maximal partition of $\mathcal{T}^{0}_{\omega}$ is
\[ \{ \mathcal{T}^{0}_{\omega0}, \mathcal{T}^{0}_{\omega1}, \{ \omega \} \}, \]
while the maximal partition of $\{ \omega \}$ is
$\{ \{ \omega \} \}$.) 

It follows directly that $\ell_{sc}(\vec{w}) \geq -1$ when
$\vec{w} \succcurlyeq (1,2)$. We claim that, in general, 
$\ell_{sc}(\vec{w}) \geq n$ whenever $\vec{w} \succcurlyeq (2n+3,4n+6)$. Assume the claim is true for $n$; we try to prove that $\ell_{sc}((2n+5,4n+10)) \geq n+1$ by checking the conditions from Proposition \ref{proposition:inductiveconnectivity}.
Note that 
\begin{align*}
\ell_{sc}((2n+5,4n+10) - (1,2)) &= \ell_{sc}((2n+4,4n+9)) \\
&\geq n
\end{align*}
and that, for $j \in \{ 2, \ldots, n+3 \}$,
\begin{align*}
\ell_{sc}((2n+5,4n+10) - (j,2j)) &= \ell_{sc}(2n - j +5, 4n -2j + 10) \\
&\geq n-j+2 
\end{align*}
since $2n-j+5 \geq 2n - 2j +7$ and
$4n-2j+10 \geq 4n -4j + 14$. 
This proves the claim by induction.

A general vertex in the associated complex has the type vector $(n,n+1)$, for some nonnegative integer $n$. It follows easily from the above computation and Proposition \ref{proposition:Finfind} that $QV$ has type $F_{\infty}$, recovering a result from \cite{QV}.
\end{example}

\begin{remark} \label{remark:conclusion!}
(other possible examples) 
Another example that could be considered under this heading is ``$2QV$''; i.e., the group locally determined by $S_{QV} \times S_{QV}$, which would be denoted
$\Gamma_{S_{(\bar{2},\bar{2})}}$ under the conventions of Example \ref{example:more}. Note that the proof of the $F_{\infty}$ property for $\Gamma_{S_{\Sigma}}$ assumes that at least one entry in $\Sigma$ occurs without a bar, so the finiteness properties of $2QV$ are unresolved by Example \ref{example:more}. The combinatorial analysis required for this example seems substantially more difficult than that required for Example \ref{example:QVisFinf}, so we will not undertake it here.

The Houghton group $H_{n}$ (Example \ref{example:Houghton}) was proved by Brown \cite{Brown} to be of type $F_{n-1}$ but not of type $F_{n}$. The inductive principle outlined in Proposition \ref{proposition:inductiveconnectivity} does not quite prove $F_{n-1}$, although it is reasonable to guess that a slight modification would be sufficient. Of course, we have not considered any methods that would allow us to prove that a group $\Gamma_{S}$ does not have type $F_{n}$.

Finally, we note that the groups described by Bieri and Sach (Example \ref{example:bieri}) appear to pose a much more substantial challenge. Some of their finiteness properties are known; we refer the reader to \cite{Bieri} for the current state of knowledge about these groups.
\end{remark}

 \bibliographystyle{plain}
\bibliography{biblio}

\end{document}